\documentclass[11pt,leqno]{amsart}
\usepackage{epsfig}
\usepackage{amssymb}
\usepackage{amscd}
\usepackage{mathrsfs}
\usepackage[all]{xy}

\newtheorem{theo}{Theorem}[section]
\newtheorem{lemma}[theo]{Lemma}
\newtheorem{prop}[theo]{Proposition}

\newtheorem{cor}[theo]{Corollary}

\theoremstyle{definition}

\newtheorem{defi}[theo]{Definition}

\theoremstyle{remark}

\newtheorem{remark}[theo]{Remark}
\newtheorem{example}[theo]{Example}

\def\Z{\mathbb{Z}}

\def\coh{{\operatorname{coh}}}
\def\Qcoh{\operatorname{Qcoh}}

\def\bD{{\mathbf D}}
\def\bR{{\mathbf R}}
\def\bL{{\mathbf L}}

\def\ba{{\mathbf a}}
\def\be{{\mathbf e}}

\def\pre-tr{\operatorname{pre-tr}}
\def\h{\operatorname{h}}
\def\Hom{\operatorname{Hom}}
\def\End{\operatorname{End}}

\newcommand{\bbZ}{{\mathbb Z}}

\newcommand{\cJ}{{\mathcal J}}

\newcommand{\csI}{{\mathscr I}}

\newcommand{\cO}{{\mathcal O}}
\newcommand{\cP}{{\mathcal P}}

\newcommand{\cM}{{\mathcal M}}
\newcommand{\cN}{{\mathcal N}}
\newcommand{\cD}{{\mathscr D}}

\newcommand{\cA}{{\mathcal A}}
\newcommand{\csA}{{\mathscr A}}
\newcommand{\cB}{{\mathcal B}}

\newcommand{\cC}{{\mathcal C}}
\newcommand{\cE}{{\mathcal E}}

\newcommand{\cU}{{\mathcal U}}

\newcommand{\cS}{{\mathcal S}}
\newcommand{\cT}{{\mathscr T}}

\newcommand{\cHom}{{\mathcal Hom}}

\newcommand{\Fun}{\operatorname{Fun}}
\newcommand{\Tot}{\operatorname{Tot}}

\newcommand{\Perf}{\operatorname{Perf}}

\newcommand{\Ker}{\operatorname{Ker}}

\newcommand{\pdim}{\operatorname{pdim}}
\newcommand{\gldim}{\operatorname{gldim}}

\newcommand{\Ext}{\operatorname{Ext}}

\newcommand{\add}{\operatorname{add}}

\newcommand{\Ind}{\operatorname{Ind}}
\newcommand{\LInd}{\operatorname{\mathbf{L}Ind}}
\newcommand{\Res}{\operatorname{Res}}
\newcommand{\Spec}{\operatorname{Spec}\,}
\newcommand{\Proj}{\operatorname{Proj}\,}

\newcommand{\id}{\operatorname{id}}

\newcommand{\op}{{\operatorname{op}}}

\newcommand{\Coker}{\operatorname{Coker}}

\newcommand{\pd}{\operatorname{pd}}

\newcommand{\barp}{{\bar{p}}}
\newcommand{\bull}{{\scriptstyle{{-1}}}}
\newcommand{\bff}{{\mathbf{f}}}

\newcommand{\Sch}{\mathop{\mathsf{Sch}}\nolimits}
\newcommand{\ASch}{\mathop{\mathsf{ASch}}\nolimits}
\newcommand{\Tria}{\mathop{\mathsf{Tria}}\nolimits}
\newcommand{\sDG}{\mathop{\mathsf{sDG}}\nolimits}

\title{Categorical resolutions of irrational singularities}

\author{Alexander Kuznetsov}
\address{\sloppy
\parbox{0.9\textwidth}{
{\bf A.K.: }Algebraic Geometry Section, Steklov Mathematical Institute,
8 Gubkin str., Moscow 119991 Russia
\hfill\\[5pt]
\phantom{{\bf A.K.: }}The Poncelet Laboratory, Independent University of Moscow
\hfill\\[5pt]
\phantom{{\bf A.K.: }}Laboratory of Algebraic Geometry, SU-HSE
\hfill
}\bigskip}
\email{akuznet@mi.ras.ru}

\author{Valery A.~Lunts}
\address{{\bf V.L.: }Department of Mathematics, Indiana University, Bloomington, IN 47405, USA} 
\email{vlunts@indiana.edu}

\thanks{A.K.\ was partially supported by
RFFI grants 10-01-93110, 10-01-93113, 11-01-00393, 11-01-00568, \hbox{11-01-92613-KO-a}, 12-01-33024, NSh-5139.2012.1,
the grant of the Simons foundation, and by AG Laboratory SU-HSE, RF government grant, ag.11.G34.31.0023.
}

\begin{document}

\begin{abstract}
We show that the derived category of any singularity over a field of characteristic $0$
can be embedded fully and faithfully into a smooth triangulated category which has a semiorthogonal 
decomposition with components equivalent to derived categories of smooth varieties. This provides 
a categorical resolution of the singularity.
\end{abstract}

\maketitle

\tableofcontents

\newcommand{\BL}{\textsf{BL}}
\newcommand{\Cone}{\textsf{Cone}}
\newcommand{\Hot}{\operatorname{\textrm{\rm Hot}}}
\newcommand{\Mor}{\operatorname{\textrm{\rm Mor}}}
\newcommand{\opp}{\textrm{op}}
\newcommand{\red}{\textrm{\sf red}}
\newcommand{\comp}{\textrm{\sf c}}
\newcommand{\perf}{\textrm{\sf perf}}
\newcommand{\mmod}{\textrm{-\rm mod}}
\newcommand{\dgm}{\textrm{-\rm dgm}}
\newcommand{\hinj}{\textrm{\rm h-inj}}
\newcommand{\hproj}{\textrm{\rm h-proj}}
\newcommand{\hflat}{\operatorname{\textrm{\rm h-flat}}}
\newcommand{\hflata}{\operatorname{\textrm{\rm h-flat}^\circ}}
\newcommand{\hflatperf}{\operatorname{\textrm{\rm h-flat-perf}}}
\newcommand{\hflatperfa}{\operatorname{\textrm{\rm h-flat-perf}^\circ}}
\newcommand{\comperf}{\operatorname{\textrm{\rm com-perf}}}
\newcommand{\comperfa}{\operatorname{\textrm{\rm com-perf}^\circ}}
\newcommand{\com}{\operatorname{\textrm{\rm com}}}
\newcommand{\coma}{\operatorname{\textrm{\rm com}^\circ}}
\newcommand{\hfp}{\textrm{h-fp}}
\newcommand{\cf}{{\text{\rm cf}}}
\newcommand{\cdcf}{\operatorname{\cD}}
\newcommand{\lotimes}{\stackrel{{\mathbb L}}\otimes}
\newcommand{\bbI}{\mathbb{I}}

\newcommand{\fa}{{\mathfrak{a}}}
\newcommand{\fr}{{\mathfrak{r}}}

\newcommand{\sa}{{\mathsf a}}
\newcommand{\sd}{{\mathsf d}}
\newcommand{\se}{{\mathsf e}}
\newcommand{\si}{{\mathsf i}}
\newcommand{\sj}{{\mathsf j}}
\newcommand{\sq}{{\mathsf q}}
\newcommand{\st}{{\mathsf t}}
\newcommand{\sap}{{\mathsf a}}
\newcommand{\sep}{{\mathsf e}}
\newcommand{\sip}{{\mathsf i}}
\newcommand{\sqp}{{\mathsf q}}
\newcommand{\stp}{{\mathsf t}}
\newcommand{\sS}{{\mathsf S}}

\newcommand{\Ob}{\operatorname{\mathrm Ob}}
\newcommand{\sY}{{\mathsf Y}}
\newcommand{\bp}{{\mathbf p}}
\newcommand{\kk}{{\Bbbk}}

\newcommand{\md}{{\text{-}\operatorname{mod}}}
\newcommand{\fmd}{{\text{-}\operatorname{fmod}}}
\newcommand{\qmd}{{\text{-}\operatorname{qmod}}}
\newcommand{\Ann}{\operatorname{Ann}}
\newcommand{\RHom}{\operatorname{RHom}}
\newcommand{\RCHom}{\operatorname{R\mathscr{H}\!\mathit{om}}}

\def\Acycl{\operatorname{\rm Acycl}}
\def\fqc{\operatorname{fQcoh}}
\def\qqc{\operatorname{qQcoh}}
\def\dqc{\operatorname{dQcoh}}

\section{Introduction}\label{s-intro}

Resolution of singularities is one of central concepts in algebraic geometry.
In many cases it allows to reduce the complicated geometry of singular schemes 
to a much more tractable geometry of smooth schemes. From the categorical point of view
a resolution $\pi:X \to Y$ manifests itself in a pair of adjoint functors
\begin{equation*}
L\pi^*:\bD(Y) \to \bD(X)
\qquad\text{and}\qquad
R\pi_*:\bD(X) \to \bD(Y)
\end{equation*}
(the derived pullback and the derived pushforward) between the derived categories 
of quasicoherent sheaves on $X$ and $Y$ respectively.
The functors are related by the projection formula
\begin{equation*}
R\pi_*(L\pi^*(F)) \cong F \lotimes R\pi_*\cO_X,
\end{equation*}
which shows that from the categorical point of view there are two completely different situations:
\begin{enumerate}
\item the canonical  morphism $\cO_Y \to R\pi_*\cO_X$ is an isomorphism;
\item the canonical  morphism $\cO_Y \to R\pi_*\cO_X$ is not an isomorphism.
\end{enumerate}
If (1) holds one says that $Y$ has {\sf rational singularities}. In this case $R\pi_*\circ L\pi^* \cong \id$,
hence the pullback functor $L\pi^*$ is fully faithful, so the singular category $\bD(Y)$ {\em embeds}\/ into
a smooth category $\bD(X)$. This embedding allows the reductions of geometrical questions on~$Y$ to those on $X$
which we mentioned above. 
On a contrary, if (2) holds, the functor $L\pi^*$ is not fully faithful, so it does not provide such a reduction.
So, from a categorical point of view, the usual resolution of a scheme which has irrational singularities
is not a resolution at all!

The goal of the present paper is to show that even for a scheme $Y$ with irrational singularities one
can construct a categorical resolution by gluing appropriately several derived categories of smooth varieties.
To be more precise, for each separated scheme of finite type~$Y$ over a field $\kk$ 
of characteristic $0$ we construct a nice triangulated
category $\cT$ with an adjoint pair of triangulated functors $\pi^*:\bD(Y) \to \cT$ and $\pi_*:\cT \to \bD(Y)$
such that $\pi_*\circ\pi^* \cong \id$ so that $\pi^*$ is fully faithful, and, moreover, $\cT$ enjoys 
a number of useful properties, see Definition~\ref{resboth} below.




Before going to precise definitions we have to explain what we mean by saying that $\cT$ is nice.
First of all, the category $\cT$ has to be smooth. The notion of smoothness is formulated in terms 
of DG-categories (an introduction into the subject can be found in section 3), so we have to start 
with recalling what is a smooth DG-category.


\begin{defi}\label{dgsmooth}
A small DG-category $\cD$ is {\sf smooth} if the diagonal bimodule $\cD$ is {\em perfect}.
In other words, if it is contained in the smallest Karoubian closed triangulated 
subcategory of the derived category $\bD(\cD^\op\otimes\cD)$ generated by representable bimodules.
\end{defi}

This leads to the following definitions.

\begin{defi}\label{trsmooth}
A cocomplete compactly generated triangulated category $\cT$ is {\sf smooth} 
if there exists a smooth DG-category $\cD$ such that its derived category $\bD(\cD)$ 
is equivalent to $\cT$ as a triangulated category.
\end{defi} 


It is well known (see~\cite{TV}) that if $X$ is a smooth variety then the derived category $\bD(X)$ is smooth.


Now let us give a precise definition of a categorical resolution. 
Recall that an object $F$ of a triangulated category $\cT$ is {\sf compact} if the functor $\Hom_\cT(F,-)$
commutes with arbitrary direct sums. The subcategory $\cT^\comp \subset \cT$ of compact objects in $\cT$
is triangulated. 
Compact objects in $\bD(Y)$, the unbounded derived category of quasicoherent sheaves 
on a separated scheme of finite type $Y$, are {\sf perfect complexes}, i.e. the objects 
which are locally quasiisomorphic to finite complexes of locally free sheaves of finite rank. 
The perfect complexes form a triangulated subcategory of $\bD^b(\coh(Y))$ which we denote by $\bD^\perf(Y)$.

\begin{defi}[\cite{K08}]\label{resboth}
A {\sf categorical resolution of a scheme $Y$}\/ is 
a smooth cocomplete compactly generated triangulated category $\cT$ with an adjoint pair of triangulated functors
\begin{equation*}
\pi^*:\bD(Y) \to \cT
\quad\text{and}\quad
\pi_*:\cT \to \bD(Y),
\end{equation*}
such that 
\begin{enumerate}
\item $\pi_*\circ\pi^* = \id$;
\item both $\pi^*$ and $\pi_*$ commute with arbitrary direct sums;
\item $\pi_*(\cT^\comp) \subset \bD^b(\coh(Y))$.
\end{enumerate}
\end{defi}

Note that the second property implies that $\pi^*(\bD^\perf(Y)) \subset \cT^\comp$ (see Lemma~\ref{cprops}).

If $\pi:X\to Y$ is a usual resolution and $Y$ has rational singularities then the category $\cT = \bD(X)$
with the functors $L\pi^*$ and $R\pi_*$ is a categorical resolution of $Y$. Note however, that if 
the singularities of $Y$ are not rational then $\bD(X)$ is {\em not} a categorical resolution of $Y$ 
since the composition $R\pi_*\circ L\pi^*$ is isomorphic to the tensor product with $R\pi_*\cO_X$, 
which is not the identity functor.

To formulate the main result of the paper we need one more notion. As we already have said the resolution
which we construct is a nice triangulated category. In fact it is even nicer than just a smooth triangulated 
category. It has a very geometric nature --- to be more precise it has a semiorthogonal decomposition 
with all components being derived categories of smooth algebraic varieties. We call such categories 
{\sf strongly geometric}. Recall also that a cocomplete compactly generated triangulated category 
$\cT$ is {\sf proper} if the category $\cT^\comp$ is $\Ext$-finite, which means that the vector 
space $\oplus_{i \in \Z} \Hom(F,G[i])$ is finite-dimensional for all $F,G \in \cT^\comp$.


The main result of the paper is the following

\begin{theo}\label{cr}
Any separated scheme of finite type $Y$ over a field of characteristic~$0$
has a categorical resolution by a strongly geometric triangulated category $\cT$.
If $Y$ is proper then so is the resolving category $\cT$.
\end{theo}

%
%
%
%

Note that we do not put any restrictions on the singularity of $Y$. In fact, it may be not normal,
reducible, and even nonreduced, the construction still works!

As we will soon explain the varieties whose derived categories appear as components of the categorical
resolution which we construct are strongly related to the usual process of desingularization of the reduced scheme $Y_\red$ --- 
if we fix a usual resolution $X \to Y_\red$ by a sequence of blowups with smooth centers $Z_0$, $Z_1$, \dots, $Z_{m-1}$ 
then the components of $\cT$ are $\bD(Z_0)$, $\bD(Z_1)$, \dots, $\bD(Z_{m-1})$ (each of them 
may be repeated several times!) and $\bD(X)$ (with repetitions if $Y$ itself is not reduced).
 
Let us outline the construction which we use.

First of all we use heavily the machinery of DG-categories (see section~\ref{s-p-dg} for 
an introduction into the subject). Accordingly we use a DG-version of the definition
of a categorical resolution. 

\begin{defi}\label{pcdgr}
A {\sf partial categorical DG-resolution} of a small pretriangulated DG-category $\cD$
is a small pretriangulated DG-category $\tilde\cD$ with a DG-functor $\pi:\cD \to \tilde\cD$ 
which induces a fully faithful functor on homotopy categories. 
If additionally $\tilde\cD$ is smooth we say that it is a {\sf categorical DG-resolution}.
\end{defi}

%


The first instrument of the construction is the notion of a {\sf gluing} of DG-categories.
Given two small DG-categories $\cD_1$, $\cD_2$ and a $\cD_1$-$\cD_2$-bimodule $\varphi$
we define a DG-category $\cD_1\times_\varphi\cD_2$. The definition can be found in section~\ref{s-glu},
here we will just say that it is a straightforward generalization of an upper triangular algebra,
with the diagonal entries being two given algebras and the upper diagonal entry being a bimodule
over these algebras. We show that the gluing of two pretriangulated categories is itself pretriangulated.
The derived and homotopy categories of the gluing have semiorthogonal decompositions
\begin{equation*}
[\cD_1\times_\varphi\cD_2] = \big\langle [\cD_1], [\cD_2] \big\rangle,
\qquad
\bD(\cD_1\times_\varphi\cD_2) = \big\langle \bD(\cD_1), \bD(\cD_2) \big\rangle.
\end{equation*}
The quasiequivalence class of the gluing depends only on the quasiisomorphism class of the gluing bimodule,
and the gluing is smooth if and only if both components $\cD_1$ and $\cD_2$ are smooth and the gluing bimodule
is perfect (see Proposition~\ref{utsm} and~\cite{LS}).

Another instrument is the category of $\csA$-modules. 
Given a scheme $S$, an integer $n$, and an ideal $\fr \subset \cO_S$ such that $\fr^n = 0$ we define
the sheaf of $\cO_S$-algebras $\csA = \csA_{S,\fr,n}$ as a certain subalgebra in
$\End_{\cO_S}(\cO_S \oplus \cO_S/\fr^{n-1} \oplus \dots \oplus \cO_S/\fr)$, see~\eqref{cadef}
for precise formula. The ringed space $(S,\csA_{S,\fr,n})$ is called {\sf an $\csA$-space}.
One of the points of the paper is that $\csA$-spaces share many good properties with schemes
(and have some advantages) so one can use them as building blocks for constructing triangulated
categories of geometric interest. \fbox{\tt !!!}

The category $\Qcoh(\csA)$ of quasicoherent $\csA$-modules on $S$ is an abelian category
and we prove that its derived category $\bD(\csA)$ gives a categorical resolution of $\bD(S)$ if the closed 
subscheme $S_0 \subset S$ corresponding to the ideal $\fr$ is smooth. Moreover,$\bD(\csA)$ comes with semiorthogonal decompositions
\begin{equation*}
\begin{array}{lll}
\bD(\csA) &=& \big\langle \bD(S_0),\bD(S_0),\dots,\bD(S_0) \big\rangle
\\
\bD^b(\coh(\csA)) &=& \big\langle \bD^b(\coh(S_0)),\bD^b(\coh(S_0)),\dots, \bD^b(\coh(S_0)) \big\rangle
\end{array}
\end{equation*}
%
%
(the number of components equals $n$).
Moreover, taking appropriate DG-enhancement $\cdcf(\csA)$ 
of the derived category of perfect $\csA$-modules gives a categorical DG-resolution
of $\cdcf(S)$, a DG-model of the category of perfect complexes on~$S$.

Now the main construction looks as follows. We consider a blowup $f: Y_1 \to Y$ with a smooth center $Z \subset Y$.
Then we show that for $n$ sufficiently large the $n$-th infinitesimal neighborhood $S$ of the subscheme $Z$
(i.e. the subscheme with $I_S = I_Z^n$) has the following important property:
\begin{equation*}
Rf_* I_{f^{-1}(S)} \cong I_S.
\end{equation*}
Here $f^{-1}(S)$ is the scheme-theoretic preimage of $S$. We call such a subscheme $S$ a {\sf nonrational center} for $f$.
We consider the nilpotent ideal $\fr := I_Z/I_S$ on $S$, so that we have $S_0 = Z$, and the category $\Qcoh(\csA_S) = \Qcoh(\csA_{S,\fr,n})$
of $\csA_S$-modules.
Further, we construct a $\cdcf(\csA_S)$-$\cdcf(Y_1)$-bimodule $\varphi$ such that the gluing
$\cdcf(\csA_S)\times_\varphi\cdcf(Y_1)$ is a partial categorical resolution of $\cdcf(Y)$.
Of course this is not yet a resolution --- there is no reason for $Y_1$ to be smooth.
However, this serves as a step of induction. With a wise choice of the first blowup center $Z$
the singularities of $Y_1$ are more simple than those of $Y$ and we can assume by induction that $\cdcf(Y_1)$
has a categorical DG-resolution~$\cD_1$. Then we replace the gluing $\cdcf(\csA_S)\times_\varphi\cdcf(Y_1)$
by $\cD := \cdcf(\csA_S)\times_{\tilde\varphi}\cD_1$, where $\tilde\varphi = \varphi\lotimes_{\cdcf(Y_1)}\cD_1$
(this procedure is called {\sf regluing}) and show that $\cD$ is a categorical DG-resolution of $\cdcf(Y)$.

The base of the induction is the case of a (possibly nonreduced) scheme $Y$ such that the associated 
reduced scheme $Y_\red$ is smooth. In this case the resolution is again provided by the derived category
of $\csA$-modules $\bD(\csA_Y)$. Thus we use the category of $\csA$-modules both in the base and in the step
of the induction. Note also that this category is essential for the construction even if the original
scheme $Y$ is reduced since the nonrational centers typically have nonreduced structure which has to be resolved.


Note that each step of the above construction seriously depends on a choice of integer~$n$ (which as it was mentioned above
can be chosen arbitrarily as soon as it is sufficiently large). It is worth mentioning that increasing $n$ by $1$
has an effect of adding one more semiorthogonal component equivalent to $\bD(Z)$ to the resolving category.
This is analogous to making an extra smooth blowup of a resolution of singularities in geometric situation.
Thus if one would like to construct ``smaller'' resolution, one should pick up $n$ as small as possible.

The constructed resolution has several important properties. First of all, replacing $Y$ by its open subsets
one obtains a presheaf of DG-categories on $Y$. If $Y$ is generically reduced then on a sufficiently small 
open subset $U \subset Y$ the corresponding DG-category coincides with $\cdcf(U)$ --- this expresses birationality 
of the resolution.
If $Y$ comes with an action of a group $G$ one can choose a resolution on which 
the same group $G$ acts in a compatible way.


Since our categorical resolution requires a usual resolution as an input it is restricted to characteristic $0$.
On the other hand, in cases where a usual resolution is known in positive characteristic (e.g. in dimension 3 
and sufficiently large characteristic) our construction can be applied to give a categorical resolution.

\bigskip

Now let us explain the relation of this paper to other notions of a categorical resolution
in the literature. The first appearance of this concept is due to Bondal and Orlov, \cite{BO02}.
The notion suggested in {\em loc.\ cit.} is much stronger than one used here --- it is assumed there
that the resolving category is the bounded derived category of an {\em abelian} category of finite
homological dimension (in particular the resolving category has a bounded t-structure), and 
moreover that $\bD^b(\coh(Y))$ is a localization of that category. We believe that the first assumption
is irrelevant and too restrictive, this is why we construct resolutions as derived categories 
of DG-categories (note however that the resolution $\bD(\csA_S)$ of a nonreduced scheme $S$
with smooth $S_0$ enjoys all the properties asked for in~\cite{BO02}). On the other hand, 
we believe that the resolutions which we construct enjoy the second property --- that $\bD^b(\coh(Y))$ 
is the quotient of $\cT^\comp$. 
Recently, Alexander Efimov announced a proof of this fact based on our approach.


Another notion, a {\em noncommutative crepant resolution}, was introduced by Van den Bergh in~\cite{VdB04}.	
It was even more restrictive than the original definition of Bondal and Orlov. In addition it was
assumed that the resolving abelian category can be realized as the category of sheaves of modules
over a certain sheaf of algebras on $Y$ with nice homological properties.

On the other hand, in~\cite{K08} the definition of a categorical resolution (in the context 
of small categories) used here first appeared.
However, it was expected that such resolution may exist
only if $Y$ has rational singularities, and the point of the paper was in finding minimal
resolutions by starting from a commutative resolution and then shrinking it by chopping out
some irrelevant semiorthogonal components. So, in a sense the result of the present paper shows
that the same approach may be used for nonrational singularities as well. 

Finally, in papers~\cite{L10a,L10b} the notion of a categorical resolution in the context 
of big categories was considered. In particular, it was shown in~\cite{L10b} that one can 
construct a categorical resolution of $\bD(Y)$ by a so-called smooth poset scheme if and only if
$Y$ has Du Bois singularities. A smooth poset scheme is obtained by ``gluing'' a finite number
of smooth schemes $X_\alpha$ (indexed by a poset) along arbitrary morphisms $f_{\alpha\beta}:X_\alpha \to X_\beta$
(if $\alpha \ge \beta$). Thus in case $Y$ has Du Bois singularities its categorical resolution
(as in Theorem~\ref{cr}) can be chosen in a more geometric way: one only needs to glue ``smooth schemes''
along honest morphisms and $\csA$-modules can be avoided. However, the resolving category in this 
paper is also constructed by similar gluings where some of smooth schemes are replaced by smooth $\csA$-spaces.

{\bf Acknowledgements.} The authors are grateful to the Indiana University in Bloomington where a significant 
part of the work has been done. We are very grateful to Dima Kaledin, Dima Orlov and Olaf Schn\"urer
for valuable discussions and Osamu Iyama for insightful comments. We also thank Bernhard Keller and
Bertrand To\"en for answering some questions. We are especially grateful to 
Sasha Efimov whose suggestions allowed to simplify considerably the section on $\csA$-modules.


\section{Preliminaries on triangulated categories}\label{s-p-tri}

General reference for the material on triangulated categories is \cite{BK89,BO95,BO02}.

\subsection{Semiorthogonal decompositions}

Let $\kk$ be a field and $\cT$ a $\kk$-linear triangulated category.

\begin{defi}[\cite{BK89,BO95}]\label{def-sod}
A {\sf semiorthogonal decomposition of triangulated category $\cT$}\/ is a collection $\cT_1,\dots,\cT_m$ of strictly full triangulated subcategories in $\cT$
({\sf components}\/ of the decomposition) such that
\begin{itemize}
\item $\Hom_\cT(\cT_i,\cT_j) = 0$ for $i >j$;
\item for any $F\in \cT$ there is a chain of maps
\begin{equation}\label{defsod}
0 = F_m \to F_{m-1} \to \dots \to F_1\to F_0 = F
\end{equation} 
such that $\Cone(F_i \to F_{i-1})\in \cT_i$ for all $i=1,\dots,m$.
\end{itemize}
If only the first property holds we will say that $\cT_1,\dots,\cT_m$ is a {\sf semiorthogonal collection} of subcategories.
We will use notation
\begin{equation*}
\cT = \langle \cT_1,\dots,\cT_m \rangle
\end{equation*}
to express a semiorthogonal decomposition of $\cT$ with components $\cT_1$, \dots, $\cT_m$.
\end{defi}

It follows from the definition that the chain of maps~\eqref{defsod} is functorial in $F$, and moreover, 
the cones of the maps in the chain are also functorial. In other words, there are functors
\begin{equation*}
\cT \to \cT_i,\qquad
F \mapsto \bp_i(F) := \Cone(F_i \to F_{i-1})
\end{equation*}
known as {\sf the projection functors of the semiorthogonal decomposition}.

For future convenience we will need to restate the definition of a semiorthogonal decomposition in the special case of $m = 2$ components.
In this case the chain~\eqref{defsod} for $F \in\cT$ looks as $0 \to F' \to F$ and the conditions are that 
$F' = \Cone(0 \to F') \in \cT_2$ and $\Cone(F' \to F) \in \cT_1$. In other words, $\cT  = \langle\cT_1,\cT_2\rangle$
is a semiorthogonal decomposition iff
\begin{itemize}
\item $\Hom(\cT_2,\cT_1) = 0$ and
\item for each $F \in \cT$ there is a distinguished triangle
\begin{equation*}
\bp_2(F) \to F \to \bp_1(F) \to \bp_2(F)[1]
\end{equation*}
with $\bp_i(F) \in \cT_i$.
\end{itemize}
One can also express the last property by saying that each object in $\cT$ can be represented as a cone of a morphism
from an object of $\cT_1$ to an object of $\cT_2$.

The following result is well known.

\begin{lemma}[\cite{BK89}]\label{sod3}
Assume that $\cT_1,\cT_2 \subset \cT$ is a semiorthogonal pair of full triangulated subcategories, such that 
\begin{itemize}
\item the embedding functor $i_1:\cT_1 \to \cT$ has a left adjoint $i_1^*:\cT \to \cT_1$, and
\item the embedding functor $i_2:\cT_2 \to \cT$ has a right adjoint $i_2^!:\cT \to \cT_2$.
\end{itemize}
Then there is a semiorthogonal decomposition
\begin{equation*}
\cT = \langle \cT_1, {}^\perp\cT_1 \cap \cT_2^\perp, \cT_2 \rangle,
\end{equation*}
where 
\begin{equation*}\arraycolsep=0pt
\begin{array}{lll}
{}^\perp\cT_1 \,&= \{ T \in \cT\ |\ \Hom(T,\cT_1) = 0 \} \,&= \Ker i_1^*,\\
\cT_2^\perp &= \{ T \in \cT\ |\ \Hom(\cT_2,T) = 0 \} &= \Ker i_2^!.
\end{array}
\end{equation*}
In particular, if ${}^\perp\cT_1 \cap\cT_2^\perp = 0$ then $\cT = \langle \cT_1,\cT_2 \rangle$.
Moreover, $\bp_1 = i_1i_1^*$, $\bp_2 = i_2i_2^!$.
\end{lemma}

\subsection{The gluing functor}\label{ss-gf}

Let $\cT = \langle\cT_1,\cT_2\rangle$ be a semiorthogonal decomposition. 

\begin{defi}\label{glubim}
The {\sf gluing bifunctor}\/ of a semiorthogonal decomposition $\cT = \langle \cT_1,\cT_2 \rangle$
is the functor $\Phi:\cT_1^\op\times\cT_2 \to \kk\mmod$ defined by
\begin{equation*}
\Phi(F_1,F_2) = \Hom_\cT(F_1,F_2[1]).
\end{equation*}
\end{defi}

For each object $F_2 \in \cT_2$ we have a contravariant cohomological functor $\Phi(-,F_2):\cT_1^\op \to \kk\mmod$.
Assume that for all $F_2\in\cT_2$ this functor is representable. Then there exists a functor
$\phi:\cT_2 \to \cT_1$ with a functorial isomorphism
\begin{equation}\label{glufun}
\Phi(F_1,F_2) \cong \Hom_{\cT_1}(F_1,\phi(F_2)).
\end{equation} 
If such a functor $\phi$ exists then it is unique up to an  isomorphism.

\begin{defi}\label{glumod}
The functor $\phi$ with the property~\eqref{glufun} is called the {\sf gluing functor}\/ 
of a semiorthogonal decomposition $\cT = \langle \cT_1,\cT_2 \rangle$.
\end{defi}

It is easy to see that if $i_1$ has also right adjoint $i_1^!$ then 
\begin{equation}\label{phiii}
\phi = i_1^!i_2[1] 
\end{equation} 
is the gluing functor.

The importance of the gluing functor is shown by the following

\begin{lemma}\label{objsod}
Assume $\cT = \langle \cT_1,\cT_2 \rangle$ and a gluing functor $\phi$ exists.
To give an object $F$ of $\cT$ is equivalent to giving an object $F_1$ of $\cT_1$, an object $F_2$ of $\cT_2$
and a morphism $f:F_1 \to \phi(F_2)$ in $\cT_1$.
\end{lemma}
\begin{proof}
Let $F \in \cT$ and put $F_i = \bp_i(F)$. Then we have a distinguished triangle
\begin{equation*}
F_2 \to F \to F_1 \to F_2[1].
\end{equation*}
The connecting morphism $F_1 \to F_2[1]$ is given by an element $f$ of $\Hom_{\cT}(F_1,F_2[1])$
that is of $\Hom_{\cT_1}(F_1,\phi(F_2))$. Thus we produce $F_1$, $F_2$, and $f$ from $F$.
Vice versa, if $F_1$, $F_2$, and $f$ are given we interpret $f$ as a morphism $F_1 \to F_2[1]$
in $\cT$ and take $F = \Cone(F_1 \to F_2[1])[-1]$.
\end{proof}

This Lemma motivates the definition of the gluing category in section~\ref{s-glu}.

\subsection{Compact objects}

Recall that a triangulated category $\cT$ is {\sf cocomplete}\/ if it has arbitrary small direct sums.

\begin{defi}[\cite{N92}]\label{defcomp}
An object $F \in \cT$ in a cocomplete triangulated category $\cT$ is {\sf compact}\/ 
if for any set of objects $G_i \in \cT$ the canonical morphism
\begin{equation}\label{homsum}
\oplus \Hom(F,G_i) \to \Hom(F,\oplus G_i)
\end{equation} 
is an isomorphism.
\end{defi}

Compact objects of a cocomplete triangulated category $\cT$ form a triangulated subcategory in $\cT$ which we denote by $\cT^\comp$.

\begin{defi}[\cite{N92}]\label{cgen}
Triangulated category $\cT$ is {\sf generated} by a class $\cS$ of compact objects if $\cS^\perp = 0$.
In particular, $\cT$ is {\sf compactly generated} if $(\cT^\comp)^\perp = 0$.
\end{defi}

An important result of Neeeman which we will frequently use is the following

\begin{prop}[\cite{N92,N01}]\label{cgenprop}
If $\cS \subset \cT^\comp$ is a set of compact objects which generates $\cT$ then the minimal
triangulated subcategory of $\cT$ containing $\cS$ and closed under arbitrary direct sums is $\cT$ itself.
\end{prop}

If $X$ is a separated $\kk$-scheme of finite type and $\bD(X)$ is the unbounded derived category
of quasicoherent sheaves on $X$ then $\bD(X)$ is cocomplete, its subcategory of compact objects coincides with the category 
of perfect complexes on $X$
\begin{equation}\label{compperf}
\bD(X)^\comp = \bD^\perf(X),
\end{equation} 
and $\bD(X)$ is compactly generated (see~\cite{N96}).
Recall that a {\sf perfect complex on $X$}\/ is an object of $\bD(X)$ which is locally quasiisomorphic
to a bounded complex of locally free sheaves of finite rank. In particular if $X$ is {\em smooth} then
\begin{equation}\label{compcoh}
\bD(X)^\comp = \bD^b(\coh (X)),
\end{equation} 
the bounded derived category of coherent sheaves.

\begin{defi}\label{csums}
A functor $\Phi:\cT_1\to\cT_2$ between cocomplete triangulated categories {\sf commutes with direct sums} 
if for any set of objects $G_i \in \cT_1$ the canonical morphism 
\begin{equation*}
\oplus \Phi(G_i) \to \Phi(\oplus G_i)
\end{equation*}
is an isomorphism. We also say that $\Phi$ {\sf preserves compactness}\/ if $\Phi(\cT_1^\comp) \subset \cT_2^\comp$.
\end{defi}

The following simple observation will be used very frequently.

\begin{lemma}\label{cprops}
Let $\Phi:\cT_1 \to \cT_2$ be a triangulated functor.\\
$(1)$ Assume that $\Phi$ is fully faithful and commutes with direct sums. If $\Phi(F)$ is compact then $F$ is compact.\\
$(2)$ Assume that $\Phi$ has a right adjoint functor $\Phi^!$ and $\cT_1$ is compactly generated. 
Then $\Phi$ preserves compactness if and only if $\Phi^!$ commutes with direct sums.
\end{lemma}
\begin{proof}
(1) Let $F,F_i \in\cT_1$. Consider the following commutative diagram
\begin{equation*}
\xymatrix{
\oplus \Hom(F,F_i) \ar[r] \ar[d]_\Phi & \Hom(F,\oplus F_i) \ar[d]^\Phi \\
\oplus \Hom(\Phi(F),\Phi(F_i)) \ar[r] & \Hom(\Phi(F),\oplus \Phi(F_i)) 
}
\end{equation*}
The vertical arrows are isomorphisms since $\Phi$ is fully faithful and commutes with direct sums.
The bottom arrow is an isomorphism since $\Phi(F)$ is compact. Hence the top arrow is an isomorphism,
so $F$ is compact.

(2) Let $F \in \cT_1$, $G_i \in \cT_2$. Consider the following commutative diagram
\begin{equation*}
\xymatrix{
\oplus \Hom(\Phi(F),G_i) \ar@{=}[r] \ar[d] & \oplus \Hom(F,\Phi^!(G_i)) \ar[r] & \Hom(F,\oplus \Phi^!(G_i)) \ar[d] \\
\Hom(\Phi(F), \oplus G_i) \ar@{=}[rr] && \Hom(F,\Phi^!(\oplus G_i))
}
\end{equation*}
where equalities stand for the adjunction isomorphisms. Let $F$ be compact. Then the arrow in the top row is an isomorphism.
If $\Phi$ preserves compactness then the left vertical arrow is an isomorphism, hence the right arrow is an isomorphism as well.
Since this is true for any compact $F$ and $\cT_1$ is compactly generated, it follows that the cone of the canonical map 
$\oplus \Phi^!(G_i) \to \Phi^!(\oplus G_i)$ is zero, hence $\Phi^!$ commutes with direct sums. Vice versa, if $\Phi^!$ commutes
with the direct sums then the right arrow is an isomorphism hence the left arrow is an isomorphism as well, hence $\Phi(F)$
is compact. Thus $\Phi$ preserves compactness.
\end{proof}

%

The following result is very useful to ensure existence of a right adjoint functor.

\begin{theo}[Brown representability \cite{N96}]\label{brown}
Assume that $\Phi:\cT_1 \to \cT_2$ is a triangulated functor, $\cT_1$ is cocomplete and generated by a set of compact objects.
Then $\Phi$ has a right adjoint functor if and only if $\Phi$ commutes with direct sums.
\end{theo}

Another useful observation is the following

\begin{lemma}\label{cff}
Let $\Phi:\cT \to \cT'$ be a triangulated functor between cocomplete triangulated categories which commutes with arbitrary direct sums.
Let $\cT_0 \subset \cT$ be a subcategory of compact objects which generates $\cT$. If $\Phi$ preserves compactness
and is fully faithful on $\cT_0$ then $\Phi$ is fully faithful on the whole $\cT$. If, moreover, $\Phi(\cT_0)$ generates $\cT'$
then $\Phi$ is an equivalence.
\end{lemma}
\begin{proof}
By Brown representability $\Phi$ has a right adjoint functor $\Phi^!$. Since $\Phi$ preserves compactness
the adjoint $\Phi^!$ also commutes with arbitrary direct sums by Lemma~\ref{cprops}. Consider the adjunction unit
$\id \to \Phi^!\Phi$ and let $\cT_1 \subset \cT$ by the full subcategory consisting of all objects $F$
on which this morphism is an isomorphism. Clearly, it is triangulated and closed under arbitrary direct sums
(since both $\Phi$ and $\Phi^!$ commute with those). Let us show that it contains $\cT_0$. Indeed, take any $F,G \in \cT_0$,
consider the triangle
\begin{equation*}
G \to \Phi^!(\Phi(G)) \to G'
\end{equation*}
and apply $\Hom(F,-)$ to it. Since the functor $\Phi$ is full and faithful on $\cT_0$ we have
$\Hom(F,\Phi^!(\Phi(G))) \cong \Hom(\Phi(F),\Phi(G)) \cong \Hom(F,G)$, hence $\Hom(F,G') = 0$.
Since this is true for any $F \in \cT_0$ and $\cT_0$ generates $\cT$ we conclude that $G' = 0$.
Thus $G \in \cT_1$. Since this is true for any $G$ we conclude that $\cT_0 \subset \cT_1$.
Now, since $\cT_0$ generates $\cT$ and $\cT_1$ is closed under arbitrary direct sums,
using Proposition~\ref{cgenprop} we conclude that $\cT_1 = \cT$, and so $\Phi^!\Phi \cong \id$. 
This implies that $\Phi$ is fully faithful on the whole $\cT$. 

Finally, assume that $\Phi(\cT_0)$ generates $\cT'$. Note that $\Phi(\cT)$ is a full 
triangulated subcategory of $\cT'$ (since $\Phi$ is fully faithful), closed under
arbitrary direct sums (since $\Phi$ commutes with those) and containing $\Phi(\cT_0)$.
Consequently, by Proposition~\ref{cgenprop} we have $\Phi(\cT) = \cT'$, hence $\Phi$ is an equivalence.
\end{proof}

%
%
%
%

\section{Preliminaries on DG-categories}\label{s-p-dg}

A general reference for DG-categories is an excellent review of Keller~\cite{Ke06}.
See also~\cite{D04} and~\cite{T07}.

\subsection{DG-categories and DG-functors}

A DG-category $\cD$ over a field $\kk$ is a $\kk$-linear category such that
\begin{itemize}
\item for all $X_1,X_2 \in \cD$ the space $\Hom_\cD(X_1,X_2)$ is equipped with a structure
of a complex of $\kk$-vector spaces, and
\item the multiplication map 
\begin{equation*}
\Hom_\cD(X_2,X_3)\otimes_\kk\Hom_\cD(X_1,X_2) \to \Hom_\cD(X_1,X_3)
\end{equation*}
is a morphism of complexes.
\end{itemize}


By definition $\Hom_\cD(X,Y) = \oplus_{k\in\bbZ} \Hom^k_\cD(X,Y)$ is a graded vector space with
a differential $d:\Hom^k_\cD(X,Y) \to \Hom^{k+1}_\cD(X,Y)$. The elements $f \in \Hom^k_\cD(X,Y)$
are called {\sf homogeneous morphisms of degree $k$}, $\deg f = k$. The second part of the definition
of a DG-category is just the Leibniz rule for the composition of homogeneous morphisms
$d(fg) = (df)g + (-1)^{\deg f}f(dg)$.

If $\cD$ is a DG-category then the opposite DG-category $\cD^\op$ is defined as the category with the same objects and
$\Hom_{\cD^\op}(X_1,X_2) = \Hom_{\cD}(X_2,X_1)$ and the composition induced by the composition in $\cD$ twisted by
the sign $(-1)^{\deg f\deg g}$.

If $\cD_1$ and $\cD_2$ are DG-categories, we define their tensor product $\cD_1\otimes_\kk\cD_2$ as the DG-category
with objects $\cD_1\times\cD_2$ and morphisms defined by
\begin{equation*}
\Hom_{\cD_1\otimes_\kk\cD_2}((X_1,X_2),(Y_1,Y_2)) = \Hom_{\cD_1}(X_1,Y_1) \otimes_\kk \Hom_{\cD_2}(X_2,Y_2)
\end{equation*}

The simplest example of a DG-category is the category $\kk\dgm$ of complexes of $\kk$-vector spaces with
\begin{equation*}
\Hom^k(V_1,V_2) = 
\prod_{i \in \bbZ} \Hom(V_1^i,V_2^{i+k}),
\qquad
d(f) = d_{V_2}\circ f - (-1)^{\deg f}f\circ d_{V_1}.
\end{equation*}

Also note that each $\kk$-linear category can be considered as a DG-category with the same $\Hom$-spaces
with zero differential and zero grading.

A $\kk$-linear functor $F:\cD_1 \to \cD_2$ is a {\sf DG-functor} if the morphisms 
\begin{equation*}
F_{X_1,X_2}:\Hom_{\cD_1}(X_1,X_2) \to \Hom_{\cD_2}(F(X_1),F(X_2))
\end{equation*}
are morphisms of complexes, i.e.\ preserve the grading and commute with the differentials.

If $\cD_1$ is a {\em small}\/ DG-category (i.e. objects of $\cD_1$ form a set) then
all DG-functors from $\cD_1$ to $\cD_2$ form a DG-category $\Fun(\cD_1,\cD_2)$ with
$\Hom_{\Fun(\cD_1,\cD_2)}(F,G)$ defined as
\begin{equation*}
\Ker \left( \prod_{X \in \cD_1} \Hom_{\cD_2}(F(X),G(X)) \xrightarrow{} \prod_{X,Y \in \cD_1} \Hom(\Hom_{\cD_1}(X,Y),\Hom_{\cD_2}(F(X),G(Y))) \right)
\end{equation*}
In other words an element $f \in \Hom_{\Fun(\cD_1,\cD_2)}^k(F,G)$ consists of all collections of morphisms $f_X \in \Hom^k_{\cD_2}(F(X),G(X))$
given for all $X \in \cD_1$ such that for any morphism $g \in \Hom_{\cD_1}^l(X,Y)$ one has $f_Y\circ F(g) = (-1)^{kl} G(g)\circ f_X$. The differential
is induced by the differentials in $\Hom_{\cD_2}(F(X),G(X))$.

A {\sf right DG-module} over $\cD$ is a DG-functor $M:\cD^\op \to \kk\dgm$. If $\cD$ is small then the category $\cD\dgm$ of right DG-modules
over $\cD$ is a DG-category and the {\sf Yoneda functor}
\begin{equation*}
\sY:\cD \to \cD\dgm,
\qquad
X \mapsto \Hom_\cD(-,X)
\end{equation*}
is a DG-functor. The DG-module $\sY(X)$ obtained by applying the Yoneda functor to an object $X \in \cD$ will be denoted by $\sY^X$.
It is called {\sf representable} DG-module, and one says that the object $X$ is the corresponding {\sf representing object}.
One has an analogue of the Yoneda Lemma
\begin{equation}\label{yone}
\Hom_{\cD\dgm}(\sY^X,M) \cong M(X)
\end{equation}
for any DG-module $M \in \cD\dgm$.
In other words, the Yoneda functor is fully faithful, so $\cD$ can be considered as 
a full DG-subcategory in $\cD\dgm$.

Analogously, a {\sf left DG-module} over $\cD$ is a DG-functor $\cD \to \kk\dgm$. Note that a left $\cD$-module 
is the same as a right $\cD^\op$-module. So, the Yoneda functor can be considered also in this case
\begin{equation*}
\sY^\op:\cD^\op \to \cD^\op\dgm,
\qquad
X \mapsto \Hom_\cD(X,-).
\end{equation*}
We will write $\sY_X$ for $\sY^\op(X)$.


%

\subsection{The homotopy category}

The {\sf homotopy category} $[\cD]$ of a DG-category $\cD$ is defined as the category with
\begin{itemize}
\item $\Ob[\cD] = \Ob\cD$,
\item $\Hom_{[\cD]}(X,Y) = H^0(\Hom_{\cD}(X,Y))$.
\end{itemize}
The homotopy category is a $\kk$-linear category.

One says that closed morphisms of degree zero $f,g \in \Hom_\cD(X,Y)$ are {\sf homotopic}
if they induce equal morphisms in $[\cD]$. In other words if there is $h \in \Hom^{-1}_\cD(X,Y)$ 
such that $f - g = dh$ (this $h$ is a {\sf homotopy} between $f$ and $g$). Further, two objects $X$ and $Y$
in~$\cD$ are {\sf homotopic} (or {\sf homotopy equivalent}) if they are isomorphic in~$[\cD]$.
In other words, if there are closed morphisms of degree zero $f \in \Hom_\cD(X,Y)$
and $g \in \Hom_\cD(Y,X)$ such that $fg$ is homotopic to $\id_Y$ and $gf$ is homotopic to $\id_X$.

Each DG-functor $F:\cD_1 \to \cD_2$ induces a functor $[F]:[\cD_1] \to [\cD_2]$ on homotopy categories.
A DG-functor $F:\cD_1 \to \cD_2$ is a {\sf quasiequivalence} if for all $X_1,Y_1 \in \cD_1$ the morphism 
$F:\Hom_{\cD_1}(X_1,Y_1) \to  \Hom_{\cD_2}(F(X_1),F(Y_1))$ is a quasiisomorphism of complexes 
(in this case $F$ is called {\sf quasi fully faithful}) and for each object $X_2 \in \cD_2$ 
there is an object $X_1 \in \cD_1$ such that $F(X_1)$ is homotopy equivalent to $X_2$.
In particular, $[F]$ is an equivalence of categories.


Two DG-categories $\cD$ and $\cD'$ are called {\sf quasiequivalent} if there is a chain of quasiequivalences
\begin{equation*}
\cD = \cD_0 \xrightarrow{\ \cong\ } \cD_1 \xleftarrow{\ \cong\ } \cD_2 \xrightarrow{\ \cong\ } \cD_3 \xleftarrow{\ \cong\ } \dots \xrightarrow{\ \cong\ } \cD_n = \cD'.
\end{equation*}
It is clear that quasiequivalent DG-categories have equivalent homotopy categories.

\subsection{Pretriangulated categories}

Let $\cD$ be a DG-category and $M$ a DG-module over~$\cD$.
For each integer $k\in\bbZ$ the {\sf shift} $M[k]$ is defined as the DG-module over $\cD$ with
\begin{equation*}
M[k](X) = M(X)[k]
\end{equation*}
for all $X \in \cD$, where the RHS is the usual shift of the complex $M(X)$.
Further, let $f:M \to N$ be a closed morphism of DG-modules of degree~$0$.
Then for any object $X \in \cD$ we have a morphism of complexes $f_X:M(X) \to N(X)$.
The {\sf cone of the morphism $f$} is defined by
\begin{equation*}
\Cone(f)(X) = \Cone(f_X:M(X)\to N(X)).
\end{equation*}
Again, the RHS is the usual cone of a morphism of complexes.

\begin{remark}\label{cone}
In other words, the value of $\Cone(f)$ on $X \in \cD$ is the direct sum of graded vector spaces $N(X) \oplus M(X)[1]$
with the differential $d(n,m) = (dn + f(m),-dm)$. Thus the cone comes with a bunch of morphisms:
\begin{equation*}
M[1] \xrightarrow{i} \Cone(f) \xrightarrow{p} M[1],
\qquad
N \xrightarrow{j} \Cone(f) \xrightarrow{s} N,
\end{equation*}
such that
\begin{equation*}
pi = \id_{M[1]},\quad
sj = \id_N,\quad
ip + js = \id_{\Cone(f)},\qquad
dj = dp = 0,\quad
di = jf,\quad
ds = -fp.
\end{equation*}
Vice versa, given a DG-module $C$ with morphisms $M[1] \xrightarrow{i} C \xrightarrow{p} M[1]$,
$N \xrightarrow{j} C \xrightarrow{s} N$ satisfying the above conditions it is easy to check 
that $C$ is isomorphic to the cone of $f$.
\end{remark}

Note also that $\Cone(M[-1] \xrightarrow{\ 0\ } N) \cong M \oplus N$ is the direct sum of DG-modules.


For any small DG-category $\cD$ the shift functor and the cone functor defined above
induce on the homotopy category $[\cD\dgm]$ of DG-modules over $\cD$ a structure 
of a triangulated category \cite{BK90}.

\begin{defi}[\cite{BK90}]\label{defpretri}
Let $\cD$ be a small DG-category. A DG-subcategory $\cD' \subset \cD\dgm$
is {\sf pretriangulated} if its homotopy category $[\cD']$ is a triangulated
subcategory of $[\cD\dgm]$. In particular, a small DG-category $\cD$ is pretriangulated if
%
%
\begin{enumerate}
\item for any $X \in \cD$ and any $k \in \bbZ$ the DG-module $\sY^X[k]$ is homotopic to a representable DG-module;
\item for any closed morphism $f \in\Hom_\cD(X_1,X_2)$ of degree $0$ in $\cD$ 
the DG-module $\Cone(\sY^f:\sY^{X_1} \to \sY^{X_2})$ is homotopic to a representable DG-module.
\end{enumerate}
\end{defi}

%

The homotopy category of a pretriangulated DG-category is triangulated. Note also that if 
$F:\cD_1 \to \cD_2$ is a DG-functor between small pretriangulated DG-categories 
then the functor $[F]:[\cD_1] \to [\cD_2]$ is triangulated.

\begin{defi}[\cite{BK90}]
An {\sf enhancement} for a triangulated category $\cT$ is a pretriangulated DG-category $\cD$
with an equivalence $\cT \cong [\cD]$ of triangulated categories.
\end{defi}


\subsection{Derived category}

Note that the category of DG-modules over a DG-category $\cD$ has arbitrary direct sums,
which are just componentwise
\begin{equation}\label{sumdgmod}
(\oplus M_i)(X) = \oplus M_i(X)
\end{equation} 
for any set of DG-modules $M_i$ and any $X\in \cD$.

A DG-module $M$ over $\cD$ is {\sf acyclic}, if for any $X\in\cD$ the complex $M(X)$ is acyclic.
The DG-subcategory of acyclic DG-modules is denoted  $\Acycl(\cD)$.
It is evidently pretriangulated and closed under arbitrary direct sums, hence its homotopy category $[\Acycl(\cD)]$
is a localizing triangulated subcategory in $[\cD\dgm]$.

The derived category of DG-modules over $\cD$ is defined as the Verdier quotient
\begin{equation*}
\bD(\cD) = [\cD\dgm]/[\Acycl(\cD)].
\end{equation*}
It is a triangulated category. 
Two DG-modules are called {\sf quasiisomorphic}, if they are isomorphic in the derived category of DG-modules.

By~\cite{N92} the functor $[\cD\dgm] \to \bD(\cD)$
commutes with direct sums. It follows that the images of representable objects are compact in $\bD(\cD)$.
In particular, $\bD(\cD)$ is compactly generated. In fact, compact objects in $\bD(\cD)$ can be described as follows.

Consider the minimal subcategory of $\cD\dgm$ containing all representable DG-modules and closed under 
shifts, cones of closed morphisms, and homotopy direct summands. Objects of this category are called
{\sf perfect} DG-modules. By~\cite[Theorem 3.4]{Ke06} compact objects of $\bD(\cD)$ are given by
perfect DG-modules. In particular, if $\cD$ is pretriangulated and closed under homotopy direct summands 
then 
\begin{equation*}
\bD(\cD)^\comp = [\cD].
\end{equation*}

A DG-module $P$ is called {\sf h-projective} if for any acyclic DG-module $A$
the complex $\Hom_{\cD\dgm}(P,A)$ is acyclic. The DG-subcategory of $\cD\dgm$ consisting
of h-projective DG-modules is denoted $\hproj(\cD)$.
A DG-module $I$ is called {\sf h-injective} if for any acyclic DG-module $A$
the complex $\Hom_{\cD\dgm}(A,I)$ is acyclic. The DG-subcategory of $\cD\dgm$ consisting
of h-injective DG-modules is denoted $\hinj(\cD)$.
Note that by definition $\hproj(\cD)$ is just the left orthogonal to $\Acycl(\cD)$ in $[\cD\dgm]$,
while $\hinj(\cD)$ is just the right orthogonal. In fact, the following is well known

\begin{theo}[\protect{\cite[Prop.\ 3.1]{Ke06}}]\label{hprhin}
There are semiorthogonal decompositions
\begin{equation*}
[\cD\dgm] = \langle [\Acycl(\cD)], [\hproj(\cD)] \rangle,
\qquad
[\cD\dgm] = \langle [\hinj(\cD)], [\Acycl(\cD)] \rangle.
\end{equation*}
In particular, we have equivalences of triangulated categories
\begin{equation*}
[\hproj(\cD)] \cong \bD(\cD) \cong [\hinj(\cD)].
\end{equation*}
Thus the categories $\hproj(\cD)$ and $\hinj(\cD)$ are pretriangulated.
\end{theo}

It follows from Theorem~\ref{hprhin} that each DG-module $M$ is quasiisomorphic to an h-projective
DG-module $P_M$, which is called its {\sf h-projective resolution}. Analogously, each DG-module $M$ 
is quasiisomorphic to an h-injective DG-module $I_M$, which is called its {\sf h-injective resolution}.
Thus the category of h-projective (resp.\ h-injective) $\cD$-modules provides an enhancement for the derived category.

Note also that the image of the Yoneda functor consists of h-projective $\cD$-modules --- this follows
immediately from~\eqref{yone}. Thus $[\cD]$ can be considered as a full subcategory of $\bD(\cD)$ and
$\cD$ is pretriangulated if and only if this subcategory is triangulated.

\subsection{Tensor products}

If $\cD$ is a small DG-category the {\sf tensor product} of a right and a left DG-modules $M \in\cD\dgm$ and $N\in\cD^\op\dgm$ is defined as
\begin{equation*}
M\otimes_\cD N = \Coker\left(
\bigoplus_{X,Y \in\cD} M(X)\otimes_\kk \Hom_\cD(Y,X) \otimes_\kk N(Y) \xrightarrow{\quad}
\bigoplus_{X \in\cD} M(X) \otimes_\kk N(X)
\right).
\end{equation*}
It is a complex of vector spaces. One checks immediately that
\begin{equation}\label{yonetens}
\sY^X\otimes_\cD N = N(X),
\qquad
M\otimes_\cD\sY_X = M(X).
\end{equation} 

A DG-module $F$ is called {\sf h-flat} if for any acyclic left DG-module $A$ the tensor product $F\otimes_\cD A$ is acyclic.
It follows from~\eqref{yonetens} that each representable DG-module is h-flat. Also one can check that each h-projective 
DG-module is h-flat.


The {\sf derived tensor product} is defined by using h-flat resolutions of either of the factors
\begin{equation*}
M\lotimes_\cD N := F_M \otimes_\cD N \cong F_M\otimes_\cD F_N \cong M \otimes_\cD F_N,
\end{equation*}
where $F_M$ and $F_N$ are h-flat resolutions of $M$ and $N$ respectively (for example one can use h-projective resolutions).
By definition the derived tensor product is defined up to a quasiisomorphism. 

Since as we already mentioned each representable DG-module is h-flat we conclude that
\begin{equation}\label{yoneltens}
\sY^X\lotimes_\cD N = N(X),
\qquad
M\lotimes_\cD\sY_X = M(X).
\end{equation}

\subsection{Bimodules}

Let $\cD_1$ and $\cD_2$ be DG-categories. A $\cD_1-\cD_2$ {\sf DG-bimodule} is a DG-module
over $\cD_1^\op\otimes\cD_2$, i.e.\ a DG-functor $\cD_1\otimes\cD_2^\op \to \kk\dgm$.

One says that a DG-bimodule $\varphi \in \cD_1^\op\otimes\cD_2\dgm$ has some property 
as a left (right) DG-module, if for each $X_2 \in \cD_2$ (resp.\ for each $X_1 \in \cD_1$)
DG-module $\varphi(-,X_2) \in \cD_1^\op\dgm$ (resp.\ $\varphi(X_1,-)\in\cD_2\dgm$)
has this property. In particular we can say that a bimodule is representable (h-projective, h-flat, h-injective, \dots)
as a left (right) DG-module (for short left representable, h-projective, \dots).

Given a left DG-module $M_1$ over $\cD_1$ and a right DG-module $M_2$ over $\cD_2$ one defines their
{\sf exterior tensor product} as 
\begin{equation*}
(M_1\otimes_\kk M_2)(X_1,X_2) := M_1(X_1)\otimes_\kk M_2(X_2).
\end{equation*}
Clearly, this is a DG-bimodule. Moreover, it is clear that an exterior product of two representable bimodules
is representable
\begin{equation*}
\sY_{X_1}\otimes_\kk\sY^{X_2} \cong \sY^{(X_1,X_2)}.
\end{equation*}
One can check that exterior tensor product of h-projective DG-modules is an h-projective DG-bimodule.
Vice versa, an h-projective DG-bimodule is both left and right h-projective.
Also one can check that an h-flat DG-bimodule is both left and right h-flat. 

Let $\varphi_{12} \in (\cD_1^\op\otimes\cD_2)\dgm$ and $\varphi_{23} \in (\cD_2^\op \otimes\cD_3)\dgm$ 
be two DG-bimodules. Their {\sf tensor product} is the $\cD_1-\cD_3$-DG-bimodule 
defined by
\begin{equation*}
(\varphi_{12} \otimes_{\cD_2} \varphi_{23})(X_1,X_3) = \varphi_{12}(X_1,-)\otimes_{\cD_2}\varphi_{23}(-,X_3)
\end{equation*}
for all $X_1 \in \cD_1$ and $X_3 \in \cD_3$. 

Analogously, {\sf the derived tensor product of bimodules} is defined by replacing the first bimodule with a right h-flat resolution
or the second with a left h-flat resolution. 

For each DG-category $\cD$ we denote by $\cD$ the {\sf diagonal bimodule}
\begin{equation*}
\cD(X_1,X_2) = \Hom_\cD(X_2,X_1) \in \cD^\op\otimes\cD\dgm.
\end{equation*}

\begin{remark}\label{dr}
Note that the positions of $X_1$ and $X_2$ in the left and in the right sides of the formula are interchanged.
\end{remark}

It is clear that the diagonal bimodule is both left and right representable. Moreover, it follows immediately from~\eqref{yonetens} that
\begin{equation}\label{diagtens}
M_1\otimes_\cD\cD \cong M_1,\qquad
\cD\otimes_\cD M_2 \cong M_2,
\end{equation}
for all $M_1 \in \cD_1\dgm$ and $M_2 \in \cD_2^\op\dgm$. In particular, the diagonal bimodule is both left and right h-flat
(but not h-flat as a bimodule), so one can derive freely the tensor product in the above formulas:
\begin{equation}\label{diagltens}
M_1\lotimes_\cD\cD \cong M_1,\qquad
\cD\lotimes_\cD M_2 \cong M_2.
\end{equation}


Let $\varphi \in \cD_1^\op\otimes\cD_2\dgm$ be a DG-bimodule. If $F:\cD'_1 \to \cD_1$ and $G:\cD'_2 \to \cD_2$ are DG-functors we define
${}_F\varphi_G$ to be a $\cD'_1$-$\cD'_2$ DG-bimodule defined by
\begin{equation*}
{}_F\varphi_G(X_1',X'_2) = \varphi(F(X'_1),G(X'_2)).
\end{equation*}
In particular, if $F:\cD' \to \cD$ is a DG-functor we have DG-bimodules ${}_F\cD$, $\cD_F$, and ${}_F\cD_F$.

\subsection{Smoothness and properness}\label{ss-smpr}

As we already mentioned in the Introduction a small DG-category $\cD$ is {\sf smooth} if the diagonal bimodule
$\cD$ is a perfect bimodule. In other words, if $\cD$ is a homotopy direct summand of a bimodule obtained
from representative bimodules by finite number of shifts and cones of closed morphism.

A small DG-category is {\sf proper} if for all objects $X,Y \in \cD$ the complex $\cD(X,Y)$ has bounded
and finite-dimensional cohomology. 

We will need the following result of To\"en and Vaqui\'e which allows to prove perfectness of a bimodule.
Let $\varphi \in (\cD_1\otimes\cD^\opp_2)\dgm$ be a DG-bimodule. Consider the derived tensor product functor
\begin{equation}\label{lphi}
L\varphi:\bD(\cD_2) \to \bD(\cD_1),
\qquad
L\varphi(-) = -\lotimes_{\cD_2}\varphi.
\end{equation}
This is a triangulated functor commuting with arbitrary direct sums.

\begin{prop}[\protect{\cite[Lemma 2.8.2]{TV}}]\label{tv}
If DG-category $\cD_2$ is smooth and the derived tensor product functor $L\varphi:\bD(\cD_2) \to  \bD(\cD_1)$
preserves compactness then $\varphi$ is a perfect DG-bimodule.
\end{prop}

\subsection{Drinfeld quotient}\label{ss-dq}

Let $\cD$ be a DG-category and $\cD_0 \subset \cD$ its small full DG-subcategory.
In~\cite{D04} Drinfeld defined a new DG-category $\cD/\cD_0$, which is known as {\sf Drinfeld quotient}.
By definition $\cD/\cD_0$ has the same objects as $\cD$ and $\Hom_{\cD/\cD_0}$ is freely generated
over $\Hom_\cD$ by generators $\varepsilon_X$, one for each object $X$ of $\cD_0$, such that
\begin{equation*}
\deg \varepsilon_X = -1,
\qquad
d(\varepsilon_X) = 1_X.
\end{equation*}
Thus $\Hom_{\cD/\cD_0}(Y,Y')$ is given by
\begin{equation}\label{hom-dq}
\bigoplus_{p=0}^\infty
\bigoplus_{X_1,\dots,X_p \in \cD_0} 
\Hom_\cD(Y,X_1)\otimes \varepsilon_{X_1}\otimes 
\Hom_\cD(X_1,X_2)\otimes \varepsilon_{X_2}\otimes \dots \otimes
\varepsilon_{X_p}\otimes 
\Hom_\cD(X_p,Y').
\end{equation}
The Drinfeld quotient $\cD/\cD_0$ comes with a natural DG-functor $\eta_{\cD,\cD_0}:\cD \to \cD/\cD_0$, which takes
$\Hom_\cD(Y,Y')$ to the $p=0$ summand above.
The following nice property of the Drinfeld quotient follows from the definition

\begin{prop}\label{dq-dgf}
Let $\cD$, $\cD'$ be DG-categories and $\cD_0 \subset \cD$, $\cD'_0 \subset \cD'$ their DG-subcategories.
If $F:\cD \to \cD'$ is a DG-functor such that $F(\cD_0) \subset \cD'_0$ then there is a DG-functor 
$\bar F:\cD/\cD_0 \to \cD'/\cD'_0$ such that the diagram
\begin{equation*}
\xymatrix{
\cD \ar[r]^-F \ar[d]_{\eta_{\cD,\cD_0}} & \cD' \ar[d]^{\eta_{\cD',\cD'_0}} \\
\cD/\cD_0 \ar[r]^-{\bar F} & \cD'/\cD'_0
}
\end{equation*}
commutes, i.e.
$\bar F\circ\eta_{\cD,\cD_0} = \eta_{\cD',\cD'_0}\circ F$.
\end{prop}
\begin{proof}
We define $\bar F$ to be the same as $F$ on objects, and to extend it on morphisms we define
$\bar{F}(\varepsilon_X) = \varepsilon_{F(X)}$. By definition of $\cD/\cD_0$ this uniquely defines $\bar F$.
\end{proof}

The definition of the Drinfeld quotient $\cD/\cD_0$ as presented above makes sense only in case when $\cD_0$ is a small DG-category.
Otherwise, the $\Hom$-spaces defined by~\eqref{hom-dq} are not sets. However, one can use the machinery of Grothendieck universes
to define the Drinfeld quotient for arbitrary $\cD_0$. We skip the details and refer the interested reader to the Appendix in~\cite{LO}.
Note by the way that the Drinfeld quotient provides a natural enhancement for the Verdier quotient.

\begin{theo}[\cite{D04}]\label{dq}
If $\cD$ is a pretriangulated DG-category and $\cD_0 \subset \cD$ is its pretriangulated DG-subcategory then 
$\cD/\cD_0$ is also pretriangulated and there is an equivalence of triangulated categories $[\cD/\cD_0] \cong [\cD]/[\cD_0]$.
\end{theo}

\subsection{Extensions of DG-functors}

In this section we discuss the relation between DG-functors on small DG-categories and functors
on their derived categories.

Let $F:\cD_1 \to \cD_2$ be a DG-functor between small DG-categories. 
For each right DG-module $M$ over $\cD_2$ we have a DG-module $M_F(X) = M(F(X))$ over $\cD_1$.
This defines a DG-functor of {\sf restriction} $\Res_F:\cD_2\dgm \to \cD_1\dgm$ which evidently takes acyclic DG-modules 
to acyclic DG-modules and so descends to a functor between derived categories which we also denote 
\begin{equation*}
\Res_F:\bD(\cD_2) \to \bD(\cD_1),
\qquad
M \mapsto M_F.
\end{equation*}
On the other hand, the DG-functor $F$ produces a $\cD_1$-$\cD_2$-bimodule ${}_F\cD_2$, 
so one can define the {\sf induction} functor $\Ind_F:\cD_1\dgm \to \cD_2\dgm$, $N \mapsto N\otimes_{\cD_1}{}_F\cD_2$ as well as 
its derived functor 
\begin{equation*}
\LInd_F:\bD(\cD_1) \to \bD(\cD_2),
\qquad
N \mapsto N\lotimes_{\cD_1}{}_F\cD_2.
\end{equation*}

\begin{prop}\label{fpb}
The derived induction functor $\LInd_F$ is left adjoint to the restriction functor $\Res_F$.
Both functors $\Res_F$ and $\LInd_F$ commute with arbitrary direct sums.
Moreover, 
\begin{equation*}
\LInd_F(\sY^X) \cong \sY^{F(X)}.
\end{equation*}
If $[F]$ is fully faithful then $\LInd_F$ is fully faithful.
Finally, if $F$ is a quasiequivalence then both $\Res_F$ and $\LInd_F$ are equivalences.
%
\end{prop}
\begin{proof}
First of all, note that $\Res_F$ commutes with arbitrary direct sums by definition and $\LInd_F$ commutes since a tensor product does.
Further, the formula for $\LInd_F(\sY^X)$ follows immediately from~\eqref{yoneltens}. 

To prove the adjunction note that the right adjoint of $\LInd_F$ exists by Brown representability. Let us denote it temporarily by $G$.
Note that for any $\cD_2$-module $M$ we have
\begin{multline*}
G(M)(X) \cong 
\Hom_{\bD(\cD_1)}(\sY^X,G(M)) \cong 
\Hom_{\bD(\cD_2)}(\LInd(\sY^X),M) \cong \\ \cong
\Hom_{\bD(\cD_2)}(\sY^{F(X)},M) \cong
M(F(X)) \cong M_F(X)
\end{multline*}
which shows that $G$ is isomorphic to the restriction functor $\Res_F$.

Finally, to show that $\LInd_F$ is fully faithful we use Lemma~\ref{cff}
with $\cT = \bD(\cD_1)$, $\cT' = \bD(\cD_2)$ and $\cT_0 = [\cD]$, and the same Lemma 
shows that $\LInd_F$ is an equivalence if $F$ is a quasiequivalence.
\end{proof}

In fact, the functors $(\Ind_F,\Res_F)$ form what is called a DG-adjoint pair.

\subsection{Derived category of quasicoherent sheaves}\label{ss-dqs}

Let $S$ be a separated scheme of finite type. Denote by $\Qcoh(S)$ the abelian category
of quasicoherent sheaves on $S$. The derived category $\bD(S)$ is defined as the Verdier quotient
\begin{equation*}
\bD(S) := [\com(S)]/[\coma(S)],
\end{equation*}
where $\com(S)$ is the DG-category of complexes over $\Qcoh(S)$ and $\coma(S) \subset \com(S)$
is the DG-category of acyclic complexes. However for our purposes another description 
of the derived category is more convenient.

Recall that a complex of quasicoherent sheaves $F$ is {\sf h-flat} if for any acyclic complex $A$
the complex $\Tot^\oplus(F\otimes_{\cO_S} A)$ (the direct sum totalization of the bicomplex
$F\otimes_{\cO_S} A$) is acyclic. By~\cite[Prop.~1.1]{AJL} there is enough  h-flat complexes in $\com(S)$ 
(that is each complex is quasiisomorphic to an h-flat complex) hence there is an equivalence
\begin{equation*}
\bD(S) \cong [\hflat(S)]/[\hflata(S)],
\end{equation*}
where $\hflat(S)$ is the category of h-flat complexes and $\hflata(S)\subset \hflat(S)$ is the category
of acyclic h-flat complexes. Using Theorem~\ref{dq} this can be rewritten as the homotopy category
of the Drinfeld quotient
\begin{equation}\label{dsdq}
\bD(S) \cong [\hflat(S)/\hflata(S)].
\end{equation}
This definition is especially useful when one is interested in the derived pullback and tensor product functors
because of the following observation of Spaltenstein.

\begin{lemma}[\cite{Sp88}]\label{pbfa}
For any morphism of schemes $f:T \to S$ the termwise pullback functor $f^*:\com(S) \to \com(T)$
takes h-flat complexes to h-flat complexes and h-flat acyclic complexes to
h-flat acyclic complexes.
The tensor product of an h-flat acyclic complex with any complex is acyclic.
\end{lemma}

This Lemma shows that for a morphism $f:T \to S$ the pullback functor $f^*$ induces a DG-functor
\begin{equation*}
f^*:\hflat(S) \to \hflat(T)
\qquad\text{such that}\qquad
f^*(\hflata(S)) \subset \hflata(T).
\end{equation*}
Consequently, by Proposition~\ref{dq-dgf} it induces a DG-functor of Drinfeld quotients
\begin{equation*}
f^*:\hflat(S)/\hflata(S) \to \hflat(T)/\hflata(T).
\end{equation*}
The induced functor on the homotopy categories is the derived pullback functor
\begin{equation*}
Lf^*:\bD(S) \to \bD(T).
\end{equation*}
Analogously one defines the derived tensor product functor.

%
%
%
%
%
%
%
%
%
%
%
%
%

Note that by construction the derived pullback functor commutes with direct sums. Therefore, by Brown Representability
(Theorem~\ref{brown}) it has a right adjoint functor
\begin{equation*}
Rf_*:\bD(T) \to \bD(S).
\end{equation*}

Let $\Sch$ be the category of separated schemes of finite type over $\kk$ and 
$\Tria$ be the 2-category of $\kk$-linear triangulated categories.
Associating with a scheme $S$ its derived category $\bD(S)$ and with each morphism
of schemes $f:T \to S$ its derived pullback functor $Lf^*:\bD(S) \to \bD(T)$ defines
a pseudofunctor
\begin{equation*}
\bD:\Sch^\opp \to \Tria,
\qquad
S \mapsto \bD(S),\quad
f:(T\to S) \mapsto (Lf^*:\bD(S) \to \bD(T))
\end{equation*}
which we will call the {\sf derived category pseudofunctor}.
In the next section we will show that it factors through the 2-category of small DG-categories.


\subsection{DG-enhancements for schemes}\label{ss-dge}

Let $\sDG$ denote the 2-category of small DG-categories over $\kk$.
Associating with a small DG-category $\cD$ its derived category $\bD(\cD)$
and with a DG-functor $F:\cD_1 \to \cD_2$ its derived induction functor $\LInd_F:\bD(\cD_1) \to \bD(\cD_2)$ 
gives a pseudofunctor $\bD:\sDG \to \Tria$.
The main result of this section is the following

\begin{theo}\label{cd-pf}
There is a pseudofunctor $\cD:\Sch^\opp \to \sDG$ such that the diagram
\begin{equation*}
\xymatrix{
\Sch^\opp \ar[rr]^\cD \ar[dr]_{\bD} && \sDG \ar[dl]^{\bD} \\ & \Tria
}
\end{equation*}
is commutative.
\end{theo}
\begin{proof}
We want to associate with any scheme $S$ a small DG-category $\cD(S)$ and with any morphism of schemes $f:T \to S$
a DG-functor $f^*:\cD(S) \to \cD(T)$. Let $\hflatperf(S)$ be the full DG-subcategory of $\hflat(S)$ consisting
of h-flat perfect complexes. Then $\hflatperf(S)/\hflata(S)$ is a full DG-subcategory in $\hflat(S)/\hflata(S)$.
Note that its homotopy category is the subcategory  $\bD^\perf(S) \subset \bD(S)$ of perfect complexes. 
This category is essentially small. Choose a small DG-subcategory
$\cD(S) \subset \hflatperf(S)/\hflata(S)$ such that
\begin{equation}\label{hds}
[\cD(S)] = [\hflatperf(S)/\hflata(S)] = \bD^\perf(S).
\end{equation}
and such that for any morphism of schemes $f:T \to S$ we have 
\begin{equation}\label{fsds}
f^*(\cD(S)) \subset \cD(T). 
\end{equation} 
For this we can first take $\cD_0(S)$ to be just an arbitrary choice of a small DG-subcategory 
$\cD_0(S) \subset \hflatperf(S)/\hflata(S)$ for which~\eqref{hds} holds
and then replace it by
\begin{equation*}
\cD(S) := \bigcup_{g:S \to S'} g^*(\cD_0(S')).
\end{equation*}
Here the union is taken with respect to the set of all isomorphism classes of morphisms of separated schemes
of finite type (two morphisms $g':S \to S'$ and $g'':S \to S''$ are isomorphic if there is an isomorphism 
$s:S' \to S''$ such that $g'' = s\circ g'$). This is indeed a set since the isomorphism classes of separated
schemes of finite type form a set and for a given scheme $S'$ all morphisms $S \to S'$ also form a set
(a morphism is determined by its graph).

Then
\begin{equation*}
f^*(\cD(S)) = 
f^*\left(\bigcup_{g:S \to S'} g^*(\cD_0(S'))\right) =
\bigcup_{g:S \to S'} f^*(g^*(\cD_0(S'))) \subset \cD(T)
\end{equation*}
since $f^*(g^*(\cD_0(S')) = (f\circ g)^*(\cD_0(S')) \subset \cD(T)$.

Now let $f:T \to S$ be a morphism of separated schemes of finite type. Restricting the DG-functor
$f^*:\hflat(S)/\hflata(S) \to \hflat(T)/\hflata(T)$ defined in the previous section 
to $\cD(S)$
and using compatibility~\eqref{fsds}
we obtain a DG-functor
\begin{equation*}
f^*:\cD(S) \to \cD(T).
\end{equation*}
This defines the pseudofunctor on morphisms.

It remains to prove that the diagram of the Theorem commutes.
In other words, we have to check that for each scheme $S$ we have an equivalence of categories
\begin{equation*}
\Psi_S:\bD(\cD(S)) \xrightarrow{\ \cong\ } \bD(S) 
\end{equation*}
and that the diagram
\begin{equation*}
\xymatrix{
\bD(S) \ar[rr]^-{\Psi_S} \ar[d]_{Lf^*} && \bD(\cD(S)) \ar[d]^{\LInd_{f^*}} \\ \bD(T) \ar[rr]^-{\Psi_T} && \bD(\cD(T))
}
\end{equation*}
commutes.

Let us start with the equivalence. Recall that $\bD(S) = [\hflat(S)/\hflata(S)]$ 
and $\cD(S)$ is a small DG-subcategory in $\hflat(S)/\hflata(S)$. Define a DG-functor
\begin{equation*}
\hflat(S)/\hflata(S) \to \cD(S)\dgm,
\qquad
M \mapsto \Hom_{\hflat(S)/\hflata(S)}(-,M).
\end{equation*}
Consider the composition
\begin{equation*}
\Psi_S:\bD(S) = [\hflat(S)/\hflata(S)] \to [\cD(S)\dgm] \to \bD(\cD(S)).
\end{equation*}
This functor commutes with arbitrary direct sums since all objects in $\cD(S)$ 
correspond to compact objects in $\bD(S)$. 
Moreover, by definition $\Psi_S(M) \cong \sY^M \in \cD(S)\dgm$ 
for any perfect complex $M$, hence $\Psi_S$ preserves compactness
and $\Psi_S(\cD(S))$ generates $\bD(\cD(S))$.
Finally, for all $M,N \in \cD(S)$ we have
\begin{multline*}
\Hom_{\bD(\cD(S))}(\Psi_S(M),\Psi_S(N)) \cong
\Hom_{\bD(\cD(S))}(\sY^M,\sY^N) \cong \\ \cong
\Hom_{[\cD(S)]}(M,N) \cong
\Hom_{\bD(S)}(M,N),
\end{multline*}
hence $\Psi_S$ is fully faithful on $[\cD(S)]$.
Applying Lemma~\ref{cff} we conclude that $\Psi_S$ is an equivalence.
%
%
%
%

It remains to check that the diagram commutes. Let us first construct a morphism of functors $\LInd_{f^*}\circ \Psi_S \to \Psi_T\circ Lf^*$.
The right adjoint of the derived induction functor is the restriction functor, so it suffices to construct
a morphism of functors $\Psi_S \to \Res_{f^*}\circ \Psi_T\circ Lf^*$. Both parts are induced by DG-functors 
$\hflat(S)/\hflata(S) \to \cD(S)\dgm$, the first by
\begin{equation*}
M \mapsto \Hom_{\hflat(S)/\hflata(S)}(-,M)
\end{equation*}
and the second by
\begin{equation*}
M \mapsto \Hom_{\hflat(T)/\hflata(T)}(f^*(-),f^*(M)).
\end{equation*}
The DG-functor $f^*$ gives a morphism from the first to the second, which on passing
to homotopy categories induces the required morphism of functors. Let us check
that it is an isomorphism.

Let $\cT \subset \bD(S)$ be the full subcategory of $\bD(S)$ formed by objects $M$ such that
the morphism $\LInd_{f^*}(\Psi_S(M)) \to \Psi_T(Lf^*(M))$ is an isomorphism. Clearly, $T$ is 
a triangulated subcategory. Moreover, it contains all objects from $[\cD(S)] \subset \bD(S)$.
Indeed, if $M \in \cD(S)$, then
\begin{equation*}
\LInd_{f^*}(\Psi_S(M)) \cong \LInd_{f^*}(\sY^M) \cong \sY^{f^*(M)} \cong \Psi_T(f^*(M)) \cong \Psi_T(Lf^*(M)),
\end{equation*}
so $M \in \cT$. Finally, $\cT$ is closed under arbitrary direct sums, since all the functors $\Psi_S$, $\Psi_T$,
$Lf^*$ and $\LInd_{f^*}$ commute with those. By Proposition~\ref{cgenprop} we conclude that $\cT = \bD(S)$, 
so the diagram is commutative.
\end{proof}

%
%
%

Note that for a separated scheme $S$ of finite type over a field the enhancement $\cD(S)$ 
of the category $\bD^\perf(S)$ of perfect complexes on $S$ is smooth (see Definition~\ref{dgsmooth})
if $S$ is smooth and proper 
if $S$ is proper by~\cite[Lemma~3.27]{TV}.

\subsection{From a DG-resolution to a categorical resolution}

Recall the definition~\ref{pcdgr} of a partial DG-resolution.

\begin{lemma}\label{compres}
If $\pi:\cD \to \cD'$ and $\pi':\cD'\to \cD''$ are partial DG-resolutions then the composition
$\pi'\circ\pi:\cD\to\cD''$ is a partial DG-resolution.
Moreover, if in addition $\pi'$ is a DG-resolution then so is $\pi'\circ\pi$.
\end{lemma}
\begin{proof}
Evidently follows from the definition.
\end{proof}


If for a scheme $Y$ a DG-resolution of the DG-category $\cdcf(Y)$ is given we can construct 
a categorical resolution of $Y$. For this we use Proposition~\ref{fpb} to induce the functors
on the derived categories.

\begin{prop}\label{dgcrcr}
Let $\cD$ be a small pretriangulated DG-category and $\pi:\cdcf(Y) \to \cD$ a DG-resolution.
Let $\pi^* = \LInd_\pi:\bD(Y) = \bD(\cdcf(Y)) \to \bD(\cD)$ be the derived induction
and let $\pi_* = \Res_\pi:\bD(\cD) \to \bD(Y)$ be the restriction functors. 
If $\pi_*([\cD]) \subset \bD^b(\coh(Y))$
%
%
then $\bD(\cD)$ is a categorical resolution of $\bD^b(\coh(Y))$.
\end{prop}
%
%
\begin{proof}
By Proposition~\ref{fpb} we know that $\pi^*$ is fully faithful, commutes with direct sums and has a right adjoint functor $\pi_*$
also commuting with direct sums. So the first two conditions of Definition~\ref{resboth} are satisfied. Since $\bD(\cD)^\comp = [\cD]$
the third condition is satisfied as well.
%
%
%
\end{proof}

\begin{remark}
If $Y$ is projective and $\cD$ is proper then one can check that the condition $\pi_*([\cD]) \subset \bD^b(\coh(Y))$ holds automatically.
Indeed, in this case the subcategory $\bD^b(\coh(Y)) \subset \bD(Y)$ consists of all objects $F$ such that
$\oplus_i \Hom(P,F[i])$ is finite dimensional for all perfect $P$. But by adjunction
\begin{equation*}
\oplus_i \Hom(P,\pi_*(G)[i]) \cong
\oplus_i \Hom(\pi^*(P),G[i]) 
\end{equation*}
is finite dimensional for all $G \in \cD$ since $\pi^*(P) \cong \pi(P)$ is representable and DG-category $\cD$ is proper.
\end{remark}

\section{The gluing}\label{s-glu}

The notion of the gluing of two DG-categories is well-known to experts.
One possible definition may be found in~\cite{Ta} under the name of upper triangular DG-categories. 
We prefer to use slightly different definition as it is more adjusted to our goals. Below we will discuss
its relation to the definition of Tabuada.


\subsection{The construction}


Let $\cD_1$ and $\cD_2$ be two small DG-categories. Consider a bimodule $\varphi \in (\cD_2^\op\otimes\cD_1)\dgm$.
The DG-category $\cD_1\times_\varphi\cD_2$ called {\sf the gluing of $\cD_1$ with $\cD_2$ along $\varphi$}, 
is defined as follows.



\begin{itemize}
\item The {\em objects} of the DG-category $\cD_1\times_\varphi\cD_2$ are triples
$M = (M_1,M_2,\mu)$, where $M_i \in \cD_i$ and $\mu \in \varphi(M_2,M_1)$ 
is a closed element of degree $0$ (recall that being a bimodule $\varphi$ 
associates a complex $\varphi(M_2,M_1)$ to any pair $(M_1,M_2)$ of objects of $\cD_1$ and $\cD_2$).
\item The {\em morphism complexes} are defined to be the sums of
\begin{equation}\label{hglu}
\Hom^k_{\cD_1\times_\varphi\cD_2}(M,N) = 
\Hom^k_{\cD_1}(M_1,N_1) \oplus \Hom^k_{\cD_2}(M_2,N_2) \oplus \varphi^{k-1}(N_2,M_1).
\end{equation} 
with the differentials given by
\begin{equation}\label{dglu}
d(f_1,f_2,f_{21}) = (d(f_1),d(f_2),-d (f_{21}) - f_2\circ\mu + \nu\circ f_1).
\end{equation} 
where $M = (M_1,M_2,\mu)$ and $N = (N_1,N_2,\nu)$.
\item The {\em multiplication} is defined by
\begin{equation}\label{mglu}
(f_1,f_2,f_{21})\circ (g_1,g_2,g_{21}) = (f_1\circ g_1,f_2\circ g_2,f_{21}\circ g_1 + (-1)^{\deg f_2}f_2\circ g_{21}).
\end{equation} 
where $f \in \Hom_{\cD_1\times_\varphi\cD_2}(M,N)$ and $g \in \Hom_{\cD_1\times_\varphi\cD_2}(L,M)$.
In particular, the identity morphism of $(M_1,M_2,\mu)$ is given by
\begin{equation*}
\id_{(M_1,M_2,\mu)} = (\id_{M_1},\id_{M_2},0).
\end{equation*}
\end{itemize}

It is an exercise left to the reader to check that this is a DG-category.

\begin{remark}\label{triglu}
Another way to understand the definition of the gluing category is by saying that
there is a distinguished triangle
\begin{equation*}
\Hom_{\cD_1\times_\varphi\cD_2}(M,N) \to \Hom_{\cD_1}(M_1,N_1) \oplus \Hom_{\cD_2}(M_2,N_2) \xrightarrow{\ (-\nu,\mu)\ } \varphi(N_2,M_1)
\end{equation*}
of complexes of vector spaces (the morphism is given by $(f_1,f_2) \mapsto f_2\circ\mu - \nu\circ f_1$).
\end{remark}

\begin{remark}
The upper triangular DG-category corresponding to DG-categories $\cD_1$ and $\cD_2$ and a bimodule $\varphi$ 
defined in~\cite{Ta} by Tabuada is DG-equivalent to the full DG-subcategory $\cD_1\sqcup_\varphi\cD_2 \subset \cD_1\times_{\varphi[1]}\cD_2$ 
with objects of the form $(M_1,0,0)$ and $(0,M_2,0)$ only. 
One can show that the derived category of $\cD_1\sqcup_\varphi\cD_2$ is equivalent to the derived category of $\cD_1\times_\varphi\cD_2$,
and moreover, even pretriangulated envelopes of these DG-categories coincide. The advantage of the DG-category
$\cD_1\sqcup_\varphi\cD_2$ is that it is smaller and its definition is simpler.
The advantage of $\cD_1\times_\varphi\cD_2$ is that it is more close to a pretriangulated category
as is shown in the following
\end{remark}

\begin{lemma}\label{dpt}
Assume that the categories $\cD_1$ and $\cD_2$ are pretriangulated. 
Then so is the gluing category $\cD_1\times_\varphi\cD_2$.
\end{lemma}
\begin{proof}
We must check that the shift of an object
and the cone of a morphism are representable in $\cD_1\times_\varphi\cD_2$. 
For the shift it is evident --- it is clear that the shift of the object $(M_1,M_2,\mu)$ is represented 
by the object $(M_1[1],M_2[1],\mu)$. To check representability of the cone pick a closed morphism 
of degree zero from $(M_1,M_2,\mu)$ to $(N_1,N_2,\nu)$. By definition it is given by a pair of closed degree zero
morphisms  $f_1:M_1 \to N_1$ and $f_2:M_2 \to N_2$ and an element $f_{21} \in \varphi(M_2,N_1)$ of degree $-1$ such that
$$
d(f_{21}) = \nu \circ f_1 - f_2 \circ \mu 
$$
(thus $f_{21}$ is a homotopy between $f_2 \circ \mu$ and $\nu \circ f_1$).
Let $C_1$ be the cone of $f_1$ and let $C_2$ be the cone of $f_2$. 
As it was observed in Remark~\ref{cone} the cones 
come with degree zero morphisms 
$$
M_k[1] \xrightarrow{i_k} C_k \xrightarrow{p_k} M_k[1],
\qquad
N_k \xrightarrow{j_k} C_k \xrightarrow{s_k} N_k.
$$
which give decompositions $C_k = M_k[1] \oplus N_k$ in the graded categories associated with $\cD_k$.
Moreover, we have
$$
d(j_k) = d(p_k) = 0,
\qquad
d(i_k) = j_kf_k,
\qquad
d(s_k) = -f_kp_k.
$$
Now we consider the element
$$
\gamma = i_2\mu p_1 + j_2\nu s_1 + j_2 f_{21} p_1 \in \varphi(C_2,C_1).
$$
The triple $(C_1,C_2,\gamma)$ is then an object of $\cD_1\times_\varphi\cD_2$. Indeed,
\begin{multline*}
d\gamma = 
(di_2)\mu p_1 +j_2\nu (ds_1) + j_2 (df_{21}) p_1 = \\ =
j_2 f_2 \mu p_1 - j_2\nu f_1 p_1 + j_2 (df_{21}) p_1 =
j_2( f_2\mu - \nu f_1 +df_{21}) p_1 = 0
\end{multline*}
so $\gamma$ is a closed element of degree $0$. Using again Remark~\ref{cone} it is straightforward to check that 
$(C_1,C_2,\gamma)$ is the cone of the morphism we have started with --- indeed, one can take $i = (i_1,i_2,0)$,
$p = (p_1,p_2,0)$, $j = (j_1,j_2,0)$ and $s = (s_1,s_2,0)$ and check that all the required relations hold.
%
%
\end{proof}

\subsection{Semiorthogonal decomposition}

For brevity we denote by $\cD$ the gluing category $\cD_1\times_\varphi\cD_2$.
Our next goal is to describe the relation of DG-category $\cD$ with the original
categories $\cD_1$ and $\cD_2$. We start by introducing some natural DG-functors:
\begin{equation}\label{i12}
\begin{array}{ll}
i_1:\cD_1 \to \cD,\qquad & M_1 \mapsto (M_1,0,0) \\
i_2:\cD_2 \to \cD, & M_2 \mapsto (0,M_2,0).
\end{array}
\end{equation} 
It turns out that these functors have adjoints on DG-level.

\begin{lemma}\label{iadj}
$(i)$
The left and the right adjoints of $i_1$ and $i_2$ respectively are given by
\begin{equation}\label{iadj12}
\begin{array}{ll}
i_1^*:\cD \to \cD_1,\qquad & (M_1,M_2,\mu) \mapsto M_1,\\
i_2^!:\cD \to \cD_2,\qquad & (M_1,M_2,\mu) \mapsto M_2.
\end{array}
\end{equation} 


\noindent$(ii)$
Assume that $\cD_1$ is pretriangulated. 
The right adjoint of $i_1$ is given by
\begin{equation}\label{iradj}
i_1^!:\cD \to \cD_1\dgm,\qquad  (N_1,N_2,\nu) \mapsto \Cone(\nu)[-1].
\end{equation} 
\end{lemma}

Here we use a natural identification $\varphi(N_2,N_1) = \Hom_{\cD_1\dgm}(\sY^{N_1},\varphi(N_2,-))$, 
so $\nu$ is considered as a closed morphism
$\sY^{N_1} \to \varphi(N_2,-)$ and so we can speak about its cone.

\begin{proof}
$(i)$ Indeed, recalling the definition of $\cD$ we see that
$$
\Hom_\cD((N_1,N_2,\nu),i_1(M_1)) =
\Hom_\cD((N_1,N_2,\nu),(M_1,0,0)) = 
\Hom_{\cD_1}(N_1,M_1),
$$
hence $i_1^*(N_1,N_2,\nu) = N_1$.
Analogously, 
$$
\Hom_\cD(i_2(M_2),(N_1,N_2,\nu)) =
\Hom_\cD((0,M_2,0),(N_1,N_2,\nu)) = 
\Hom_{\cD_2}(M_2,N_2),
$$
hence $i_2^!(N_1,N_2,\nu) = N_2$. 

%

$(ii)$
We have
\begin{multline*}
\Hom_\cD(i_1(M_1),(N_1,N_2,\nu)) =
\Hom_\cD((M_1,0,0),(N_1,N_2,\nu)) = \\ =
\Hom_{\cD_1}(M_1,N_1) \oplus \varphi(N_2,M_1)[-1]. 
\end{multline*}
If we think of $\varphi(N_2,M_1)$ as of the $\Hom$-complex from $\sY^{M_1}$ to $\varphi(N_2,-)$
in the category of $\cD_1$-modules, then the differential of the RHS will match with that
of the $\Hom$-complex from $\sY^{M_1}$ to the shifted by $-1$ cone of $\nu$ considered as a morphism from $\sY^{N_1}$ to $\varphi(N_2,-)$.
This shows that $i_1^!(N) = \Cone(\nu)[-1]$. 
%
\end{proof}

Using formulas~\eqref{i12}, \eqref{iadj12}, and~\eqref{iradj} one easily computes
\begin{equation}\label{i1i2}
i_1^*\circ i_1 = \id_{\cD_1},
\qquad
i_2^!\circ i_2 = \id_{\cD_2},
\qquad
i_1^*\circ i_2 = 0,
\qquad
i_1^!\circ i_2 = \varphi[-1].
\end{equation}

%

Recall that the last equation agrees with the definition of the gluing functor
for semiorthogonal decompositions discussed in section~\ref{ss-gf}.
In fact we have the following

\begin{cor}\label{sodglu}
Assume $\cD_1$ and $\cD_2$ are small pretriangulated categories {\rm(}hence so is $\cD${\rm)}.
The functors $i_1:[\cD_1] \to [\cD]$ and $i_2:[\cD_2] \to [\cD]$ are fully faithful and give a semiorthogonal
decomposition $[\cD] = \langle [\cD_1],[\cD_2] \rangle$ with the gluing bifunctor induced by $\varphi$. 
\end{cor}
\begin{proof}
We denote $\cT_i = [\cD_i]$, $\cT = [\cD]$ for short.
The DG-functors $i_1$, $i_2$, $i_1^*$, and $i_2^!$ descend to triangulated functors between the homotopy categories,
moreover the adjunctions are preserved. So, the first two equations of~\eqref{i1i2} prove
that $i_1:\cT_1 \to \cT$ and $i_2:\cT_2 \to \cT$ are fully faithful and have left and right adjoints respectively. 
The third equation shows that the essential images of $\cT_1$ and $\cT_2$ in $\cT$ are semiorthogonal. Moreover, 
the functor $i_1^!\circ i_2[1]:\cT_2 \to \cT_1$ is isomorphic to the functor $L\varphi$ 
induced by the bimodule $\varphi$ by the fourth equation. 
So, it remains to check that the subcategories
$\cT_1$ and $\cT_2$ generate $\cT$. For this note that by Lemma~\ref{sod3}
there is a semiorthogonal decomposition $\cT = \langle \cT_1,{}^\perp\cT_1 \cap \cT_2^\perp,\cT_2 \rangle$,
where 
$$
{}^\perp\cT_1 \cap \cT_2^\perp = \Ker i_1^*\cap \Ker i_2^!,
$$
so it remains to check that if both $i_1^*M = M_1$ and $i_2^!M = M_2$ are null-homotopic then $M$
is null-homotopic as well. Indeed, let $h_1 \in \Hom^{-1}_{\cD_1}(M_1,M_1)$ and $h_2 \in \Hom^{-1}_{\cD_2}(M_2,M_2)$
be the contracting homotopies of $\id_{M_1}$ and $\id_{M_2}$ respectively. In other words we assume that
$d_1(h_1) = \id_{M_1}$ and $d_2(h_2) = \id_{M_2}$. We want to extend it to a homotopy of $\id_{M}$.
This means that we have to find $h_{21} \in \varphi(M_2,M_1)$ such that
$$
d(h_{21}) = h_2\circ\mu - \mu\circ h_1.
$$
For this we first note that the RHS above is closed. Indeed, since $d\mu = 0$ we have
$$
d(h_2\circ\mu - \mu\circ h_1) = 
d(h_2)\circ\mu - \mu\circ d(h_1) =
\mu - \mu = 0.
$$
Let us take 
$$
h_{21} = (h_2\circ\mu - \mu\circ h_1) \circ h_1.
$$
Since the expression in parentheses is a closed morphism we have
$$
d(h_{21}) = 
(h_2\circ\mu - \mu\circ h_1) \circ d(h_1) = 
h_2\circ\mu - \mu\circ h_1
$$
and we are done.
\end{proof}

%

\subsection{Derived category}\label{ss-dglu}

One has also a semiorthogonal decomposition of the derived category of the gluing.
First note that by Proposition~\ref{fpb} the functors $i_1$ and $i_2$ defined by~\eqref{i12} 
extend to derived categories and give functors which we denote here by
\begin{equation*}
I_1 = \LInd_{i_1}:\bD(\cD_1) \to \bD(\cD),\quad
I_2 = \LInd_{i_2}:\bD(\cD_2) \to \bD(\cD),
\end{equation*}
and which have right adjoints denoted here by
\begin{equation*}
I_1^! = \Res_{i_1}:\bD(\cD) \to \bD(\cD_1),\quad
I_2^! = \Res_{i_2}:\bD(\cD) \to \bD(\cD_2),
\end{equation*}
induced by restriction of DG-modules. In other words,
\begin{equation}\label{Iadj}
I_1^!(F) = F_{i_1},
\qquad
I_2^!(F) = F_{i_2},
\end{equation} 
for any $F \in \cD\dgm$.

\begin{prop}\label{derglu}
Let $\cD = \cD_1\times_\varphi\cD_2$. Then one has 
a semiorthogonal decomposition
\begin{equation*}
\bD(\cD) = \left\langle \bD(\cD_1),\bD(\cD_2)\right\rangle
\end{equation*}
with the gluing functor isomorphic to $L\varphi:\bD(\cD_2) \to \bD(\cD_1)$.
Moreover, for any DG-module $F$ over $\cD$ there is a distinguished triangle
\begin{equation}\label{I1s}
F_{i_1} \to I_1^*F \to F_{i_2}\lotimes_{\cD_2}\varphi.
\end{equation} 
\end{prop}

Note that one does not need to assume that any of the categories is pretriangulated.

\begin{proof}
By Proposition~\ref{fpb} the functors $I_1$ and $I_2$ are fully faithful.
Moreover, 
\begin{equation*}
I_2^!I_1(F) = (F\lotimes_{\cD_1}{}_{i_1}\cD)_{i_2} = F\lotimes_{\cD_1}{}_{i_1}\cD_{i_2} = 0
\end{equation*}
since ${}_{i_1}\cD_{i_2} = 0$ by definition of gluing.
Applying two times Lemma~\ref{sod3} we conclude that there is a semiorthogonal decomposition $\bD(\cD) = \langle \cT,\bD(\cD_1),\bD(\cD_2) \rangle$,
where 
\begin{equation*}
\cT = \bD(\cD_1)^\perp \cap \bD(\cD_2)^\perp.
\end{equation*}
It remains to check that $\cT = 0$. For this we note that for any object $(M_1,M_2,\mu)$
of $\cD$ one can consider $\mu \in \varphi(M_2,M_1)$ as a morphism from $\sY^{(M_1,0,0)}$ to $\sY^{(0,M_2,0)}[1]$
in $\bD(\cD)$, and the cone of this morphism is a representable $\cD$-module, represented precisely
by the initial object $(M_1,M_2,\mu)$. In other words, in $\bD(\cD)$ there is a canonical distinguished
triangle
\begin{equation}\label{glutri}
\sY^{(0,M_2,0)} \xrightarrow{\quad} \sY^{(M_1,M_2,\mu)} \xrightarrow{\quad} \sY^{(M_1,0,0)} \xrightarrow{\ \mu\ } \sY^{(0,M_2,0)}[1].
\end{equation} 
In particular, it follows that all representable $\cD$-modules are contained in the subcategory of $\bD(\cD)$
generated by $\bD(\cD_1)$ and $\bD(\cD_2)$. Thus $\cT$ is orthogonal to all representable $\cD$-modules,
hence $\cT = 0$.

To establish the formula for the gluing functor $\phi:\bD(\cD_2) \to \bD(\cD_1)$ we
recall that as it was mentioned in section~\ref{ss-gf} we have $\phi = I_1^!I_2[1]$. 
%
By definition of $I_1^!$ and $I_2$ we have
\begin{equation*}
I_1^!(I_2(N)) = 
(N\lotimes_{\cD_2}{}_{i_2}\cD)_{i_1} = N\lotimes_{\cD_2} {}_{i_2}\cD_{i_1}
\end{equation*}
and it remains to note that by~\eqref{hglu} we have ${}_{i_2}\cD_{i_1} = \varphi[-1]$, hence 
$\phi(N) \cong N\lotimes_{\cD_2}\varphi$.

Finally, to check~\eqref{I1s} we first consider the standard triangle
\begin{equation*}
I_2I_2^!F \to F \to I_1I_1^*F.
\end{equation*}
Applying the functor $I_1^!$ we obtain 
\begin{equation*}
I_1^!I_2I_2^!F \to I_1^!F \to I_1^*F
\end{equation*}
since $I_1^!\circ I_1$ is the identity functor.
As we already checked that $I_1^!I_2$ is isomorphic to the derived tensor product by $\varphi[-1]$,
using~\eqref{Iadj} we deduce the claim.
%
%
%
%
%
\end{proof}

\subsection{Smoothness and properness}

It is useful to note that the quasiequivalence class of the gluing depends only on the quasiisomorphism
class of the gluing bimodule.

\begin{lemma}\label{qisqeq}
Assume that $\varphi$ and $\psi$ are quasiisomorphic $(\cD_2^\op\otimes\cD_1)$-modules.
Then the gluings $\cD_1\times_\varphi\cD_2$ and $\cD_1\times_\psi\cD_2$ are quasiequivalent.
\end{lemma}
\begin{proof}
First assume that $\xi : \varphi \to \psi$ is a closed degree $0$ morphism in the category 
of bimodules which is a quasiisomorphism. Then we define a DG-functor 
from $\cD_1\times_\varphi\cD_2$ to $\cD_1\times_\psi\cD_2$ by
$$
(M_1,M_2,\mu) \mapsto (M_1,M_2,\xi(\mu))
$$
on objects and by
$$
(f_1,f_2,f_{21}) \mapsto (f_1,f_2,\xi(f_{21}))
$$
on morphisms. It is very easy to check that it is a DG-functor which acts by quasiisomorphisms on $\Hom$ complexes.
To check that it is a quasiequivalence we have to show that it is homotopically essentially surjective.
In other words, for any $M = (M_1,M_2,\nu) \in \cD_1\times_\psi\cD_2$ we have to construct an object
in $\cD_1\times_\varphi\cD_2$ whose image is homotopy equivalent to $M$. For this we choose $\mu \in \varphi(M_2,M_1)$
such that $\xi(\mu)$ is homologous to $\nu$ and take $(M_1,M_2,\mu)$. So, we have to show that
$(M_1,M_2,\xi(\mu))$ is homotopic to $(M_1,M_2,\nu)$. Since $\xi(\mu)$ and $\nu$ are homologous,
we have $\nu - \xi(\mu) = d\rho$ for some $\rho \in \psi(M_2,M_1)[-1]$.  Therefore
$(\id_{M_1},\id_{M_2},\rho)$ gives a morphism from $(M_1,M_2,\xi(\mu))$ to $(M_1,M_2,\nu)$
and $(\id_{M_1},\id_{M_2},-\rho)$ gives a morphism in the opposite direction. The compositions
are the identity morphisms, so we are done.

Thus we have proved the result when the quasiisomorphism is given by a morphism. In general case
we can connect $\varphi$ to $\psi$ by a chain of quasiisomorphisms, each of them gives a quasiequivalence.
Thus the gluings are connected by a chain of quasiequivalences, so the categories are quasiequivalent.
\end{proof}

For many purposes it is convenient to have an h-projective gluing bimodule.
By the above Lemma we can always replace $\varphi$ with its h-projective resolution
without changing the quasiequivalence class of the gluing DG-category.

\begin{remark}
Further we will also show that replacing the categories $\cD_1$ and $\cD_2$ by quasiequivalent DG-categories
one does not change the quasiequivalence class of the gluing, see Proposition~\ref{regluequi}.
\end{remark}

Recall the notion of smoothness and properness of a DG-category, see section~\ref{ss-smpr}.


\begin{prop}\label{utsm}
Let $\cD_1$ and $\cD_2$ be smooth DG-categories. A gluing $\cD = \cD_1\times_\varphi\cD_2$
is smooth if the gluing bimodule $\varphi$ is perfect in $\bD(\cD_2^\op\otimes\cD_1)$.
If $\cD_1$ and $\cD_2$ are proper and $\varphi$ is perfect then $\cD$ is proper.
\end{prop}
\begin{proof}
Our goal is to show that the diagonal bimodule $\cD$ is perfect. Using the same argument
as in~\cite[Prop.~3.11]{L10a} we conclude that $\cD$ fits into a distinguished triangle
\begin{equation*}
\cD_1\lotimes_{\cD_1^\op\otimes\cD_1} (\cD^\op\otimes\cD) \oplus
\cD_2\lotimes_{\cD_2^\op\otimes\cD_2} (\cD^\op\otimes\cD) \to
\cD \to
\varphi\lotimes_{\cD_2^\op\otimes\cD_1} (\cD^\op\otimes\cD)
\end{equation*}
and observe that in the first term both summands are perfect since $\cD_i$ is a perfect
bimodule over $\cD_i$ by smoothness of $\cD_i$. Thus $\cD$ is perfect if the third term is 
perfect. It remains to note that perfectness of $\varphi$ in $\bD(\cD_2^\op\otimes\cD_1)$ 
implies perfectness of $\varphi\lotimes_{\cD_2^\op\otimes\cD_1} (\cD^\op\otimes\cD)$
in $\bD(\cD^\op\otimes\cD)$ by extension of scalars.

The properness of $\cD$ follows immediately from the definition.
\end{proof}


%
%
%

\subsection{The gluing and semiorthogonal decompositions}

Here we show that any category with a semiorthogonal decomposition is quasiequivalent to a gluing.

\begin{prop}\label{uniglu2}
Assume that $\cD$ is a pretriangulated category and a semiorthogonal decomposition
of its homotopy category is given 
$$
[\cD] = \langle \cT_1, \cT_2 \rangle.
$$
Then DG-category $\cD$ is quasiequivalent to a gluing of two pretriangulated DG-categories $\cD_1$ and $\cD_2$
such that $[\cD_1] = \cT_1$ and $[\cD_2] = \cT_2$.
\end{prop}
\begin{proof}
We take $\cD_k$ to be the full DG-subcategory of $\cD$ having the same objects as $\cT_k$ 
and take 
\begin{equation*}
\varphi := {}_{i_2}\cD_{i_1}[1].
\end{equation*}
Note that both DG-categories $\cD_1$ and $\cD_2$ are pretriangulated.
We construct
the DG-functor $\alpha:\cD_1\times_\varphi\cD_2 \to \cD\dgm$ as follows. On objects it acts as
$$
\alpha(M_1,M_2,\mu) = \Cone (\mu:M_1[-1] \to M_2).
$$
Here we observe that $\mu \in \varphi(M_2,M_1) = \Hom_\cD(M_1,M_2)[1] = \Hom_\cD(M_1[-1],M_2)$ by definition of $\varphi$.
Note that if $\mu:M_1[-1] \to M_2$ and $\nu:N_1[-1] \to N_2$ are two closed morphisms of degree zero
then the cones $C_M = \Cone (\mu:M_1[-1] \to M_2)$ and $C_N = \Cone (\nu:N_1[-1] \to N_2)$
come with direct sum decompositions 
\begin{equation*}
\xymatrix@1{M_2 \ar@<.5ex>[r]^{j_M} & C_M \ar@<.5ex>[l]^{s_M}  \ar@<.5ex>[r]^{p_M} & M_1  \ar@<.5ex>[l]^{i_M} }
\quad\text{and}\quad
\xymatrix@1{N_2 \ar@<.5ex>[r]^{j_N} & C_N \ar@<.5ex>[l]^{s_N}  \ar@<.5ex>[r]^{p_N} & N_1  \ar@<.5ex>[l]^{i_N} }
\end{equation*}
with
\begin{equation*}
\arraycolsep=.5cm
\begin{array}{ccc}
dj_M = dp_M = 0, & di_M = j_M\mu, & ds_M = -\mu p_M,\\
dj_N = dp_N = 0, & di_N = j_N\nu, & ds_N = -\nu p_N,
\end{array}
\end{equation*}
It follows that any $f  \in \Hom_\cD(C_M,C_N)$ can be written as
\begin{equation*}
f = (i_Np_N + j_Ns_N) f (i_Mp_M + j_Ms_M) = 
i_N f_{11} p_M +
i_N f_{12} s_M +
j_N f_{21} p_M +
j_N f_{22} s_M,
\end{equation*}
where
\begin{equation*}
\arraycolsep = .5cm
\begin{array}{cc}
f_{11} = p_N f i_M \in \Hom_\cD(M_1,N_1), &
f_{12} = p_N f j_M \in \Hom_\cD(M_2,N_1), \\
f_{21} = s_N f i_M \in \Hom_\cD(M_1,N_2), &
f_{22} = s_N f j_M \in \Hom_\cD(M_2,N_2).
\end{array}
\end{equation*}
Therefore
\begin{multline*}
df = d(i_N f_{11} p_M + i_N f_{12} s_M + j_N f_{21} p_M + j_N f_{22} s_M) = \\
i_N (d f_{11} - (-1)^{\deg f} f_{12} \mu) p_M +
i_N d f_{12} s_M + \\
j_N (d f_{21} + \nu f_{11} - (-1)^{\deg f} f_{22} \mu) p_M +
j_N (d f_{22} + \nu f_{12}) s_M.
\end{multline*}
This formula shows that the map
\begin{equation*}
\Hom_\cD(C_M,C_N) \to \Hom_\cD(M_2,N_1),
\qquad
f \mapsto f_{12}
\end{equation*}
is a morphism of complexes. Moreover, the complex $\Hom_\cD(M_2,N_1)$ is acyclic,
as the cohomology of $\Hom_\cD(M_2,N_1)$ is just $\Ext$ spaces from $M_2$ to $N_1$ which are zero
since $M_2 \in \cT_2$ and $N_1 \in \cT_1$ and those subcategories of $\cT$ are semiorthogonal.
It follows that the map
\begin{equation*}
\begin{array}{c}
\alpha:\Hom_{\cD_1}(M_1,N_1) \oplus \Hom_{\cD_2}(M_2,N_2) \oplus \Hom_\cD(M_1,N_2) \to \Hom_\cD(C_M,C_N)\\
(f_{11},f_{22},f_{21})  \mapsto  i_N f_{11} p_M + j_N f_{21} p_M + j_N f_{22} s_M,
\end{array}
\end{equation*}
is a quasiisomorphism. But since $\Hom_\cD(M_1,N_2) = \cD(N_2,M_1) = \varphi(N_2,M_1)[-1]$
(and so the differentials of $\cD$ and $\varphi$ differ by a sign), it follows that the LHS
literally coincides with $\Hom_{\cD_1\times\varphi\cD_2}((M_1,M_2,\mu),(N_1,N_2,\nu))$ and we can consider
the above map as a definition of the functor $\alpha$ on morphisms. One can easily check that
this is compatible with the compositions, and so correctly defines a DG-functor.
And since we just checked that the morphism on $\Hom$-complexes is a quasiisomorphism, 
the functor $\alpha$ is quasi fully faithful. 
%
%
%
%
%
%
%
%

To show it is a quasiequivalence we have to check that any object of $\cD$ is homotopic to an object
in the image of $\alpha$. Take any $M$ in $\cD$ and let 
$$
M_1[-1] \to M_2 \to M \to M_1 \
$$
be its decomposition in $\cT$ with respect to the original semiorthogonal decomposition.
Let $\mu:M_1[-1] \to M_2$ be the morphism from that triangle. 
Then $\alpha(M_1,M_2,\mu) = \Cone(\mu)$ which is isomorphic to $M$ in $\cT = [\cD]$. 
Thus in $\cD$ there is a homotopy equivalence of $M$ and $\alpha(M_1,M_2,\mu)$.
\end{proof}

\subsection{Regluing}\label{sreglu}

Recall the notion of a partial categorical DG-resolution~\ref{pcdgr}.
Assume that $\cD_1,\cD_2$ are small pretriangulated DG-categories and
$\varphi\in (\cD_2^\op \otimes \cD_1)\dgm$ is a DG-bimodule.
Assume also $\tau_1 :\cD_1 \to \tilde\cD_1$ and $\tau_2 :\cD_2 \to \tilde\cD_2$ are 
partial categorical DG-resolutions and $\tilde\varphi \in (\tilde\cD_2^\op \otimes \tilde\cD_1)\dgm$
is another DG-bimodule. We will say that the bimodules $\varphi$ and $\tilde\varphi$ are {\sf compatible} if
a quasiisomorphism
\begin{equation}\label{compphi}
c:\varphi \xrightarrow{\ \cong\ } {}_{\tau_2}\tilde\varphi_{\tau_1}
\end{equation} 
is given. Consider the gluings $\cD = \cD_1\times_\varphi\cD_2$ and $\tilde\cD = \tilde\cD_1\times_{\tilde\varphi}\tilde\cD_2$.
We denote by $I_1,I_2$ the embedding functors of the semiorthogonal decomposition of $\bD(\cD)$ and by $\tilde I_1,\tilde I_2$
the embedding functors of the semiorthogonal decomposition of $\bD(\tilde\cD)$.

\begin{prop}\label{compglu}
If the bimodules $\varphi$ and $\tilde\varphi$ are compatible then 
there is a DG-functor $\tau:\cD_1\times_\varphi\cD_2 \to \tilde\cD_1\times_{\tilde\varphi}\tilde\cD_2$
which is a partial categorical DG-resolution, such that
\begin{enumerate}\renewcommand{\theenumi}{\alph{enumi}}
\item $\Res\tau_1\circ L\tilde\varphi \circ \LInd_{\tau_2} \cong L\varphi$, an isomorphism of functors $\bD(\cD_2) \to \bD(\cD_1)$;
\item $\tilde I_k\circ \LInd_{\tau_k} \cong \LInd_\tau\circ I_k$ for $k = 1,2$, an isomorphism of functors $\bD(\cD_k) \to \bD(\tilde\cD)$;
\item $\Res_{\tau_k}\circ \tilde I_k^! \cong I_k^!\circ\Res_\tau$ for $k = 1,2$, an isomorphism of functors $\bD(\tilde\cD) \to \bD(\cD_k)$;
\item $\Res_\tau\circ \tilde I_1 \cong I_1\circ\Res_{\tau_1}$, an isomorphism of functors $\bD(\tilde\cD_1) \to \bD(\cD)$;
\item if the canonical morphism $L\varphi \circ \Res_{\tau_2} \to \Res\tau_1\circ L\tilde\varphi$ of functors 
is an isomorphism then we have $\Res_\tau\circ \tilde I_2 \cong I_2\circ\Res_{\tau_2}$, an isomorphism of functors $\bD(\tilde\cD_2) \to \bD(\cD)$.
\end{enumerate}
%
%
\end{prop}
\begin{proof}
First, we note that for any $(M_1,M_2,\mu) \in \cD_1\times_\varphi\cD_2$ we have
\begin{equation*}
c(\mu) \in {}_{\tau_2}\tilde\varphi_{\tau_1}(M_2,M_1) = \tilde\varphi(\tau_2(M_2),\tau_1(M_1)),
\end{equation*}
so $(\tau_1(M_1),\tau_2(M_2),c(\mu))$ is an object of the gluing $\tilde\cD_1\times_{\tilde\varphi}\tilde\cD_2$.
We define a DG-functor 
\begin{equation*}
\tau:\cD_1\times_\varphi\cD_2 \to \tilde\cD_1\times_{\tilde\varphi}\tilde\cD_2,
\qquad
(M_1,M_2,\mu) \mapsto (\tau_1(M_1),\tau_2(M_2),c(\mu)).
\end{equation*}
The fact that $\tau $ is a DG-functor is straightforward. Let us check that it is quasi fully faithful.
For this just note that $\Hom_{\tilde\cD}(\tau M,\tau N)$ equals to
\begin{equation*}
\Hom_{\tilde\cD_1}(\tau_1 M_1,\tau_1 N_1) \oplus
\Hom_{\tilde\cD_2}(\tau_2 M_2,\tau_2 N_2) \oplus
\tilde\varphi(\tau_2 N_2,\tau_1 M_1)[-1],
\end{equation*}
with appropriate differential.
The first summand here is quasiisomorphic to $\Hom_{\cD_1}(M_1,N_1)$ since $\tau_1 $ is quasi fully faithful,
the second summand is quasiisomorphic to $\Hom_{\cD_2}(M_2,N_2)$ since $\tau_2 $ is quasi fully faithful,
and the third summand is nothing but ${}_{\tau_2 }\tilde\varphi_{\tau_1 }(N_2,M_1)[-1]$ which is quasiisomorphic
via the morphism $c$ to $\varphi(N_2,M_1)[-1]$. Hence $\tau $ is quasi fully faithful.

It remains to verify properties (a)--(d). The first is clear:
\begin{equation*}
\Res_{\tau_1}(L\tilde\varphi(\LInd_{\tau_2}(M_2))) =
M_2\lotimes_{\cD_2}{}_{\tau_2}\tilde\cD_2\lotimes_{\tilde\cD_2}\tilde\varphi_{\tau_1} \cong
M_2\lotimes_{\cD_2}{}_{\tau_2}\tilde\varphi_{\tau_1} \cong
M_2\lotimes_{\cD_2}\varphi =
L\varphi(M_2),
\end{equation*}
hence the required isomorphism of functors.

Let $\tilde\imath_k:\tilde\cD_k \to \tilde\cD$ be the embedding DG-functor. Then
$\tilde\imath_k\circ\tau_k = \tau\circ i_k$ by definition of~$\tau$. This implies an isomorphism of the induction
functors $\LInd_{\tilde\imath_k}\circ\LInd_{\tau_k} \cong \LInd_\tau\circ\LInd_{i_k}$ which is precisely (b)
by definition of $I_k$ and $\tilde I_k$. Note also that (c) follows from (b) by passing to the right adjoint functors.

Further, let us prove (d). First consider the composition $I_2^!\circ\Res_\tau\circ\tilde I_1$. By (c) it is isomorphic
to $\Res_{\tau_2}\circ \tilde I_2^!\circ \tilde I_1$. But $\tilde I_2^!\circ\tilde I_1 = 0$ by semiorthogonality 
of $\bD(\tilde\cD_1)$ and $\bD(\tilde\cD_2)$ in $\bD(\tilde\cD)$. Thus we conclude that $I_2^!\circ\Res_\tau\circ\tilde I_1 = 0$.
It follows from the semiorthogonal decomposition of $\bD(\cD)$ that $\Res_\tau\circ\tilde I_1$ is in the image of $I_1$, hence
\begin{equation*}
\Res_\tau\circ\tilde I_1 =
I_1\circ I_1^!\circ\Res_\tau\circ\tilde I_1 \cong
I_1\circ \Res_{\tau_1}\circ\tilde I_1^!\circ\tilde I_1 \cong
I_1\circ \Res_{\tau_1},
\end{equation*}
the first is since $I_1$ is fully faithful, the second is by (c), the third is since $\tilde I_1$ is fully faithful.
%
%
%
%

Finally, let us prove (e). For this first note that
\begin{equation*}
I_2^!\circ\Res_\tau\circ\tilde I_2 \cong 
\Res_{\tau_2}\circ\tilde I_2^!\circ \tilde I_2 \cong 
\Res_{\tau_2}
\end{equation*}
by part (c) and full faithfulness of $\tilde I_2$. Therefore for any $N_2 \in \bD(\tilde\cD_2)$
the component of $\Res_\tau(\tilde I_2(N_2))$ in $\bD(\cD_2)$ equals $\Res_{\tau_2}(N_2)$. Hence its decomposition triangle 
looks as
\begin{equation*}
I_2(\Res_{\tau_2}(N_2)) \to \Res_\tau(\tilde I_2(N_2)) \to I_1(M_1)
\end{equation*}
for some $M_1 \in \bD(\tilde\cD_1)$. It remains to show that $M_1 = 0$. For this we apply the functor $I_1^!$ to the above triangle.
Since $I_1$ is fully faithful we get
\begin{equation*}
I_1^!(I_2(\Res_{\tau_2}(N_2))) \to I_1^!(\Res_\tau(\tilde I_2(N_2))) \to M_1.
\end{equation*}
Note that $I_1^!\circ I_2 \cong L\varphi[-1]$ by Proposition~\ref{derglu}, hence the first term is $L\varphi(\Res_{\tau_2}(N_2))[-1]$. 
On the other hand, $I_1^!\circ\Res_\tau\circ\tilde I_2 \cong \Res_{\tau_1}\circ\tilde I_1^!\circ\tilde I_2 \cong \Res_{\tau_1}\circ L\tilde\varphi[-1]$
(the first is by part (c) and the second again by Proposition~\ref{derglu}). So, if the canonical morphism 
$L\varphi\circ\Res_{\tau_2} \to \Res_{\tau_1}\circ L\tilde\varphi$ is an isomorphism then $M_1 = 0$ and we are done.
\end{proof}

\begin{remark}\label{parte}
Note that if $\tau_2 = \id$ then the condition of part (e) follows from (a). So, in this case the 
claim of part (e) holds.
\end{remark}

Now assume that $\varphi\in (\cD_2^\op \otimes \cD_1)\dgm$ is an h-projective bimodule 
(recall that by Lemma~\ref{qisqeq} we can replace the gluing bimodule
by its h-projective resolution without changing the quasiequivalence class
of the resulting category) and define the bimodule 
$\tilde\varphi \in (\tilde\cD_2^\op \otimes\tilde\cD_1)\dgm$ by
\begin{equation}\label{tiphi}
\tilde\varphi = \tilde\cD_{2\tau_2 } \otimes_{\cD_2} \varphi \otimes_{\cD_1} {}_{\tau_1 }\tilde\cD_1.
\end{equation}
Note that by definition of a DG-functor we have the canonical morphisms of bimodules
\begin{equation*}
\tau_1:\cD_1 \to {}_{\tau_1 }\tilde\cD_{1\tau_1 }
\qquad\text{and}\qquad
\tau_2:\cD_2 \to {}_{\tau_2 }\tilde\cD_{2\tau_2 }
\end{equation*}
which are quasiisomorphisms since $\tau_1 $ and $\tau_2 $ are quasi fully faithful.
Tensoring with $\varphi$ we obtain a morphism
\begin{equation*}
c:\varphi = \cD_2\otimes_{\cD_2}\varphi \otimes_{\cD_1} \cD_1 \xrightarrow{\ \tau_2\otimes 1\otimes \tau_1\ }
{}_{\tau_2 }\tilde\cD_{2\tau_2 } \otimes_{\cD_2} \varphi \otimes_{\cD_1} {}_{\tau_1 }\tilde\cD_{1\tau_1 } =
{}_{\tau_2 }\tilde\varphi_{\tau_1 }.
\end{equation*}
Note that since $\varphi$ is h-projective, this map is a quasiisomorphism, so $\tilde\varphi$ is compatible with~$\varphi$.
We call the gluing $\tilde\cD = \tilde\cD_1\times_{\tilde\varphi}\tilde\cD_2$ with $\tilde\varphi$ defined by~\eqref{tiphi}
the {\sf regluing} of $\cD$. By Proposition~\ref{compglu} the regluing $\tilde\cD$ is a partial categorical DG-resolution
of $\cD$. The following Proposition describes its properties.

\begin{prop}\label{reglu}
Let $\tilde\cD = \tilde\cD_1\times_{\tilde\varphi}\tilde\cD_2$ be the regluing of $\cD = \cD_1\times_{\varphi}\cD_2$. Then we have
\begin{enumerate}\renewcommand{\theenumi}{\alph{enumi}}
\item $L\tilde\varphi \cong \LInd_{\tau_1}\circ L\varphi\circ\Res_{\tau_2}$, an isomorphism of functors $\bD(\tilde\cD_2) \to  \bD(\tilde\cD_1)$;
\item $\Res_\tau\circ \tilde I_k \cong I_k\circ\Res_{\tau_1}$ for $k = 1,2$, an isomorphism of functors $\bD(\tilde\cD_k) \to \bD(\cD)$.
\end{enumerate}
%
%
\end{prop}
\begin{proof}
%
The first property is straightforward
\begin{multline*}
L\tilde\varphi(M) = 
M\lotimes_{\tilde\cD_2}\tilde\varphi =
M\lotimes_{\tilde\cD_2}(\tilde\cD_{2\tau_2}\lotimes_{\cD_2}\varphi\lotimes_{\cD_1}{}_{\tau_1}\tilde\cD_1) \cong \\ \cong
M_{\tau_2}\lotimes_{\cD_2}\varphi\lotimes_{\cD_1}{}_{\tau_1}\tilde\cD_1 \cong
\LInd_{\tau_1}(L\varphi(\Res_{\tau_2}(M)))
\end{multline*}
which is precisely the required formula. For (b) note that $k = 1$ case is given by Proposition~\ref{compglu}(d),
so it remains to consider the case $k = 2$. For this we note that
\begin{equation*}
\Res_{\tau_1}\circ L\tilde\varphi \cong
\Res_{\tau_1}\circ \LInd_{\tau_1}\circ L\varphi\circ \Res_{\tau_2} \cong
L\varphi\circ \Res_{\tau_2},
\end{equation*}
so the condition of Proposition~\ref{compglu}(e) is satisfied. Hence we have 
the required isomorphism by Proposition~\ref{compglu}(e).
\end{proof}

The following result is not needed for the rest of the paper, but we add it for completeness.

\begin{prop}\label{regluequi}
In the assumptions of Proposition~$\ref{compglu}$ the functor 
\begin{equation*}
\tau:\cD_1\times_\varphi\cD_2 \to \tilde\cD_1\times_{\tilde\varphi}\tilde\cD_2
\end{equation*}
is a quasiequivalence if and only if both $\tau_1$ and $\tau_2$ are quasiequivalences.
\end{prop}
\begin{proof}
Assume both $\tau_1$ and $\tau_2$ are quasiequivalences. Since we already know that $\tau$ is quasi fully faithful,
it only suffices to check that it is essentially surjective on the homotopy categories. So, let
$(\tilde M_1,\tilde M_2,\tilde\mu)$ be an object of $\tilde\cD_1\times_{\tilde\varphi}\tilde\cD_2$.
Since $\tau_1$ and $\tau_2$ are quasiequivalences there are objects $M_1 \in \cD_1$ and $M_2 \in \cD_2$
such that $\tilde M_1$ is homotopic to $\tau_1(M_1)$ and $\tilde M_2$ is homotopic to $\tau_2(M_2)$.
The above homotopies induce a homotopy equivalence of $\tilde\varphi(\tilde M_2,\tilde M_1)$ 
and $\tilde\varphi(\tau_2(M_2),\tau_1(M_1))$. The latter by assumption is quasiisomorphic
to $\varphi(M_2,M_1)$.
Thus, we can find
a closed element of zero degree $\mu \in \varphi(M_2,M_1)$ such that $c(\mu)$ goes to $\tilde\mu$
under the homotopy equivalence of the complexes mentioned above. It is quite easy to check then
that $\tau(M_1,M_2,\mu)$ is homotopic to $(\tilde M_1,\tilde M_2,\tilde\mu)$, so $\tau$
is essentially surjective on the homotopy categories.

Vice versa, assume $\tau$ is a quasiequivalence. We have to check that both $\tau_1$ and $\tau_2$
are essentially surjective on homotopy categories. Take any $\tilde M_1 \in \tilde\cD_1$.
Then $(\tilde M_1,0,0)$ is an object of $\tilde\cD_1\times_{\tilde\varphi}\tilde\cD_2$.
Hence there is an object $(M_1,M_2,\mu)$ of $\cD_1\times_\varphi\cD_2$ such that 
$\tau(M_1,M_2,\mu)$ is homotopic to $(\tilde M_1,0,0)$. By definition of the gluing it follows 
that $\tilde M_1$ is homotopic to $\tau_1(M_1)$. Thus $\tau_1$ is essentially surjective
on homotopy categories. Analogously one considers $\tau_2$.
%
\end{proof}


\section{Partial categorical resolution of a nonreduced scheme}\label{s-nred}

The goal of this section is to construct a categorical resolution
for a nonreduced scheme for which the corresponding reduced scheme is smooth. 
Our construction is inspired by the result of Auslander~\cite{A}.

\subsection{$\csA$-modules}

Let $S$ be a separated nonreduced scheme of finite type over $\kk$. Let $\fr \subset \cO_S$ 
be a nilpotent sheaf of ideals. Choose an integer $n$ such that 
$$
\fr^n = 0.
$$
Note that we do not require $n$ to be the minimal integer with this property.
In particular we can take $\fr = 0$ and arbitrary $n \ge 1$.

Consider a sheaf of noncommutative algebras
\begin{equation}\label{cadef}
\csA = \csA_{S,\fr,n} :=
\begin{pmatrix}
\cO_S & \fr & \fr^2 & \dots & \fr^{n-1} \\
\cO_S/\fr^{n-1} & \cO_S/\fr^{n-1} & \fr/\fr^{n-1} & \dots & \fr^{n-2}/\fr^{n-1} \\
\cO_S/\fr^{n-2} & \cO_S/\fr^{n-2} & \cO_S/\fr^{n-2} & \dots & \fr^{n-3}/\fr^{n-2} \\
\vdots & \vdots & \vdots & \ddots & \vdots \\
\cO_S/\fr & \cO_S/\fr & \cO_S/\fr & \dots & \cO_S/\fr 
\end{pmatrix}
\end{equation}
with the algebra structure induced by the embedding
\begin{equation}\label{caemb}
\csA_{S,\fr,n}
\subset
\End_{\cO_S}(\cO_S \oplus \cO_S/\fr^{n-1} \oplus \cO_S/\fr^{n-2} \oplus \dots \oplus \cO_S/\fr ).
\end{equation}
In other words, $\csA_{ij} = \fr^{\max(j-i,0)}/\fr^{n+1-i}$ and
the multiplication is the map
\begin{equation*}
\csA_{ij}\otimes\csA_{jk} = (\fr^{\max(j-i,0)}/\fr^{n+1-i}) \otimes (\fr^{\max(k-j,0)}/\fr^{n+1-j}) \to \fr^{\max(k-i,0)}/\fr^{n+1-i} = \csA_{ik}
\end{equation*}
induced by the multiplication map $\fr^a\otimes\fr^b \to \fr^c$ which is defined for all $c \le a+b$ --- we note that in our case
this condition is satisfied since
\begin{equation*}
\max(k-i,0) \le \max(j-i,0) + \max(k-j,0).
\end{equation*}
Also note that 
\begin{equation*}
\max(j-i,0)+(n+1-j) \ge n+1-i
\quad\text{and}\quad
\max(k-j,0) + (n+1-i) \ge n+1-i, 
\end{equation*}
so the above map descends to a map of the quotients.

We will denote 
\begin{equation}\label{ei}
\epsilon_i = 1 \in \csA_{ii} = \cO_S/\fr^{n+1-i}.
\end{equation}
Note that $\epsilon_1$, \dots, $\epsilon_n$ is a system of orthogonal idempotents with
\begin{equation*}
\epsilon_1 + \dots + \epsilon_n = 1.
\end{equation*}

The algebra $\csA = \csA_{S,\fr,n}$ will be called {\sf the Auslander algebra} ($\csA$ is for Auslander!).
The scheme $S$ will be referred to as the {\sf underlying scheme}, the ideal $\fr$ as {\sf the defining ideal} 
and the integer $n$ as {\sf the width} of the Auslander algebra.

\begin{remark}
Auslander in his paper~\cite{A} in case of a zero-dimensional scheme $S$ considered 
the algebra~$\End_{\cO_S}(\cO_S \oplus \cO_S/\fr^{n-1} \oplus \cO_S/\fr^{n-2} \oplus \dots \oplus \cO_S/\fr )$
and showed that it has finite global dimension. One can show (see section~\ref{ss-gldim}) that a similar statement
is true for the algebra $\csA$ for the scheme $S$ of arbitrary dimension as soon as the reduced scheme $S_\red$
is smooth (note that this holds automatically if $S$ is zero-dimensional).
Moreover, we show in section~\ref{ss-ausgen} that there is a more general class of generalized Auslander algebras which have 
similar properties. Now we just note that the algebra we consider coincides with the original Auslander algebra
if and only if $\fr^a:\fr^b = \fr^{a-b}$ for all $a > b$ (the LHS is the ideal of all elements multiplication by which
takes $\fr^b$ to $\fr^a$).
\end{remark}

In what follows we always consider {\em right} $\csA$-modules, unless the opposite is specified.

An $\csA$-module is called {\sf quasicoherent} ({\sf coherent}) if it is such as an $\cO$-module.
We denote by $\Qcoh(\csA)$ the category of quasicoherent $\csA$-modules on $S$ and by $\coh(\csA)$ 
its subcategory of coherent $\csA$-modules.

\begin{remark}\label{oa}
Note that taking $\fr = 0$ and $n = 1$ we obtain 
\begin{equation*}
\csA_{S,0,1} \cong \cO_S. 
\end{equation*}
Thus the category of $\cO_S$-modules is just a very special case of the category of $\csA$-modules.
\end{remark}

%
%
%

\begin{example}\label{dn}
A simple but rather instructive example is the case of $S = \Spec \kk[t]/t^2$, $\fr = t\kk[t]/t^2$, $n = 2$.
In this case the category $\Qcoh(\csA_S)$ is the category of representations of the quiver with relation
\begin{equation*}
\xymatrix@1@C=40pt{\bullet \ar@<-.5ex>[r]_\beta & \bullet \ar@<-.5ex>[l]_\alpha},
\qquad \beta\alpha = 0
\end{equation*} 
and $\csA_{S,\fr,2}$ is just the path algebra of this quiver. On the other hand, if we take the same $S$, $\fr$,
and choose $n = 3$ we will get a quiver
\begin{equation*}
\xymatrix@1@C=40pt{\bullet \ar@<-.5ex>[r]_{\beta_2} & \bullet \ar@<-.5ex>[l]_{\alpha_2} \ar@<-.5ex>[r]_{\beta_1} & \bullet \ar@<-.5ex>[l]_{\alpha_1}},
\qquad \beta_2\alpha_2 = \alpha_1\beta_1,\ \beta_1\alpha_1 = 0,\ \beta_1\beta_2 = 0
\end{equation*} 
and $\csA_{S,\fr,3}$ is again the path algebra of this quiver.
\end{example}

\subsection{The category of $\csA$-spaces}\label{ss-afun}

An {\sf $\csA$-space} is a triple $(S,\fr_S,n_S)$, where $S$ is a scheme, $\fr$ is a nilpotent ideal in $\cO_S$ 
and $n$ is an integer such that $\fr^n = 0$. An $\csA$-space $(S,\fr_S,n_S)$ can be thought of as a ringed space 
with the underlying topological space of $S$ and the sheaf of (noncommutative) rings $\csA_S = \csA_{S,\fr,n}$. 
Sometimes we will denote an $\csA$-space $(S,\fr_S,n_S)$ simply by $(S,\cA_S)$.
With any $\csA$-space we associate the category $\Qcoh(\csA)$ of quasicoherent sheaves of right $\csA$-modules.




There are several ways how one can define morphisms of $\csA$-spaces. 
Here we will use the simplest possible definition.
More general definition will be discussed in section~\ref{ss-genmaps}.

Assume that $(S,\fr_S,n_S)$ and $(T,\fr_T,n_T)$ is a pair of $\csA$-spaces
and let $f:T \to S$ be a morphism of underlying schemes.
Assume that the defining ideals $\fr_T$ and $\fr_S$ and the widths $n_T$ and $n_S$
are compatible in the following sense:
\begin{equation}\label{rst}
f^{-1}(\fr_S) \subset \fr_T
\quad\text{and}\quad
n_S \le n_T.
\end{equation} 
It follows from~\eqref{rst} that $f^{-1}(\fr_S^i) \subset \fr_T^i$ for all $i$.
Therefore the morphism of schemes induces maps
\begin{equation*}
f^\bull(\fr_S^i/\fr_S^j) = f^\bull(\fr_S^i)/f^\bull(\fr_S^j) \to  \fr_T^i/\fr_T^j
\end{equation*}
for all $i$ and $j$ which are compatible with the multiplication maps.
Combining all these together we obtain a morphism of sheaves of algebras $f^\bull(\csA_S) \to \csA_T$
which takes the unit of the first to the idempotent $\epsilon_1 + \dots + \epsilon_{n_S}$ of the second.
In particular, $(\epsilon_1 + \dots + \epsilon_{n_S})\csA_T$
has a natural structure of a $f^{-1}\csA_S$-$\csA_T$-bimodule which is projective over $\csA_T$.


We define the pullback and the pushforward functors with respect
to a morphism $f$ by
\begin{equation}\label{ffpp}
\begin{array}{lll}
f^*: & \Qcoh(\csA_S) \to \Qcoh(\csA_T), \quad &
M \mapsto f^\bull M\otimes_{f^\bull\csA_S}(\epsilon_1 + \dots + \epsilon_{n_S})\csA_T,\\
f_*: & \Qcoh(\csA_T) \to \Qcoh(\csA_S), \quad & 
N \mapsto (f_*N)(\epsilon_1 + \dots + \epsilon_{n_S}).
\end{array}
\end{equation}

\begin{lemma}\label{apbpf}
The functor $f^*$ is right exact and the functor $f_*$ is left exact.
Moreover, the functor $f^*$ is the left adjoint of the functor $f_*$.
\end{lemma}
\begin{proof}
The functor $f^*$ is the composition of the exact functor $f^{-1}$ and the right exact functor of tensor product,
so it is right exact. Analogously, the functor $f_*$ is the composition of the 
left exact functor $f_*$ and of the exact functor $N \mapsto N(\epsilon_1 + \dots + \epsilon_{n_S})$. 
Finally, the adjunction is classical
\begin{multline*}
\Hom_{\csA_T}(f^*M,N) =
\Hom_{\csA_T}(f^{-1}M\otimes_{f^{-1}\csA_S}(\epsilon_1 + \dots + \epsilon_{n_S})\csA_T,N) \cong \\ \cong
\Hom_{f^{-1}\csA_S}(f^{-1}M,\cHom_{\csA_T}((\epsilon_1 + \dots + \epsilon_{n_S})\csA_T,N)) \cong 
\\ \cong
\Hom_{f^{-1}\csA_S}(f^{-1}M,N(\epsilon_1 + \dots + \epsilon_{n_S})) \cong
\Hom_{\csA_S}(M,f_*N(\epsilon_1 + \dots + \epsilon_{n_S})).
\end{multline*}
Here the first isomorphism is the classical adjunction of $\cHom$ and $\otimes$, 
the second is evident, and the third is the classical adjunction of $f^{-1}$ and $f_*$.
\end{proof}

\begin{lemma}\label{ppgf}
Assume that
\begin{equation*}
(U,\fr_U,n_U) \xrightarrow{\ g\ } (T,\fr_T,n_T) \xrightarrow{\ f\ } (S,\fr_S,n_S)
\end{equation*}
are morphisms of $\csA$-spaces. 
Then there is an isomorphism of functors $g^*\circ f^* \cong (f\circ g)^*$ from $\Qcoh(\csA_S)$ to $\Qcoh(\csA_U)$.
\end{lemma}
\begin{proof}
Indeed, 
\begin{multline*}
g^*(f^*(F)) =
g^{-1}(f^{-1}F\otimes_{f^{-1}\csA_S} (\epsilon_1+\dots+\epsilon_{n_S})\csA_T) \otimes_{g^{-1}\csA_T} (\epsilon_1+\dots+\epsilon_{n_T})\csA_U \cong \\ \cong
g^{-1}f^{-1}F\otimes_{g^{-1}f^{-1}\csA_S} ((\epsilon_1+\dots+\epsilon_{n_S})g^{-1}\csA_T \otimes_{g^{-1}\csA_T} (\epsilon_1+\dots+\epsilon_{n_T})\csA_U) \cong 
\\ \cong
g^{-1}f^{-1}F\otimes_{g^{-1}f^{-1}\csA_S} (\epsilon_1+\dots+\epsilon_{n_S})\csA_U 
= (f\circ g)^*(F)
\end{multline*}
since $(\epsilon_1+\dots+\epsilon_{n_S})(\epsilon_1 + \dots + \epsilon_{n_T}) = \epsilon_1 + \dots + \epsilon_{n_S}$ by~\eqref{rst}.
\end{proof}

\subsection{Derived category}

Let $\com(\csA)$ denote the category of (unbounded) complexes of quasicoherent $\csA$-modules.

\begin{lemma}\label{ahf}
$(1)$ The category $\Qcoh(\csA)$ is a Grothendieck category.

\noindent
$(2)$ The category $\com(\csA)$ has enough h-injectives complexes.

\noindent
$(3)$ The category $\com(\csA)$ has enough h-flat complexes.
\end{lemma}
\begin{proof} Note that all these assertions hold for the category $\Qcoh(S)$
instead of $\Qcoh(\csA)$.

(1) The category $\Qcoh(\csA)$ has arbitrary direct sums and filtered direct limits 
in $\Qcoh(\csA)$ are exact. It remains to show that it has a generator.
Recall that $\csA$ is an $\cO_S$-algebra. Consider the corresponding extension and restriction
of scalars functors
\begin{equation*}
\begin{array}{rclrcl}
\Ind_S &:& \Qcoh(S) \to \Qcoh(\csA),\quad &\cM &\mapsto& \cM \otimes_{\cO_S}\csA,\\
\Res_S &:& \Qcoh(\csA) \to \Qcoh(S),\quad &\cN &\mapsto& \cN_{\cO_S}.
\end{array}
\end{equation*}
It is a classical fact that $\Ind_S$ is the left adjoint of $\Res_S$.

Let $U\in \Qcoh(S)$ be a generator. We claim that $\Ind_S(U)$ is a generator of
$\Qcoh(\csA)$. Indeed, let $\cM _0\varsubsetneq \cM$ be objects in $\Qcoh(\csA)$. Then
$\Res_S (\cM _0)\varsubsetneq \Res_S (\cM)$. Hence the map
$\Hom (U,\Res_S (\cM _0))\to \Hom (U,\Res_S (\cM))$ is not surjective. But then
the map $\Hom (\Ind_S(U),\cM _0)\to \Hom (\Ind_S(U),\cM)$ is also not surjective. So
$\Ind_S(U)$ is a generator and hence $\Qcoh(\csA)$ is a Grothendieck category.

(2) This follows from (1) and~\cite{KSch}.

(3) The proof given in~\cite{AJL} for the category $\Qcoh(S)$ works also for $\Qcoh(\cB)$
for any sheaf of $\cO_S$-algebras $\cB$ which is quasi-coherent as $\cO_S$-module.
\end{proof}

We denote by $\ASch$ the category of $\csA$-spaces with the underlying scheme being separated of finite type
and with morphisms defined in section~\ref{ss-afun}. As we already observed, each scheme can be thought of as an $\csA$-space
of width $1$ and each morphism of schemes gives a morphism of the corresponding $\csA$-spaces. 
Thus the category $\Sch$ is a subcategory of $\ASch$. Our goal is to extend
the derived category pseudofunctor $\bD:\Sch^\opp \to \Tria$ to the category $\ASch$.

As we know by Lemma~\ref{ahf} the category $\com(\Qcoh(\csA_S))$ has enough h-flat objects. Hence 
$[\com(\csA_S)]/[\coma(\csA_S)] \cong [\hflat(\csA_S)]/[\hflata(\csA_S)]$, where
$\coma(\csA_S) \subset \com(\csA_S)$ is the DG-subcategory of acyclic complexes,
$\hflat(\csA_S)$ is the DG-category of h-flat complexes of $\csA_S$-modules
and $\hflata(\csA_S)$ is its DG-subcategory of h-flat acyclic complexes.
So, analogously to the case of schemes we can identify the derived category as
\begin{equation*}
\bD(\csA_S) \cong [\hflat(\csA_S)/\hflata(\csA_S)],
\end{equation*}
the homotopy category of the Drinfeld quotient.
It is clear that $\bD^b(\coh(\csA_S))$, the bounded derived category of coherent $\csA_S$-modules
is equivalent to the subcategory of $\bD(\csA_S)$ consisting of complexes with finite number
of cohomology sheaves all of which are coherent. Indeed, by a standard argument (see~\cite[II.2.2]{SGA6})
it is enough to check that a quotient of a coherent $\csA_S$-module is coherent and that any quasicoherent
$\csA_S$-module is a filtered direct limit of coherent $\csA_S$-modules. The first is evident and for 
the second note that if $\cN$ is an $\csA_S$-module and $\Res\cN = \lim \cM_\lambda$ with $\cM_\lambda$
coherent sheaves on $S$, then $\cN$ is the limit of images of $\Ind(\cM_\lambda)$ and each of those
is a coherent $\csA_S$-module.


\begin{prop}
Associating with an $\csA$-space $(S,\csA_S)$ its derived category $\bD(\csA_S)$ is a pseudofunctor
$\ASch^\opp \to \Tria$ from the category opposite to $\csA$-spaces to the $2$-category of triangulated categories,
extending the same named pseudofunctor $\Sch^\opp \to \Tria$.
\end{prop}
\begin{proof}
To define a pseudofunctor $\ASch^\opp \to \Tria$ we have to associate with a morphism $f:(T,\csA_T) \to (S,\csA_S)$
of $\csA$-spaces a functor $\bD(\csA_S)\to  \bD(\csA_T)$. This role will be played by the derived pullback 
functor $Lf^*$ which we will define along the same lines as the usual derived pullback functor
was defined in section~\ref{ss-dqs}.

Using analogue of Lemma~\ref{pbfa} we deduce that the functor $f^*$ takes $\hflat(\csA_S)$ to $\hflat(\csA_T)$
and $\hflata(\csA_S)$ to $\hflata(\csA_T)$, hence by Proposition~\ref{dq-dgf} induces a DG-functor
$\hflat(\csA_S)/\hflata(\csA_S) \to \hflat(\csA_T)/\hflata(\csA_T)$ of Drinfeld quotients.
Passing to homotopy categories we define the derived functor
\begin{equation*}
Lf^*:\bD(\csA_S) = [\hflat(\csA_S)/\hflata(\csA_S)] \xrightarrow{\ f^*\ } [\hflat(\csA_T)/\hflata(\csA_T)] = \bD(\csA_T).
\end{equation*}
The pseudofunctoriality and compatibility with the derived category pseudofunctor on schemes is evident.
\end{proof}


Let $f:(T,\csA_T) \to (S,\csA_S)$ be a morphism of $\csA$-spaces.
By construction the derived pullback functor $Lf^*$ commutes with arbitrary direct sums.
Therefore, by Brown Representability it has a right adjoint functor 
$Rf_*:\bD(\csA_T) \to \bD(\csA_S)$.

Consider a slight generalization of the induction functor 
$\Ind^k_S:\Qcoh(S) \to \Qcoh(\csA_S)$, $M \mapsto M\otimes_{\cO_S}(\epsilon_1+\dots+\epsilon_k)\csA_S$.
Its right adjoint functor $\Res^k_S:\Qcoh(\csA_S)\to \Qcoh(S)$ is given by $N \mapsto \Res_S(N(\epsilon_1+\dots+\epsilon_{k}))$.
Note that it is exact. The above argument allows to define the derived induction functor $\LInd^k_S:\bD(S) \to \bD(\csA_S)$
by applying $\Ind^k_S$ to h-flat complexes. On the other hand, the exactness of the restriction functor shows that
it descends to the derived category $\Res^k_S:\bD(\csA_S)\to \bD(S)$ and the adjunction is preserved.




\begin{lemma}\label{rfscoh}
If $f:(T,\fr_T,n_T) \to (S,\fr_S,n_S)$ is a morphism of $\csA$-spaces then 
\begin{equation*}
\Res_S(Rf_*(F)) \cong Rf_*(\Res^{n_S}_T(F)).
\end{equation*}
In particular, if $f$ is proper and of finite type then $Rf_*(\bD^b(\coh(\csA_T)) \subset \bD^b(\coh(\csA_S))$.
\end{lemma}
\begin{proof}
Both sides are compositions of right adjoint functors, so it is enough to check 
an isomorphism of their left adjoint functors
$Lf^*\circ \LInd_S \cong \LInd^{n_S}_T\circ Lf^*$.
Since the induction of an h-flat complex of $\cO_S$-modules is evidently h-flat over $\csA_S$, it is enough to note that
we have an isomorphism for underived functors. The latter is clear as
\begin{multline*}
f^*(\Ind_S(M)) = f^*(M\otimes_{\cO_S}\csA_S) = 
f^{-1}(M\otimes_{\cO_S}\csA_S) \otimes_{f^{-1}\csA_S}(\epsilon_1+\dots+\epsilon_{n_S})\csA_T \cong \\ \cong
f^{-1}M\otimes_{f^{-1}\cO_S}f^{-1}\csA_S \otimes_{f^{-1}\csA_S}(\epsilon_1+\dots+\epsilon_{n_S})\csA_T \cong
f^{-1}M\otimes_{f^{-1}\cO_S}(\epsilon_1+\dots+\epsilon_{n_S})\csA_T \cong \\ \cong
f^{-1}M\otimes_{f^{-1}\cO_S}\cO_T\otimes_{\cO_T}(\epsilon_1+\dots+\epsilon_{n_S})\csA_T =
\Ind^{n_S}_T(f^*M).
\end{multline*}
%
%
The second statement is clear since $M \in \bD(\csA_T)$ being bounded and coherent implies $\Res^k_T(M)$ is such
for any $k$, and vice versa, if $\Res_S(M) \in \bD(S)$ is bounded and coherent then so is $M$.
\end{proof}

\subsection{Semiorthogonal decomposition}\label{ss-aussod}


Let $S_0$ and $S'$ be the subschemes of $S$ corresponding to the ideals $\fr$ and $\fr^{n-1}$ respectively.
For both the underlying topological space is $S$ but the sheaves of rings are
\begin{equation*}
\cO_{S_0} := \cO_S/\fr,
\quad\text{and}\quad
\cO_{S'} = \cO_S/\fr^{n-1}.
\end{equation*}

Consider the two-sided ideal $\csI$ in $\csA$ generated by 
the idempotent $1-\epsilon_1=\epsilon_2+\dots+\epsilon_n$.

\begin{lemma}\label{csi}
We have 
\begin{equation}\label{cidef}
\csI := \csA(1-\epsilon_1)\csA = 
\begin{pmatrix}
\fr & \fr & \fr^2 & \dots & \fr^{n-1} \\
\cO_S/\fr^{n-1} & \cO_S/\fr^{n-1} & \fr/\fr^{n-1} & \dots & \fr^{n-2}/\fr^{n-1} \\
\cO_S/\fr^{n-2} & \cO_S/\fr^{n-2} & \cO_S/\fr^{n-2} & \dots & \fr^{n-3}/\fr^{n-2} \\
\vdots & \vdots & \vdots & \ddots & \vdots \\
\cO_S/\fr & \cO_S/\fr & \cO_S/\fr & \dots & \cO_S/\fr 
\end{pmatrix}
\end{equation}
In particular, $\csA/\csI \cong \cO_S/\fr \cong \cO_{S_0}$.
\end{lemma}
\begin{proof}
Straightforward.
\end{proof}

By Lemma~\ref{csi} the sheaf $\cO_{S_0}$ has a natural structure of an algebra over $\csA_S$.
Thus we have the restriction of scalars functor
\begin{equation}\label{sia}
\si:\Qcoh({S_0}) \to \Qcoh(\csA),
\qquad
M \mapsto M.
\end{equation} 
On the other hand, consider the algebra
\begin{equation}\label{csap}
\csA'= \csA_{S',\fr,n-1} = \oplus_{i,j \ge 2}\csA_{ij} = (1-\epsilon_1)\csA(1-\epsilon_1)
\end{equation}
(which is just the same type algebra constructed from the scheme $S'$), and note that the following $(n-1)\times n$ matrix
\begin{equation}\label{omea}
(1-\epsilon_1)\csA = 
\begin{pmatrix}
\cO_S/\fr^{n-1} & \cO_S/\fr^{n-1} & \fr/\fr^{n-1} & \dots & \fr^{n-2}/\fr^{n-1} \\
\cO_S/\fr^{n-2} & \cO_S/\fr^{n-2} & \cO_S/\fr^{n-2} & \dots & \fr^{n-3}/\fr^{n-2} \\
\vdots & \vdots & \vdots & \ddots & \vdots \\
\cO_S/\fr & \cO_S/\fr & \cO_S/\fr & \dots & \cO_S/\fr 
\end{pmatrix}
\end{equation}
is an $\csA'$-$\csA$-bimodule, so we can also define a functor
\begin{equation}\label{sea}
\se:\Qcoh(\csA') \to \Qcoh(\csA),
\qquad
M \mapsto M\otimes_{\csA'}(1-\epsilon_1)\csA.
\end{equation} 

\begin{remark}\label{eproj}
It is useful to note that the functor $\se$ takes projective $\csA'$-modules to projective $\csA$-modules.
This follows immediately from the fact that $(1-\epsilon_1)\csA$ is projective as an $\csA$-module.
\end{remark}

\begin{lemma}
The functor $\sip$ has a left adjoint functor $\sip^*:\Qcoh(\csA) \to \Qcoh(S_0)$ 
\begin{equation}\label{sisa}
\sip^*(M) = M\otimes_{\csA}\cO_{S_0},
\end{equation} 
and the functor $\sep$ has a right adjoint functor $\sep^!:\Qcoh(\csA) \to \Qcoh(\csA')$ 
\begin{equation}\label{sesa}
\se^!(N) = N(1-\epsilon_1).
\end{equation} 
\end{lemma}
\begin{proof}
The left adjoint to the restriction of scalars is the extension of scalars functor,
which is precisely the functor $\si^*$. On the other hand, it is clear that
\begin{equation*}
\Hom_{\csA}(\se(M),N) =
\Hom_{\csA}(M\otimes_{\csA'}(1-\epsilon_1)\csA,N) \cong
\Hom_{\csA'}(M,\Hom_{\csA}((1-\epsilon_1)\csA,N)) 
\end{equation*}
hence the right adjoint of $\se$ is given by
$N \mapsto \Hom_{\csA}((1-\epsilon_1)\csA,N) = N(1-\epsilon_1)$
which gives~\eqref{sesa}.
%
%
%
\end{proof}

\begin{remark}
In fact all these functors are just pullbacks and pushforward for generalized morphisms of $\csA$-spaces,
see Example~\ref{egm}.
\end{remark}

It is clear from~\eqref{sia}, \eqref{sea} and~\eqref{sesa} that the functors $\si$, $\se$, and $\se^!$ are exact. We extend them termwise
to the categories of complexes, and to the derived categories. On the other hand, the functor 
$\si^*$ is only right exact. We extend it termwise to the category of h-flat complexes of quasicoherent $\csA$-modules
and thus obtain its left derived functor $L\si^*:\bD(\csA) \to \bD(S_0)$,
\begin{equation}\label{lsesa}
L\si^*(M) = M\lotimes_{\csA}\cO_{S_0}.
\end{equation} 

\begin{prop}\label{ie-equ}
The functor $L\si^*:\bD(\csA) \to \bD(S_0)$ defined by~\eqref{lsesa} is left adjoint to the functor $\si:\bD(S_0) \to \bD(\csA)$,
and the functor $\se^!:\bD(\csA) \to \bD(\csA')$ is right adjoint to $\se:\bD(\csA') \to \bD(\csA)$.
Moreover, we have
\begin{equation}\label{iec}
L\sip^*\circ\sip \cong \id_{\bD(S_0)},
\qquad
\sep^!\circ\sip = 0,
\qquad
L\sip^*\circ\sep = 0,
\qquad
\sep^!\circ\sep \cong \id_{\bD(\csA')}.
\end{equation} 
In particular, the functors $\si$ and $\se$ are fully faithful.
\end{prop}
\begin{proof}
The adjunctions of the functors on the derived level follow from those on the abelian level.
So, let us verify equations~\eqref{iec}. The first composition is the functor 
\begin{equation*}
M \mapsto M\lotimes_{\cO_{S_0}} (\cO_{S_0}\lotimes_{\csA}\cO_{S_0}),
\end{equation*}
so we have to compute the derived tensor product $\cO_{S_0}\lotimes_{\csA}\cO_{S_0}$.
For this we use the defining exact sequence
\begin{equation}\label{cicos0}
0 \to \csI \to \csA\to \cO_{S_0}\to 0
\end{equation}
and note that
\begin{equation}\label{csieq}
\csI = \csA(1-\epsilon_1)\csA \cong \csA(1-\epsilon_1)\lotimes_{\csA'}(1-\epsilon_1)\csA.
\end{equation}
Indeed, it is clear from~\eqref{omea} that $(1-\epsilon_1)\csA$ is isomorphic to $\csA'\epsilon_2 \oplus \csA'$ as an $\csA'$-module
(if we remove the first column from~\eqref{omea} we get $\csA'$, the first column itself just repeats the first column 
of $\csA'$, so the corresponding $\csA'$-module is the projective module corresponding to the first idempotent of $\csA'$
which is the second idempotent of $\csA$), hence it is projective and the RHS above is isomorphic to
\begin{equation*}
\csA(1-\epsilon_1)\lotimes_{\csA'}(\csA'\epsilon_2 \oplus \csA') \cong
\csA(1-\epsilon_1)\epsilon_2 \oplus \csA(1-\epsilon_1) =
\csA\epsilon_2 \oplus \csA(1-\epsilon_1) 
\end{equation*}
which clearly coincides with $\csI$. Now it follows that 
\begin{equation*}
\csI\lotimes_{\csA}\cO_{S_0} \cong 
\csA(1-\epsilon_1)\lotimes_{\csA'}(1-\epsilon_1)\csA\lotimes_{\csA}\cO_{S_0} \cong 
\csA(1-\epsilon_1)\lotimes_{\csA'}(1-\epsilon_1)\cO_{S_0} = 0
\end{equation*}
since $1-\epsilon_1$ acts trivially on $\cO_{S_0}$. Now tensoring~\eqref{cicos0} with $\cO_{S_0}$
we conclude that 
\begin{equation*}
\cO_{S_0}\lotimes_{\csA}\cO_{S_0} \cong
\csA\lotimes_{\csA}\cO_{S_0} \cong
\cO_{S_0}.
\end{equation*}
This shows that the composition $L\si^*\circ\si$ is isomorphic to the identity.

All the other compositions are evident. Indeed,
$\se^!(\si(M)) = M(1-\epsilon_1) = 0$ for any object $M \in \bD(S_0)$ since as we already
observed $1-\epsilon_1$ acts on all $\cO_{S_0}$-modules trivially. Analogously,
\begin{equation*}
L\si^*(\se(M)) = 
M\lotimes_{\csA'}(1-\epsilon_1)\csA\lotimes_{\csA}\cO_{S_0} \cong
M\lotimes_{\csA'}(1-\epsilon_1)\cO_{S_0} = 0
\end{equation*}
and the reason is the same.
Finally, by~\eqref{csap}
\begin{equation*}
\se^!(\se(M)) = 
M\lotimes_{\csA'}(1-\epsilon_1)\csA(1-\epsilon_1) =
M\lotimes_{\csA'}\csA' \cong M
\end{equation*}
so the composition $\se^!\circ\se$ is the identity.
\end{proof}

Now we can prove that the category $\bD(\csA)$ has a semiorthogonal decomposition.

\begin{prop}\label{dqcsod}
There are semiorthogonal decompositions
$$
\bD(\csA) = \langle \si(\bD(S_0)), \se(\bD(\csA')) \rangle,
\qquad
\bD^b(\coh(\csA)) =  \langle \si(\bD^b(\coh(S_0))), \se(\bD^b(\coh(\csA'))) \rangle.
$$
\end{prop}
\begin{proof}
First consider the category $\bD(\csA)$.
We already know that the functors $\si$ and $\se$ are fully faithful and have the left
and the right adjoint respectively. Moreover, 
\begin{equation*}
\Hom_{\bD(\csA)}(\se(M'),\si(M_0)) \cong
\Hom_{\bD(\csA')}(M',\se^!(\si(M_0))) = 0
\end{equation*}
by~\eqref{iec}, hence the subcategories are semiorthogonal.
It remains to check that each object $M \in \bD(\csA)$
fits into a triangle with the other vertices in $\si(\bD(S_0))$ and $\se(\bD(\csA'))$.
For this we tensor the sequence~\eqref{cicos0} with $M$ to get a distinguished triangle
\begin{equation*}
M\lotimes_{\csA}\csI \to M \to M\lotimes_{\csA}\cO_{S_0}.
\end{equation*}
Note that the last term is nothing but $\si(L\si^*(M))$, so it remains to note that
\begin{equation*}
\se(\se^!(M)) \cong
(M\lotimes_{\csA}\csA(1-\epsilon_1))\lotimes_{\csA'}(1-\epsilon_1)\csA \cong
M\lotimes_{\csA}\csI
\end{equation*}
(the first isomorphism is by definitions of $\se$ and $\se^!$, see~\eqref{sea}and~\eqref{sesa},
and the second follows from~\eqref{csieq}). So, we see that the above triangle can be rewritten
as
\begin{equation}\label{eemii}
\se(\se^!(M)) \to M \to \si(L\si^*(M))
\end{equation}
which proves the first decomposition.

For the second decomposition note that we clearly have $\si(\bD^b(\coh(S_0))) \subset \bD^b(\coh(\csA))$.
Further, since $(1-\epsilon_1)\csA$ is a projective $\csA'$-module (see the proof of Proposition~\ref{ie-equ}),
we deduce that $\se(\bD^b(\coh(\csA'))) \subset \bD^b(\coh(\csA))$. So, it remains to check 
inclusions for the adjoint functors.

The inclusion $\se^!(\bD^b(\coh(\csA))) \subset \bD^b(\coh(\csA'))$
follows immediately from~\eqref{sesa}. To prove the inclusion for $L\si^*$ consider the triangle~\eqref{eemii}
with $M \in \bD^b(\coh(\csA))$. Then from the above it follows that its first vertex is also in $\bD^b(\coh(\csA))$,
hence the same is true for the last vertex. It remains to note that $\si(L\si^*(M)) \in \bD^b(\coh(\csA))$
implies $L\si^*(M) \in \bD^b(\coh(S_0))$ since the functor $\si$ is evidently conservative.
\end{proof}

Iterating the above argument we deduce the following 

\begin{cor}\label{sodqqc}
There are semiorthogonal decompositions
\begin{equation*}
\begin{array}{lll}
\bD(\csA_{S,\fr,n}) &=& \langle \underbrace{\bD(S_0),\bD(S_0),\dots,\bD(S_0)}_{\text{$n$ components}} \rangle,\\
\bD^b(\coh(\csA_{S,\fr,n})) &=& 
\langle \underbrace{\bD^b(\coh(S_0)), \bD^b(\coh(S_0)), \dots, \bD^b(\coh(S_0))}_{\text{$n$ components}} \rangle.
\end{array}
\end{equation*}
\end{cor}


\begin{example}
Let $S = \Spec \kk[t]/t^2$, $\fr = t\kk[t]/t^2$, $n=2$. Then $S_0 = \Spec\kk$, so the semiorthogonal decomposition
is just an exceptional pair $(E_1,E_2)$ generating the category. The corresponding exceptional objects are 
\begin{equation*}
E_1 = (\xymatrix@1{\kk \ar@<-.5ex>[r] & 0 \ar@<-.5ex>[l]})
\quad\text{and}\quad
E_2 = (\xymatrix@1{\kk \ar@<-.5ex>[r]_0 & \kk \ar@<-.5ex>[l]_1 }),
\end{equation*}
the simple module of the first vertex and the projective module of the second vertex respectively.
The easiest way to check that the pair is exceptional is by using the projective resolution
$0 \to P_2 \to P_1 \to E_1 \to 0$.
\end{example}

\subsection{Perfect and compact $\csA$-modules}


We will say that a complex of $\csA$-modules is {\sf perfect} if it is locally quasiisomorphic to a finite complex
of projective $\csA$-modules. All perfect complexes form a triangulated subcategory of $\bD^b(\coh(\csA))$ which 
we denote by $\bD^\perf(\csA)$. 


\begin{prop}\label{bperf}
If $S_0$ is smooth then $\bD^b(\coh(\csA)) = \bD^\perf(\csA)$.
\end{prop}
\begin{proof}
As we already observed, the embedding $\bD^\perf(\csA) \subset \bD^b(\coh(\csA))$ is automatic.
For the opposite we use induction on $n$. For $n = 1$ the fact is well known. Now assume that $n > 1$.
Take any $M \in \bD^b(\coh(\csA))$ and consider the triangle~\eqref{eemii}. 
In the proof of Proposition~\ref{dqcsod} we checked that $\se^!(M) \in \bD^b(\coh(\csA))$ and
$L\si^*(M) \in \bD^b(\coh(S_0))$. Hence by induction hypothesis both are perfect.
%
On one hand, as it was observed in Remark~\ref{eproj} the functor $\se$
takes projective $\csA'$-modules to projective $\csA$-modules, hence $\se(\se^!(M))$ is perfect.
%
On the other hand, it follows that $L\si^*(M)$ is locally quasiisomorphic to a finite complex of free $\cO_{S_0}$-modules.
So, it remains to check that $\si(\cO_{S_0})$ is perfect. 

For this we note that $\cO_{S_0} \cong L\si^*(\csA)$, 
hence the triangle~\eqref{eemii} for $M = \csA$ reads as
\begin{equation*}
\se(\se^!(\csA)) \to \csA \to \si(\cO_{S_0}).
\end{equation*}
Its first vertex is perfect by the above argument and its second vertex is evidently perfect. Hence the third
vertex is perfect and we are done.
%
%
%
%
%
\end{proof}

It is also clear that every perfect complex of $\csA$-modules is a compact object of $\bD(\csA)$
(the argument of Neeman \cite[Example 1.13]{N96} can be easily adjusted to our situation).
Therefore the category $\bD(\csA)$ is compactly generated.
Moreover, if $S_0$ is smooth then one can also check that all compact objects are perfect.

\begin{prop}\label{cperf}
If $S_0$ is smooth then $\bD(\csA)^\comp = \bD^\perf(\csA)$.
\end{prop}
\begin{proof}
We will use induction on $n$. If $n = 1$ the $\bD(\csA) = \bD(S) = \bD(S_0)$ and the statement follows from~\eqref{compperf}.

Assume $n > 1$. Let $M$ be an arbitrary compact object in $\bD(\csA)$. By Lemma~\ref{cprops}(2) we know that $L\si^*(M) \in \bD(S_0)^\comp$ 
since the functor $\si$ commutes with arbitrary direct sums. It follows that $L\si^*(M) \in \bD^b(\coh(S_0))$,
hence $\si(L\si^*(M)) \in \bD^b(\coh(\csA))$. Thus by Proposition~\ref{bperf}
\begin{equation*}
\si(L\si^*(M)) \in \bD^\perf(\csA). 
\end{equation*}
In particular, $\si(L\si^*(M))$ is compact.
Using triangle~\eqref{eemii} we then conclude that $\se(\se^!(M))$ is compact as well. 
Further, Lemma~\ref{cprops}(1) shows that $\se^!(M) \in \bD(\csA')^\comp$ since $\se$ is fully faithful and commutes with direct sums, 
hence by induction assumption $\se^!(M) \in \bD(\csA')^\perf$. Again, it follows that
$\se(\se^!(M)) \in \bD^b(\coh(\csA))$, hence by Proposition~\ref{bperf}
\begin{equation*}
\se(\se^!(M)) \in \bD^\perf(\csA). 
\end{equation*}
Looking again at triangle~\eqref{eemii} we conclude that $M$ is perfect.
\end{proof}

\subsection{DG-enhancement}

We have an analogue of Theorem~\ref{cd-pf}.


\begin{theo}\label{cda-pf}
There is a pseudofunctor $\cD:\ASch^\opp \to \sDG$ extending the pseudofunctor of Theorem~\ref{cd-pf}, such that the diagram
\begin{equation*}
\xymatrix{
\ASch^\opp \ar[rr]^\cD \ar[dr]_{\bD} && \sDG \ar[dl]^{\bD} \\ & \Tria
}
\end{equation*}
is commutative.
\end{theo}
\begin{proof}
The proof is analogous to that of Theorem~\ref{cd-pf}.
We define
\begin{equation*}
\cD(\csA_S) \subset \hflatperf(\csA)/\hflata(\csA)
\end{equation*}
to be a small DG-subcategory such that $[\cD(\csA_S)] \cong \bD^\perf(\csA_S)$ and
such that for any morphism of $\csA$-spaces $f:(T,\csA_T) \to (S,\csA_S)$
we have $f^*(\cD(\csA_S)) \subset \cD(\csA_T)$. 
This gives a DG-functor $f^*:\cD(\csA_S) \to \cD(\csA_T)$. 
The commutativity of the diagram is proved by the arguments of Theorem~\ref{cd-pf}
and compatibility with the pseudofunctor on schemes is by construction.
\end{proof}




The following observation is very important.

\begin{theo}\label{dasm}
If the scheme $S_0$ is smooth then the category $\cdcf(\csA)$ is a smooth DG-category.
If additionally $S_0$ is proper then the category $\cdcf(\csA)$ is proper.
\end{theo}
\begin{proof}
We use induction in $n$. The case $n = 1$ is evident as $\cdcf(\csA) \cong \cdcf(S)$ in this case
and $S = S_0$ is smooth. So we assume $n > 1$. By Proposition~\ref{uniglu2} and Proposition~\ref{regluequi} 
the semiorthogonal decomposition of Proposition~\ref{dqcsod} implies that the DG-category $\cdcf(\csA)$ is quasiequivalent to the gluing
\begin{equation*}
\cdcf(\csA) \cong \cdcf(S_0)\times_\varphi \cdcf(\csA')
\end{equation*}
for appropriate bimodule $\varphi$. Moreover, the categories
$\cdcf(S_0)$ and $\cdcf(\csA')$ are smooth (the first by smoothness of $S_0$ and 
the second by the induction hypothesis), so by Proposition~\ref{utsm} we only have to check 
that the gluing bimodule $\varphi$ is perfect. 

Note that the functor $\si:\bD(S_0) \to \bD(\csA)$ commutes with direct sums, hence has 
a right adjoint functor $\si^!:\bD(\csA) \to \bD(S_0)$.
By Proposition~\ref{tv} to check that $\varphi$ is perfect it suffices to check that the functor
$L\varphi:\bD(\csA') \to \bD(S_0)$ preserves compactness. By Proposition~\ref{derglu} this functor
is isomorphic to the gluing functor $\si^!\circ\se:\bD(\csA') \to \bD(S_0)$ of the semiorthogonal decomposition.
So, we need to check that the composition $\si^!\circ\se$ preserves compactness.

Since the functor $\si$ is fully faithful and commutes with direct sums, by Lemma~\ref{cprops}(1) this is equivalent to checking that the functor
$\si\circ\si^!\circ\se:\bD(\csA') \to \bD(\csA)$ preserves compactness. Note that $\se$ preserves
compactness by Lemma~\ref{cprops}(2), since its right adjoint $\se^!$ commutes with direct sums.
Hence it suffices to check that $\si\circ\si^!:\bD(\csA) \to \bD(\csA)$ preserves compactness. 
Since an object in $\bD(\csA)$ is compact if and only if it is bounded and coherent 
(by Propositions~\ref{bperf} and~\ref{cperf}), and since the latter by definition is equivalent to boundedness 
and coherence of its image under $\Res_S$, it is enough to check that for any compact $N \in \bD(\csA)^\comp$
one has $\Res_S(\si(\si^!(N))) \in \bD^b(\coh(S))$.

Note that $\Res_S\circ\si\circ\si^!$ is right adjoint to the functor $\si\circ L\si^*\circ\LInd_S:\bD(S)\to \bD(\csA)$, which is given by
$M \mapsto (M\lotimes_{\cO_S}\csA)\lotimes_{\csA}\cO_{S_0} \cong M\lotimes_{\cO_S}\cO_{S_0}$.
It follows immediately that the functor $\Res_S\circ\si\circ\si^!$ is isomorphic to
\begin{equation*}
\Res_S(\si(\si^!(N))) = \RCHom_{\csA}(\cO_{S_0},N).
\end{equation*}
It remains to note that by Proposition~\ref{bperf} the sheaf $\cO_{S_0}$ considered as a right $\csA$-module
has a finite resolution by locally projective $\csA$-modules of finite rank, which means that
$\RCHom_{\csA}(\cO_{S_0},N)$ is bounded and coherent. And this is precisely what we had to check.
\end{proof}

\subsection{The resolution}\label{ss-rhos}

%

Let us consider one very special morphism of $\csA$-spaces, namely the morphism of $\csA$-spaces 
%
%
\begin{equation*}
(S,\fr,n) \xrightarrow{\ \rho_S } (S,0,1)
\end{equation*}
induced by the identity morphism $\id_S:S \to S$ of the scheme $S$.
In the target we just take the same underlying reducible scheme $S$ as in the source, 
while the defining ideal is taken to be $0$ and the width is taken to be $1$.
Note that $\csA_{S,0,1} = \cO_S$ as it was observed in Remark~\ref{oa}.
Consider the corresponding pullback functor $\rho_S^*:\Qcoh(S) \to \Qcoh(\csA_S)$ 
and the corresponding pushforward functor $\rho_{S*}:\Qcoh(\csA_S) \to \Qcoh(S)$. By definition~\eqref{ffpp} we have
\begin{equation}\label{auspis}
\rho_S^*(M) = M\otimes_{\cO_S}(\epsilon_1\csA),
\qquad
\rho_{S*}(N) = N\epsilon_1.
\end{equation} 

\begin{lemma}\label{rrs}
The functor $\rho_S^*$ is the left adjoint of $\rho_{S*}$ and $\rho_{S*}\circ\rho_S^*\cong \id$.
\end{lemma}
\begin{proof}
The adjunction is the particular case of Lemma~\ref{apbpf}.
%
The composition of functors is also easy to compute
\begin{equation*}
\rho_{S*}(\rho_S^*(M)) = (M\otimes_{\cO_S}(\epsilon_1\csA))\epsilon_1 = M\otimes_{\cO_S}(\epsilon_1\csA\epsilon_1) = M\otimes_{\cO_S}\cO_S = M,
\end{equation*}
so $\rho_{S*}\circ\rho_S^* = \id$.
\end{proof}


\begin{example}\label{dnrho}
Let $S = \Spec\kk[t]/t^2$. Then 
$\rho_S^*(M) = (\xymatrix@1{M \ar@<-.5ex>[r] & M/tM \ar@<-.5ex>[l]_-{t}})$.
Alternatively, it can be written as $\rho_S^*(M) = M\otimes_{\kk[t]/t^2} P_1$, 
where $P_1 = (\xymatrix@1{\kk[t]/t^2 \ar@<-.5ex>[r] & k \ar@<-.5ex>[l]_-{t}})$
is the projective module of the first vertex of the quiver.
\end{example}


In the same way as general pullback functors do the functor $\rho_S^*$ extends 
to a DG-functor $\rho_S^*:\cdcf(S) \to \cdcf(\csA_S)$
as well as 
to a functor on derived categories $L\rho_S^*:\bD(S) \to \bD(\csA_S)$.
The functor $\rho_{S*}$ is exact by~\eqref{auspis} and so automatically descends
to a functor on derived categories $\rho_{S*}:\bD(\csA_S) \to \bD(S)$.
The adjunction between $\rho_S^*$ and $\rho_{S*}$ induces an adjunction
between $L\rho_S^*$ and $\rho_{S*}$ and isomorphism of Lemma~\ref{rrs}
gives an isomorphism
%
%
%
%
\begin{equation}\label{lrrs}
\rho_{S*}\circ L\rho_S^* \cong \id_{\bD(S)}.
\end{equation}
The following result now easily follows.

\begin{theo}\label{ausres}
If the scheme $S_0$ is smooth then 
\begin{enumerate}
\item the functor $\rho_S^*:\cdcf(S) \to \cdcf(\csA_S)$ is a categorical DG-resolution;
\item the functor $L\rho_S^*:\bD(S) \to \bD(\csA_S)$ is a categorical resolution.
\end{enumerate}
In particular, the functor $\rho_{S*}$ takes $\bD^b(\coh(\csA_S))$ to $\bD^b(\coh(S))$.
Finally, if $S_0$ is proper then so is the category $\cdcf(\csA_S)$.
\end{theo}
\begin{proof}
The category $\cdcf(\csA_S)$ is smooth by Theorem~\ref{dasm}. Moreover, by~\eqref{lrrs} 
the functor $L\rho_S^*$ is fully faithful, hence $\rho_S^*$ is quasi fully faithful.
So, part $(1)$ follows.

By Proposition~\ref{dgcrcr} to deduce part (2) we only have to check that the functor $\rho_{S*}$
takes $[\cD(\csA_S)] = \bD(\csA_S)^\perf$ to $\bD^b(\coh(S))$. By Proposition~\ref{bperf}
$\bD(\csA_S)^\perf = \bD^b(\coh(\csA_S))$ and by Lemma~\ref{rfscoh} we have
$\rho_{S*}(\bD^b(\coh(\csA_S)) \subset \bD^b(\coh(S))$ which completes the proof.


Finally, the properness of $\cdcf(\csA_S)$ is also proved in Theorem~\ref{dasm}.
\end{proof}


We will also need one more observation. Assume that $f:(T,\fr_T,n_T) \to (S,\fr_S,n_S)$ is a morphism of $\csA$-spaces.
Then we have a commutative diagram 
\begin{equation*}
\xymatrix{
(T,\fr_T,n_T) \ar[d]_f \ar[r]^{\rho_T} & (T,0,1) \ar[d]^f \\
(S,\fr_S,n_S) \ar[r]^{\rho_S} & (S,0,1)
}
\end{equation*}
of morphisms of $\csA$-spaces which includes the resolution morphisms $\rho_S$ and $\rho_T$.

\begin{lemma}\label{pprho}
There is an isomorphism of DG-functors $f^*\rho_S^* \cong \rho_T^*f^*:\cdcf(S) \to \cdcf(\csA_T)$.
\end{lemma}
\begin{proof}
Follows from Lemma~\ref{ppgf}.
\end{proof}

\section{A categorical resolution of a scheme}\label{s-int}

Let $Y$ be a separated scheme of finite type over $\kk$.
Let $f:X \to Y$ be a blow up of a smooth subvariety $Z \subset Y$.
We consider an appropriate thickening $S$ of $Z$ in $Y$ and construct
a partial categorical DG-resolution of $\cdcf(Y)$ by gluing $\cdcf(\csA)$ 
(the DG-resolution of the scheme $S$ constructed in section~\ref{s-nred})
with~$\cdcf(X)$.
Then we use an inductive procedure to combine such partial resolutions into
a categorical DG-resolution of $\cdcf(Y)$.

\subsection{Nonrational centers}

For a morphism $f:X\to Y$ of schemes and a subscheme $S\subset Y$ we denote by $f^{-1}(S)$
the scheme-theoretic preimage of $S$. By definition $f^{-1}(S)$ is the subscheme of $X$
defined by the ideal $f^{-1}I_S\cdot\cO_X$, where $I_S\subset \cO_Y$ is the ideal of $S$.

If $f:X \to Y$ is a proper morphism of (possibly nonreduced and reducible) schemes we say that $f$
is birational if for any irreducible component $Y' \subset Y$ the scheme-theoretic preimage $X' := f^{-1}(Y')$
is irreducible and the restriction of $f$, $X'_\red \to Y'_\red$ is birational.

\begin{defi}
Let $f:X \to Y$ be a proper birational morphism.
A subscheme $S \subset Y$ is called {\sf a nonrational center of $Y$ with respect to $f$}
if the canonical morphism
$$
I_S \to Rf_*(I_{f^{-1}(S)})
$$
is an isomorphism.
\end{defi}

\begin{remark}
It is easy to see that if $Y$ is integral and $f:X \to Y$ is a resolution of singularities then the empty subscheme 
is a nonrational center for $f:X\to Y$ if and only if $Y$ has rational singularities. Thus if $S$ is a nonrational 
center for $f:X\to Y$ and $X$ is smooth then $Y \setminus S$ has rational singularities. This justifies the name.
%
\end{remark}

The following lemma shows that a nonrational center exists in a pretty large generality.

\begin{lemma}\label{innrc}
Assume $f:X \to Y$ is a blow up of a sheaf of ideals $I$ on $Y$.
Then for $n \gg 0$ the subscheme of $Y$ corresponding to the ideal $I^n$ is a nonrational center with respect to $f$.
\end{lemma}
\begin{proof}
Recall that $X$ is the projective spectrum of the sheaf of graded algebras $\oplus_{k=0}^\infty I^k$ on $Y$.
Further, it is clear that the graded sheaf of modules corresponding to the sheaf $f^{-1}I^n\cdot\cO_X$
is $\oplus_{k=0}^\infty I^{k+n}$ which is nothing but the line bundle $\cO_{X/Y}(n)$.
On the other hand, for $n \gg 0$ we have $Rf_*\cO_{X/Y}(n) = I^n$, see e.g.~\cite[Ex.~II.5.9]{Ha}.
\end{proof}

Consider the Cartesian square
\begin{equation}\label{xyts}
\vcenter{\xymatrix{
X \ar[d]_f & T \ar[l]_j \ar[d]^p \\ Y & S \ar[l]_i
}}
\end{equation} 
where $S$ is a subscheme of $Y$, $T = f^{-1}(S)$ is its scheme-theoretic preimage,
$i$ and $j$ are the closed embeddings, and $p$ is the restriction of $f$ to $T$.

\begin{lemma}\label{nrc-tri}
If $S$ is a nonrational center for the morphism $f$ and $T = f^{-1}(S)$ is its scheme-theoretic preimage
then there is a distinguished triangle in $\bD(Y)$
\begin{equation}\label{triy}
\cO_Y \to Rf_*\cO_X \oplus i_*\cO_S \to Rf_*j_*\cO_T.
\end{equation}
\end{lemma}
\begin{proof}
We have a commutative diagram
\begin{equation*}
\xymatrix{
i_*\cO_S[-1] \ar@{..>}[drr] \ar[r] \ar[d] & I_S \ar[r] \ar@{=}[d] & \cO_Y \ar[r] \ar[d] & i_*\cO_S \ar[d] \\
Rf_*j_*\cO_T[-1] \ar[r] & Rf_*I_T \ar[r] & Rf_*\cO_X \ar[r] & Rf_*j_*\cO_T
}
\end{equation*}
with rows being distinguished triangles. 
It follows that the dotted arrow is zero.
Extending the commutative square marked with $\star$ to the diagram of the octahedron axiom
\begin{equation*}
\xymatrix{ 
i_*\cO_S[-1] \ar@{}[dr]|-\bigstar \ar[r] \ar@{=}[d] & I_S \ar[d] \ar[r] & \cO_Y \ar[d] \\
i_*\cO_S[-1] \ar[r]_-0 & Rf_*\cO_X \ar[r] \ar[d] & Rf_*\cO_X \oplus i_*\cO_S \ar[d] \\
& Rf_*j_*\cO_T \ar@{=}[r] & Rf_*j_*\cO_T
}
\end{equation*}
we see the required triangle~\eqref{triy} as its right column.
%
\end{proof}

\subsection{The naive gluing}

Let $f:X \to Y$ be a blowup with smooth center $Z \subset Y$ and let $S$ be a nonrational center for~$f$.
Let $T$ be the scheme-theoretic preimage of $S$ and consider the diagram~\eqref{xyts}. 
This diagram allows to define a gluing of $\cdcf(S)$ with $\cdcf(X)$ by using $\cdcf(T)$ 
to construct the gluing DG-bimodule. To be more precise we define a DG-bimodule
$\varphi_0 \in \cdcf(X)^\op \otimes \cdcf(S)\dgm$ by
\begin{equation}\label{phi0}
\varphi_0(G_X,G_S) = \Hom_{\cdcf(T)}(p^*G_S,j^*G_T)
\end{equation}
with the bimodule structure given by DG-functors $j^*$ and $p^*$,
and consider the gluing
\begin{equation*}
\cD_0 := \cdcf(S) \times_{\varphi_0} \cdcf(X).
\end{equation*}
The key observation is that $\cD_0$ provides a partial categorical DG-resolution of $\cdcf(Y)$ as soon as $S$ is a nonrational center for $f$.
Thus adding the category of a nonrational center to the category of the blowup allows to cure nonrationality of the singularity.

Indeed, let us construct a DG-functor from $\cdcf(Y)$ to $\cD_0$. Let
\begin{equation*}
t = f\circ j = i\circ p:T \to Y
\end{equation*}
and note that for any $F \in \cdcf(Y)$ we have
\begin{equation*}
\varphi_0(f^*F,i^*F) = \Hom_{\cdcf(T)}(p^*i^*F,j^*f^*F) = \Hom_{\cdcf(T)}(t^*F,t^*F).
\end{equation*}
This allows to define
\begin{equation}
\mu_F := 1_{t^*F} \in \Hom_{\cdcf(T)}(t^*F,t^*F) = \varphi_0(f^*F,i^*F).
\end{equation} 
By construction $\mu_F$ is closed of degree zero, so $(i^*F,f^*F,\mu_F)$ is a well defined object of the gluing $\cD_0$.
Moreover, it is clear that
\begin{equation}
\pi_0(F) := (i^*F,f^*F,\mu_F)
\end{equation} 
is a well defined DG-functor $\pi_0:\cdcf(Y) \to \cD_0$.
We denote by $\pi_0^*$ its extension to derived categories:
\begin{equation*}
\pi_0^* = \LInd_{\pi_0}:\bD(Y) \to \bD(\cD_0).
\end{equation*}

\begin{prop}\label{pisfai}
The functor $\pi_0^*:\bD(Y) \to \bD(\cD_0)$ is fully faithful.
\end{prop}
\begin{proof}
By Proposition~\ref{fpb} it suffices to check that $\pi_0$ is quasi fully faithful. 
According to Remark~\ref{triglu} we have a distinguished triangle
\begin{equation*}
\Hom_{\cD_0}(\pi_0(F),\pi_0(G)) 
\xrightarrow{\qquad} 
\begin{array}{c}
\Hom_{\cdcf(S)}(i^*F,i^*G) \\
\oplus \\
\Hom_{\cdcf(X)}(f^*F,f^*G) 
\end{array}
\xrightarrow[\mu_F]{\ -\mu_G\ } 
\varphi_0(f^*G,i^*F).
\end{equation*}
Since the DG-bimodule structure on $\varphi_0$ is given by the pullback functors $p^*$ and $j^*$
respectively and since the elements $\mu_F$ and $\mu_G$ are given by the units, we can rewrite this triangle as
\begin{equation*}
\Hom_{\cD_0}(\pi_0(F),\pi_0(G)) 
\xrightarrow{\qquad} 
\begin{array}{c}
\Hom_{\cdcf(S)}(i^*F,i^*G) \\
\oplus \\
\Hom_{\cdcf(X)}(f^*F,f^*G) 
\end{array}
\xrightarrow[j^*]{\ -p^*\ } 
\Hom_{\cdcf(T)}(t^*F,t^*G).
\end{equation*}
Note that the composition of the first arrow of this triangle with the action of the DG-functor 
$\pi_0:\Hom_{\cdcf(Y)}(F,G) \to \Hom_{\cD_0}(\pi_0(F),\pi_0(G))$ 
is given by DG-functors $i^*$ and $f^*$, so to prove that $\pi_0$ is quasi fully faithful it is enough to check that the following
triangle 
\begin{equation}\label{tridi}
\Hom_{\cdcf(Y)}(F,G) 
\xrightarrow[f^*]{\ i^*\ } 
\begin{array}{c}
\Hom_{\cdcf(S)}(i^*F,i^*G) \\
\oplus \\
\Hom_{\cdcf(X)}(f^*F,f^*G) 
\end{array}
\xrightarrow[j^*]{\ -p^*\ } 
\Hom_{\cdcf(T)}(t^*F,t^*G)
\end{equation}
is distinguished.
%
%
%
%
%
%
%

For this we identify the vertices of the triangle in terms of $Y$. 
Indeed, these vertices compute $\Ext$ groups in the derived categories 
$\Ext^\bullet_{\bD(Y)}(F,G)$, 
$\Ext^\bullet_{\bD(S)}(Li^*F,Li^*G)$, 
$\Ext^\bullet_{\bD(X)}(Lf^*F,Lf^*G)$, and
$\Ext^\bullet_{\bD(T)}(Lt^*F,Lt^*G) \cong \Ext^\bullet_{\bD(T)}(Lj^*Lf^*F,Lj^*Lf^*G)$ respectively.
Using 
the pullback--pushforward adjunctions
and the projection formula we rewrite
\begin{equation*}
\Ext^\bullet_{\bD(S)}(Li^*F,Li^*G) \cong 
\Ext^\bullet_{\bD(Y)}(F,i_*Li^*G) \cong
\Ext^\bullet_{\bD(Y)}(F,G\lotimes i_*\cO_S)
\end{equation*}
for the first summand of the second vertex,
\begin{equation*}
\Ext^\bullet_{\bD(X)}(Lf^*F,Lf^*G) \cong 
\Ext^\bullet_{\bD(Y)}(F,Rf_*Lf^*G) \cong
\Ext^\bullet_{\bD(Y)}(F,G\lotimes Rf_*\cO_X),
\end{equation*}
for the second summand of the second vertex, and
\begin{equation*}
\Ext^\bullet_{\bD(T)}(Lj^*Lf^*F,Lj^*Lf^*G) \cong 
\Ext^\bullet_{\bD(Y)}(F,Rf_*j_*Lj^*Lf^*G) \cong
\Ext^\bullet_{\bD(Y)}(F,G\lotimes Rf_*j_*\cO_T)
\end{equation*}
for the third vertex. It is clear that the morphism from the first vertex 
of~\eqref{tridi} to the summands of the second correspond
under above isomorphisms to the maps obtained by tensoring morphisms $\cO_Y \to i_*\cO_S$
and $\cO_Y \to Rf_*\cO_X$ with $G$ and applying $\Ext_{\bD(Y)}^\bullet(F,-)$. Analogously, 
the morphisms from the summands of the second vertex to the third
are obtained by the tensoring morphisms $i_*\cO_S \to i_*Rp_*\cO_T = Rf_*j_*\cO_T$
and $Rf_*\cO_X \to Rf_*j_*\cO_T$ with $G$ and applying $\Ext_{\bD(Y)}^\bullet(F,-)$. 
Therefore the whole triangle~\eqref{tridi} is obtained by tensoring with $G$
the triangle
\begin{equation*}
\cO_Y \to i_*\cO_S \oplus Rf_*\cO_X \to Rf_*j_*\cO_T
\end{equation*}
and applying $\Ext_{\bD(Y)}^\bullet(F,-)$. The above triangle is distinguished by Lemma~\ref{nrc-tri}
which implies that the triangle~\eqref{tridi} is distinguished as well.
%
%
%
\end{proof}

Note that by Brown representability the functor $\pi_0^*$ 
has a right adjoint which we denote by 
\begin{equation*}
\pi_{0*} = \Res_{\pi_0}:\bD(\cD_0) \to \bD(Y). 
\end{equation*}
Recall that the derived category $\bD(\cD_0)$ of the gluing $\cD_0$ has a semiorthogonal decomposition
$\bD(\cD_0) = \langle \bD(S),\bD(X) \rangle$, see section~\ref{ss-dglu}. In particular, we have 
embedding functors $I_1:\bD(S) \to  \bD(\cD_0)$ and $I_2:\bD(X) \to \bD(\cD_0)$ commuting with 
arbitrary direct sums. We will implicitly identify the categories $\bD(S)$ and $\bD(X)$
with their images in $\bD(\cD_0)$.

%

\begin{prop}\label{pils}
The functor $\pi_{0*}$ restricted to $\bD(S) \subset \bD(\cD_0)$ is isomorphic to the pushforward $i_*:\bD(S) \to \bD(Y)$.
The functor $\pi_{0*}$ restricted to $\bD(X) \subset \bD(\cD_0)$ is isomorphic 
to the composition $Rf_*(-\lotimes I_T):\bD(X) \to \bD(Y)$. In particular,
$\pi_{0*}$ takes both $\bD^b(\coh(S))$ and $\bD^b(\coh(X))$ to $\bD^b(\coh(Y))$.
\end{prop}
\begin{proof}
The functor $\pi_{0*} = \Res_{\pi_0}$ commutes with arbitrary direct sums by Proposition~\ref{fpb}.
The same is true for the embedding functors $\bD(S) \to  \bD(\cD_0)$ and $\bD(X) \to \bD(\cD_0)$.
Hence the restrictions of the functor $\pi_{0*}$ onto $\bD(S)$ and $\bD(X)$ commute with arbitrary direct sums. 
By \cite[Thm.~7.2.]{T07} it is enough to prove the claim only for objects of subcategories 
$[\cdcf(S)] \subset \bD(S)$ and $[\cdcf(X)] \subset \bD(X)$ respectively.

So, take arbitrary $F \in \cdcf(Y)$ and $G_S \in \cdcf(S)$. Since $\pi_{0*}$ is the right adjoint of $\pi_0^*$
which coincides with $\pi_0$ on $\cdcf(Y)$ we have 
$\Hom_{\bD(Y)}(F,\pi_{0*}(G_S,0,0)) \cong \Hom_{\bD(\cD_0)}(\pi_0F,(G_S,0,0))$ which is just the cohomology of the complex
\begin{equation*}
\Hom_{\cD_0}(\pi_0(F),(G_S,0,0)) =
\Hom_{\cD_0}((i^*F,f^*F,\mu_F),(G_S,0,0)) =
\Hom_{\cdcf(S)}(i^*F,G_S)
\end{equation*}
that is nothing but
$\Hom_{\bD(S)}(Li^*F,G_S) \cong \Hom_{\bD(Y)}(F,i_*G_S)$
which proves the first claim.

Now let $G_X \in \cdcf(X)$. Then
$\Hom_{\bD(Y)}(F,\pi_{0*}(0,G_X,0)) \cong \Hom_{\bD(\cD_0)}(\pi_0F,(0,G_X,0))$ which is just the cohomology of the complex
$\Hom_{\cD_0}(\pi_0(F),(0,G_X,0))$ which by Remark~\ref{triglu} fits into a distinguished triangle
\begin{equation*}
\Hom_{\cD_0}(\pi_0(F),(0,G_X,0)) \xrightarrow{\quad}
\Hom_{\cdcf(X)}(f^*F,G_X) \xrightarrow{\ \mu_F\ }
\varphi_0(G_X,i^*F).
\end{equation*}
Further the argument goes along the lines of that of Proposition~\ref{pisfai}.
First, we rewrite the triangle as
\begin{equation}\label{tripix}
\Hom_{\cD_0}(\pi_0(F),(0,G_X,0)) \xrightarrow{\ }
\Hom_{\cdcf(X)}(f^*F,G_X) \xrightarrow{j^*}
\Hom_{\cdcf(T)}(j^*f^*F,j^*G_X).
\end{equation}
Further we identify the second vertex with
%
$\Ext^\bullet_{\bD(X)}(Lf^*F,G_X) \cong \Ext^\bullet_{\bD(Y)}(F,Rf_*G_X)$, the third vertex with
\begin{equation*}
\Ext^\bullet_{\bD(T)}(Lj^*Lf^*F,Lj^*G_X) \cong 
\Ext^\bullet_{\bD(Y)}(F,Rf_*j_*Lj^*G_X) \cong 
\Ext^\bullet_{\bD(Y)}(F,Rf_*(G_X\lotimes j_*\cO_T)),
\end{equation*}
and the map between them with the map obtained by tensoring morphism $\cO_X \to j_*\cO_T$ with $G_X$
and applying $\Ext_{\bD(Y)}^\bullet(F,Rf_*(-))$. Looking at the triangle
\begin{equation*}
I_T \to \cO_X \to j_*\cO_T
\end{equation*}
we see that the first term of~\eqref{tripix} is quasiisomorphic to $\Ext^\bullet_{\bD(Y)}(F,Rf_*(G_X\lotimes I_T))$
which proves the second claim of the Proposition.


For the last claim note that the functors 
$i_*$ and $Rf_*$ preserve boundedness and coherence since the morphisms $i$ and $f$ are proper, 
and  the same is true for the functor $-\lotimes I_T$ because $I_T \cong \cO_{X/Y}(1)$ is a line bundle.
\end{proof}

We know by Proposition~\ref{pisfai} that $\cD_0$ is a partial categorical DG-resolution of $\cdcf(Y)$.
It is not a resolution yet since both components $\cdcf(S)$ and $\cdcf(X)$ of $\cD_0$
are not smooth in general. Of course we can resolve $\cdcf(S)$ by appropriate Auslander algebra $\csA_S$
as discussed in section~\ref{s-nred} and we can apply the induction to resolve $\cdcf(X)$ by appropriate
smooth DG-category $\cD'$. Then we can apply the regluing procedure to obtain a (partial) resolution of $\cdcf(Y)$
by $\cdcf(\csA_S)\times_{\tilde\varphi_0}\cD'$ and to check that it is smooth we will have to check that
the induced bimodule $\tilde\varphi_0$ is perfect. It turns out, however, that this is not the case.
The reason for this is that the pullback functor $L\rho_S^*:\bD(S) \to \bD(\csA_S)$ does not preserve
boundedness (for details one can see the proof of Theorem~\ref{mainind} below). So, to get rid of this problem
we replace the naive gluing $\cD_0 = \cdcf(S)\times_{\varphi_0}\cdcf(X)$ with the gluing of $\cdcf(\csA_S)$
and $\cdcf(X)$ which is not the regluing of $\cD_0$ (see Remark~\ref{ff0}).

\subsection{The gluing which works}

Again, let $f:X \to Y$ be a blowup with smooth center $Z \subset Y$. Denote $I = I_Z$. 
Let $S$ be a nonrational center for~$f$ such that $I_S = I^n$ (such $n$ exists by Lemma~\ref{innrc}). 
Let $\fr_S = I/I^n \subset \cO_S$ and consider the $\csA$-space $(S,\fr,n)$. 

Consider also the ideal $\cJ = f^{-1}I\cdot\cO_X$. Then by definition of the blowup this is an invertible sheaf of ideals,
$\cJ = \cO_{X/Y}(1)$, and $I_T = \cJ^n$, where $T =f^{-1}(S)$. We also put $\fr_T = \cJ/\cJ^n$ and $n_T = n_S = n$.
Then the morphism $p:T \to S$ gives a morphism $(T,\fr_T,n) \to (S,\fr_S,n)$ of $\csA$-spaces which we denote by $\barp$.
So, we have a commutative diagram
\begin{equation}\label{xyats}
\vcenter{\xymatrix{
X \ar[d]_f & T \ar[l]_j \ar[d]^p & (T,\csA_T) 
\ar[l]_{\rho_T} \ar[d]^{\barp} \\ Y & S \ar[l]_i & (S,\csA_S) \ar[l]_{\rho_S}
}}
\end{equation} 
of $\csA$-spaces. Now we replace $\cdcf(S)$ by $\cdcf(\csA_S)$ and $\cdcf(T)$ by $\cdcf(\csA_T)$ in the construction
of the previous section and show that the corresponding gluing works.

%

%
%
%

Explicitly, we define the gluing DG-bimodule 
$\varphi \in \cdcf(X)^\op\otimes\cdcf(\csA_S)\dgm$
as
\begin{equation}\label{defbfi}
\varphi(G_X,G_S) := \Hom_{\cdcf(\csA_T)}(\barp^*G_S,\rho_T^*j^*G_X)
\end{equation} 
for all $G_S \in \cdcf(\csA_S)$, $G_X \in \cdcf(X)$. 
Consider DG-functor $\tau_1 := \rho_S:\cdcf(S) \to \cdcf(\csA_S)$ defined in section~\ref{ss-rhos} 
and the identity DG-functor $\tau_2 := \id:\cdcf(X) \to \cdcf(X)$.

\begin{lemma}\label{ff0comp}
The bimodules $\varphi_0 \in \cdcf(X)^\op\otimes\cdcf(S)\dgm$ and $\varphi \in \cdcf(X)^\op\otimes\cdcf(\csA_S)\dgm$
are compatible in the sense of section~\ref{sreglu}.
\end{lemma}
\begin{proof}
Indeed,
\begin{multline*}
\varphi(G_X,\rho_S^*G_S) = 
\Hom_{\cdcf(\csA_T)}(\barp^*\rho_S^*G_S,\rho_T^*j^*G_X) \cong
\Hom_{\cdcf(\csA_T)}(\rho_T^*p^*G_S,\rho_T^*j^*G_X) \cong 
\\ \cong
\Hom_{\cdcf(T)}(p^*G_S,j^*G_X) =
\varphi_0(G_X,G_S),
\end{multline*}
the first is the definition of $\varphi$, the second is the commutativity of the right square of~\eqref{xyats},
the third is quasi full faithfulness of $\rho^*_T$, and the fourth is the definition of $\varphi_0$.
\end{proof}

Consider the gluing DG-category
\begin{equation}\label{defcd}
\cD = \cdcf(\csA_S)\times_\varphi \cdcf(X).
\end{equation} 
By Proposition~\ref{compglu} we have a quasi fully faithful DG-functor $\tau:\cD_0 \to \cD$
which takes $(G_S,G_X,\mu)$ to $(\rho_S^*(G_S),G_X,c(\mu))$ where $c$ is the quasiisomorphism
of Lemma~\ref{ff0comp}. Composing it with the DG-functor $\pi_0:\cdcf(Y) \to \cD_0$ we obtain 
a DG-functor 
\begin{equation*}
\pi = \tau\circ\pi_0:\cdcf(Y) \to \cD,
\qquad
\pi(F) = (\rho_S^*i^*F,f^*F,c(\mu_F)).
\end{equation*}

\begin{theo}\label{mainres}
Let $f:X \to Y$ be the blowing up of an ideal $I$ and $S$ be the subscheme of $Y$ defined by a power
of the ideal $I$ which is a nonrational center for $f$. Then the DG-functor
$\pi:\cdcf(Y) \to \cD = \cdcf(\csA_S)\times_\varphi\cdcf(X)$
is a partial categorical DG-resolution of singularities. Moreover,
the functor $\pi_* = \Res_\pi:\bD(\cD) \to \bD(Y)$ takes both $\bD^b(\coh(\csA_S))$ and $\bD^b(\coh(X))$
to $\bD^b(\coh(Y))$.
\end{theo}
\begin{proof}
The DG-functor $\pi$ is a composition of the DG-functor $\pi_0$ which is quasi fully faithful by Proposition~\ref{pisfai}
and of the DG-functor $\tau:\cD_0 \to \cD$ which is quasi fully faithful by Proposition~\ref{compglu}. Hence $\pi$ is quasi
fully faithful, and so it is a partial categorical DG-resolution. Further, the functor $\Res_\pi$ is isomorphic to the composition
$\Res_{\pi_{0}}\circ\Res_\tau$. By Proposition~\ref{compglu}(d) the functor $\Res_\tau$ on the component $\bD(\csA_S)$ equals $\rho_{S*}$,
hence takes $\bD^b(\coh(\csA_S))$ to $\bD^b(\coh(S))$ by Theorem~\ref{ausres}. Similarly, by Proposition~\ref{compglu}(e) 
(see also Remark~\ref{parte}) the functor $\Res_\tau$ on the component $\bD(X)$ is the identity. By Proposition~\ref{pils}
the functor $\Res_{\pi_{0}}$ take both $\bD^b(\coh(S))$ and $\bD^b(\coh(X))$ to $\bD^b(\coh(Y))$, so the claim follows.
\end{proof}

The following property of the gluing bimodule $\varphi$ is crucial.

\begin{lemma}\label{phicoh}
The functor $L\varphi:\bD(X) \to \bD(\csA_S)$ takes $\bD^b(\coh(X))$ to $\bD^b(\coh(\csA_S))$.
\end{lemma}
\begin{proof}
By definition~\eqref{lphi} the functor $L\varphi$ is the functor of derived tensor product over $\cD(X)$ with $\varphi$.
On the other hand, by~\eqref{defbfi} we have $\varphi = {}_{\rho_T^*j^*}\cD(\csA_T)_{\barp^*}$,
the restriction of the diagonal bimodule over $\cD(\csA_T)$. Hence
\begin{equation*}
L\varphi(M) = 
M\lotimes_{\cD(X)}\varphi =
M\lotimes_{\cD(X)}{}_{\rho_T^*j^*}\cD(\csA_T)_{\barp^*} =
\Res_{\barp^*}(\LInd_{\rho_T^*j^*}(M)).
\end{equation*}
The derived induction functor $\LInd_{\rho_T^*j^*}$ is isomorphic to the composition of the derived pullback functors
$L\rho_T^*Lj^*$ by Theorem~\ref{cda-pf}. The restriction functor $\Res_{\barp^*}$ is right adjoint 
to the derived induction functor $\LInd_{\barp^*}$ which by the same Theorem is isomorphic to $L\barp^*$.
Hence $\Res_{\barp^*} \cong R\barp_*$. Thus
\begin{equation*}
L\varphi \cong R\barp_*\circ L\rho_T^* \circ Lj^*:\bD(X) \to  \bD(\csA_S).
\end{equation*}
%
%
%
The composition of functors $L\rho_T^*\circ Lj^*$ is given by $F \mapsto F\lotimes_{\cO_X}\csA_T$.
Note that all ideals $\cJ^k \subset \cO_X$ are invertible, hence all the components $\cJ^k/\cJ^l$
of the sheaf of algebras $\csA_T$ are perfect $\cO_X$-modules. Therefore the functor
$L\rho_T^*\circ Lj^*$ preserves both boundedness and coherence of sheaves.
On the other hand, the functor $R\barp_*$ preserves boundedness and coherence by Lemma~\ref{rfscoh}.
The Lemma follows.
%
\end{proof}

\begin{remark}\label{ff0}
Now it is already clear that the gluing $\cD$ is not the regluing of $\cD_0$.
Indeed, if it were so then we would have an isomorphism $L\varphi \cong L\rho_S^*\circ L\varphi_0$.
But it is easy to find an object $M \in \bD^b(\coh(X))$ such that $L\varphi_0(M)$ is not perfect
and its pullback under $L\rho_S^*$ is unbounded.
\end{remark}

%
%
%

\subsection{The inductive construction of a categorical DG-resolution}

%
%
%

Now finally we are ready to prove the main result of the paper, Theorem~\ref{cr}.

We start with a separated scheme $Y$ of finite type over a field $\kk$ of characteristic $0$.
By Theorem~1.6 of \cite{BM97} there exists a chain 
\begin{equation}\label{frs}
\vcenter{\xymatrix{
Y_m \ar[r] & Y_{m-1} \ar[r] & \dots \ar[r] & Y_1 \ar[r] & Y_0 \ar@{=}[r] & Y \\
& Z_{m-1} \ar@{^{(}->}[u] & & Z_{1} \ar@{^{(}->}[u] & Z_{0} \ar@{^{(}->}[u] 
}}
\end{equation}
of blowups with smooth centers $Z_i$ such that $(Y_m)_\red$ is smooth. 
Indeed, the following Lemma shows that the usual resolution of the reduced part of $Y$ 
by a chain of smooth blowups does the job.


\begin{lemma}
Let $Y$ be a possibly nonreduced scheme and $Z \subset Y_\red$ a smooth subscheme of its reduced part.
Then $(\BL_Z Y)_\red = \BL_Z(Y_\red)$, the reduced part of the blowup equals the blowup of the reduced part.
\end{lemma}
\begin{proof}
The claim is local with respect to $Y$ so we can assume that $Y = \Spec R$ is affine.
Let $I \subset R$ be the ideal of $Z$ and let $\fr \subset R$ be the nilpotent radical of $R$. 
By assumption $\fr \subset I$. By definition of a blowup
we have $\BL_Z Y = \Proj(\oplus I^k)$. On the other hand, the ideal of $Z$ on $Y_\red = \Spec (R/\fr)$ is $I/\fr$, hence
$\BL_Z(Y_\red) = \Proj(\oplus I^k/(I^k\cap\fr))$. The natural epimorphism of graded algebras
\begin{equation*}
\oplus I^k \xrightarrow{\ \ \ } \oplus I^k/(I^k\cap\fr)
\end{equation*}
shows that $\BL_Z(Y_\red)$ is a closed reduced subscheme of $\BL_Z(Y)$. Its ideal is given by
$\oplus (I^k \cap \fr)$, hence is nilpotent. But a reduced subscheme with nilpotent ideal 
is nothing but the reduced part of the scheme, hence the claim.
%
%
%
\end{proof}


As we already know by Theorem~\ref{mainres} one can construct 
a partial categorical DG-resolution of $\cdcf(Y)$ by appropriate gluing of $\cdcf(Y_1)$ 
with $\cdcf(\csA_S)$, where $S$ is an appropriate thickening of $Z_0$. Moreover, the component
$\cdcf(\csA_S)$ of this gluing is smooth by Theorem~\ref{ausres} and one can also check that
the gluing bimodule is perfect. So, all nonsmoothness of the gluing comes from nonsmoothness 
of the scheme $Y_1$. Thus, to obtain a resolution we only have to replace (by using the regluing procedure) 
the component $\cdcf(Y_1)$ by its categorical resolution. This shows that one can use induction to construct the resolution.
In fact, it turns out that technically it is much more convenient to include in the induction
hypotheses some properties of the resolution as well. So, we state the following

\begin{theo}\label{mainind}
There is a small pretriangulated DG-category $\cD$ glued from several copies of $\cdcf(Y_m)$ and several copies of $\cdcf(Z_i)$ for $0\le i\le m-1$, 
and a DG-functor $\pi:\cdcf(Y) \to \cD$ such that
\begin{enumerate}
\item $\cD$ is a categorical DG-resolution of $\cdcf(Y)$;
\item the functor $\pi_*=\Res_\pi:\bD(\cD) \to \bD(Y)$ takes $[\cD] = \bD(\cD)^\comp$ to $\bD^b(\coh(Y))$;
\item if $Y$ is proper then so is $\cD$.
\end{enumerate}
\end{theo}
%
%
%
\begin{proof}
We fix a chain of blowups~\eqref{frs} and use induction on $m$.
If $m = 0$ then the scheme $Y_\red$ is smooth, so $\cD = \cdcf(\csA_Y)$ with $\pi = \rho_Y$ 
provide a categorical DG-resolution of $\cdcf(Y)$ by Theorem~\ref{ausres} and moreover 
the properties (2) and (3) are satisfied.

Assume that $m > 0$. In this case we let $X = Y_1$, $Z = Z_0$ and denote the blowup map by $f:X \to Y$.
Since $f$ is a blowup of a subscheme $Z \subset Y$, by Lemma~\ref{innrc} an appropriate infinitesimal neighborhood 
$S$ of $Z$ (i.e. the subscheme given by the ideal $I_Z^n$ for some $n$, where $I_Z \subset \cO_Y$ is the ideal of~$Z$)
is a nonrational center for $f$. Therefore Theorem~\ref{mainres} applies 
and we have a partial categorical DG-resolution which we denote here by
\begin{equation*}
\pi_1:\cdcf(Y) \to \cdcf(\csA_S)\times_\varphi\cdcf(X)
\end{equation*}
with bimodule $\varphi$ defined by~\eqref{defbfi}.
Since the scheme $X = Y_1$ can be resolved by $m-1$ smooth blowups the induction applies,
hence  we have a categorical DG-resolution 
\begin{equation*}
\tau_2:\cdcf(X) \to \cD_2.
\end{equation*}
Taking $\tau_1 = \id:\cdcf(\csA_S) \to \cdcf(\csA_S)$ and applying the regluing procedure (see section~\ref{sreglu}) 
we obtain a partial categorical DG-resolution
\begin{equation}
\tau:\cdcf(\csA_S) \times_\varphi \cdcf(X) \to \cD(\csA_S)\times_{\tilde\varphi}\cD_2
\end{equation} 
for appropriate bimodule $\tilde\varphi$. By Lemma~\ref{compres} the composition
\begin{equation*}
\pi = \tau\circ\pi_1:\cdcf(Y) \to \cD := \cdcf(\csA_S)\times_{\tilde\varphi}\cD_2
\end{equation*}
is a partial categorical DG-resolution, so the only thing to check for part (1) 
is that the category $\cD$ is smooth.

For this we note that $\cD_2$ is smooth by part (1) of the induction hypothesis and $\cD(\csA_S)$ is smooth
by Theorem~\ref{ausres}. So, by Proposition~\ref{utsm} it remains to check that the bimodule $\tilde\varphi$ is perfect.
For this we use Proposition~\ref{tv}. Since $\cD_2$ is smooth it suffices to check
that the functor $L\tilde\varphi:\bD(\cD_2) \to \bD(\csA_S)$ induced by the bimodule $\tilde\varphi$ takes $[\cD_2]$ 
to $\bD(\csA_S)^\comp$ which by Propositions~\ref{cperf} and~\ref{bperf} is $\bD^b(\coh(\csA_S))$.
For this we note that by Proposition~\ref{reglu} this functor is isomorphic to the composition
\begin{equation*}
\bD(\cD_2) \xrightarrow{\ \Res_{\tau_2}\ } \bD(X) \xrightarrow{\ L\varphi\ } \bD(\csA_S). 
\end{equation*}
The functor
$\Res_{\tau_2}$ takes $[\cD_2]$ to $\bD^b(\coh(X))$ by part (2) of the induction hypothesis 
and the functor $L\varphi$ takes $\bD^b(\coh(X))$ to $\bD^b(\coh(\csA_S))$ by Lemma~\ref{phicoh}.
Thus we have proved part~(1).



Further, 
we have to check that the pushforward functor $\Res_\pi:\bD(\cD) \to \bD(Y)$
takes $[\cD]$ to $\bD^b(\coh(Y))$. Since each object of $[\cD]$ sits in a triangle
with the other vertices in $\cdcf(\csA_S)$ and $[\cD_2]$ it suffices to check 
the statement for those. Let $F \in [\cdcf(\csA_S)]$. By Proposition~\ref{reglu}(b)
we have $\Res_{\tau}(F) = \Res_{\tau_1}(F) = F$ since $\tau_1 = \id$, hence
\begin{equation*}
\Res_{\pi}(F) = \Res_{\pi_1}(\Res_\tau(F)) = \Res_{\pi_1}(F) \in \bD^b(\coh(Y)) 
\end{equation*}
by Proposition~\ref{pils}.
On the other hand, let $F \in [\cD_2]$. Then again by Proposition~\ref{reglu}(b) we have
$\Res_\tau(F) = \Res_{\tau_2}(F)$ and this is in $\bD^b(\coh(X))$ by the induction hypothesis.
Then $\Res_\pi(F) = \Res_{\pi_1}(\Res_{\tau_2}(F))$ is in $\bD^b(\coh(Y))$ again by Proposition~\ref{pils}.
This completes part (2).

Finally, we have to check that $\cD$ is proper as soon as $Y$ is. For this we note that the first center $Z_0$
and the first blowup $X = Y_1$ are both proper. Hence by induction hypothesis the resolution $\cD_2$ of $\cdcf(X)$ is proper.
On the other hand, the category $\cdcf(\csA_S)$ is proper by Theorem~\ref{ausres}. So, as we have already seen
that the gluing bimodule $\tilde\varphi$ is perfect, Proposition~\ref{utsm} applies and we conclude that $\cD$ is proper.
\end{proof}

Now we are ready to prove Theorem~\ref{cr}. Indeed, we take $\cT = \bD(\cD)$
and apply Theorem~\ref{mainind} together with Proposition~\ref{dgcrcr}.
The semiorthogonal decompositions of $\cT$ and $\cT^\comp$ are given by 
(iterations of) Corollary~\ref{sodglu} and Proposition~\ref{derglu}:
\begin{equation}\label{sods}
\begin{array}{rcl}
\bD(\cD) & = & \langle 
\underbrace{\bD(Z_0),\dots,\bD(Z_0)}_{\text{$n_0$ times}},
\dots,\\
&& \qquad\qquad
\underbrace{\bD(Z_{m-1}),\dots,\bD(Z_{m-1})}_{\text{$n_{m-1}$ times}},
\underbrace{\bD((Y_{m})_\red),\dots,\bD((Y_{m})_\red)}_{\text{$n_{m}$ times}} \rangle \\
\bD(\cD)^\comp & = & \langle \underbrace{\bD^b(\coh(Z_0)),\dots,\bD^b(\coh(Z_0))}_{\text{$n_0$ times}},\dots,\\
&& \qquad\qquad
\underbrace{\bD^b(\coh(Z_{m-1})),\dots,\bD^b(\coh(Z_{m-1}))}_{\text{$n_{m-1}$ times}},\\
&& \qquad\qquad\qquad\qquad
\underbrace{\bD^b(\coh((Y_m)_\red)),\dots,\bD^b(\coh((Y_m)_\red))}_{\text{$n_m$ times}}, \rangle 
\end{array}
\end{equation}
Here $n_i$ for $0 \le i \le m-1$ denotes the power of the ideal $I_{Z_i}$ which gives $I_{S_i}$, the ideal
of the nonrational center at step $i$, and $n_m$ stands for the nilpotence degree of the nilradical
of the scheme $Y_m$.

\subsection{Properties of the resolution}\label{s-props}

In this section we discuss some properties of the categorical resolution constructed above.


For each open subset $U \subset Y$ we consider the DG-resolution of $U$ obtained by a base change from the diagram~\eqref{frs}.
To be more precise we define $U_0 = U$ and $U_{k+1}$ to be the blowup of $U_k$ with center $Z_k^U := U_k\cap Z_k$.
Moreover, for each step we choose as a nonrational center for this blowup the subscheme $S_k^U := U_k \cap S_k$.
We denote by $\cD_U$ the obtained categorical DG-resolution of $\cdcf(U)$.

\begin{prop}
The association $U \mapsto \cD_U$ defines a presheaf of DG-categories on $Y$.
\end{prop}
\begin{proof}
Evident.
\end{proof}

\begin{remark}
One can check that $\cD_U$ is quasiequivalent to the Drinfeld quotient of $\cD$
by the subcategory generated by all objects in the components $\cdcf(Z_k)$ of $\cD$
which are cohomologically supported on $Z_k \setminus Z_k^U$ as well as by all objects in $\cdcf(Y_m)$
which are cohomologically supported on $Y_m \setminus U_m$.
\end{remark}

Note that for $U$ sufficiently small (contained in the complement of the union of images of all $Z_k$ in $Y$)
we have $\cD_U \cong \cdcf(U)$. Thus the constructed resolution is ``birational''.


\begin{prop}
Let $g:Y \to Y$ be an automorphism which preserves the resolution~\eqref{frs}.
Then there is a quasiautoequivalence $g^*:\cD \to \cD$ which preserves 
the semiorthogonal decompositions~\eqref{sods} of $\bD(\cD)$ and $\bD(\cD)^\comp$
and on each component of these decompositions is compatible with the pullback functor 
induced by the restriction of $g$ onto the corresponding scheme $Z_i$ or $Y_m$.
Moreover, the quasiautoequivalence $g^*$ extends to a quasiautoequivalence of the presheaf $\cD_U$.
\end{prop}
\begin{proof}
Evident.
\end{proof}

\section{Appendix A. More on gluings}

In this Appendix we show that the operation of gluing described in section~\ref{s-glu}
is compatible with tensor products of DG-categories and provide a description of quasifunctors
to and from the gluing. We use freely the notation from section~4.

\subsection{The gluing and tensor products}

First we observe that the opposite DG-category of the gluing is itself obtained by a gluing.
For this note that 
\begin{equation*}
(\cD_2^\op\otimes\cD_1)\dgm = ((\cD_1^\op)^\op\otimes(\cD_2)^\op)\dgm,
\end{equation*}
thus any bimodule $\varphi \in (\cD_2^\op\otimes\cD_1)\dgm$ can be also used 
for the gluing of $\cD_2^\op$ with $\cD_1^\op$.

\begin{lemma}\label{gluop}
One has a DG-equivalence
$
(\cD_1\times_\varphi\cD_2)^\op \cong \cD_2^\op \times_\varphi\cD_1^\op.
$
\end{lemma}
\begin{proof}
The objects and the $\Hom$-complexes are the same by definition.
\end{proof}

Let $\cC$ be a small DG-category. Consider the bimodule 
\begin{equation*}
\bar\varphi := \cC\otimes_\kk \varphi \in (\cC^\op\otimes\cC\otimes\cD_2^\op\otimes\cD_1)\dgm = ((\cC\otimes\cD_2)^\op \otimes (\cC\otimes\cD_1))\dgm.
\end{equation*}
We can use it to form the gluing $(\cC\otimes\cD_1)\times_{\bar\varphi}(\cC\otimes\cD_2)$.

\begin{prop}\label{glutens}
There is an equivalence of categories
\begin{equation*}
\bD(\cC\otimes(\cD_1\times_\varphi\cD_2)) \cong \bD((\cC\otimes\cD_1)\times_{\bar\varphi}(\cC\otimes\cD_2)).
\end{equation*}
\end{prop}
\begin{proof}
Consider a DG-functor $\Delta:\cC\otimes(\cD_1\times_\varphi\cD_2) \to (\cC\otimes\cD_1)\times_{\bar\varphi}(\cC\otimes\cD_2)$ defined by
\begin{equation*}
C \otimes (M_1,M_2,\mu) \mapsto (C \otimes M_1,C \otimes M_2,1_C\otimes\mu).
\end{equation*}
Note that
\begin{multline*}
\Hom_{(\cC\otimes\cD_1)\times_{\bar\varphi}(\cC\otimes\cD_2)}((C \otimes M_1,C \otimes M_2,1_C\otimes\mu),(D \otimes N_1,D \otimes N_2,1_D\otimes\nu)) = \\ =
\Hom_{\cC\otimes\cD_1}(C \otimes M_1,D \otimes N_1) \oplus \Hom_{\cC\otimes\cD_2}(C \otimes M_2,D \otimes N_2) \oplus \bar\varphi(D \otimes N_2,C \otimes M_1) = \\ =
\Hom_\cC(C,D)\otimes \Big( \Hom_{\cD_1}(M_1,N_1) \oplus \Hom_{\cD_2}(M_2,N_2) \oplus \varphi(N_2,M_1) \Big) = \\ =
\Hom_{\cC\otimes(\cD_1\times_\varphi\cD_2)}(C \otimes (M_1,M_2,\mu),D \otimes (N_1,N_2,\nu))
\end{multline*}
which means that the functor $\Delta$ is fully faithful.

By Proposition~\ref{fpb} the DG-functor $\Delta$ extends to a fully faithful triangulated functor
$\LInd_\Delta:\bD(\cC\otimes(\cD_1\times_\varphi\cD_2)) \to \bD((\cC\otimes\cD_1)\times_{\bar\varphi}(\cC\otimes\cD_2))$
which has a right adjoint. Thus it remains to check that the orthogonal in 
$\bD((\cC\otimes\cD_1)\times_{\bar\varphi}(\cC\otimes\cD_2))$ to the image of $\LInd_\Delta$ 
is zero. For this it suffices to check that any representable $((\cC\otimes\cD_1)\times_{\bar\varphi}(\cC\otimes\cD_2))$-module
is contained in the triangulated category generated by the image of $\LInd_\Delta$.

To check this take any object $(C_1 \otimes M_1,C_2 \otimes M_2,\bar\mu) \in (\cC\otimes\cD_1)\times_{\bar\varphi}(\cC\otimes\cD_2)$.
The canonical triangle~\eqref{glutri} then shows that
\begin{equation*}
(C_1 \otimes M_1,C_2 \otimes M_2,\bar\mu) \cong \Cone((C_1 \otimes M_1,0,0)[-1] \xrightarrow{\ \bar\mu\ } (0,C_2 \otimes M_2,0)).
\end{equation*}
Since both $(C_1 \otimes M_1,0,0) = \Delta(C_1\otimes (M_1,0,0))$ and $(0,C_2\otimes M_2,0) = \Delta(C_2\otimes (0,M_2,0))$
are in the image of the functor $\LInd_\Delta$ the result follows.
%
%
%
%
\end{proof}

\begin{remark}
In fact the above argument shows also that the pretriangulated envelope of $\cC\otimes(\cD_1\times_\varphi\cD_2)$
is quasiequivalent to the pretriangulated envelope of $(\cC\otimes\cD_1)\times_{\bar\varphi}(\cC\otimes\cD_2)$.
\end{remark}

\begin{cor}
We have a semiorthogonal decomposition
\begin{equation*}
\bD(\cC\otimes(\cD_1\times_\varphi\cD_2)) = \langle \bD(\cC\otimes\cD_1),\bD(\cC\otimes\cD_2) \rangle
\end{equation*}
with the gluing functor equal to $-\lotimes_{\cD_2}\varphi:\bD(\cC\otimes\cD_2) \to \bD(\cC\otimes\cD_1)$.
%
%
\end{cor}
\begin{proof}
The semiorthogonal decomposition follows from the combination of Proposition~\ref{glutens} and Proposition~\ref{derglu}.
By Proposition~\ref{derglu} the gluing functor is isomorphic to $-\lotimes_{\cC\otimes\cD_2}\bar\varphi$.
But since $\bar\varphi = \cC\otimes_\kk\varphi$, this is isomorphic to $-\lotimes_{\cD_2}\varphi$.
\end{proof}

Combining this Corollary with Proposition~\ref{uniglu2} one easily deduces the following

\begin{cor}
Assume that $\cD$ is a small pretriangulated DG-category with a semiorthogonal decomposition
$[\cD] = \langle \cT_1,\cT_2 \rangle$. Then for any small DG-category $\cC$ the derived category
there are semiorthogonal decompositions
\begin{equation*}
\bD(\cC\otimes\cD) = \langle \bD(\cC\otimes\cD_1), \bD(\cC\otimes\cD_2) \rangle,
\qquad
\Perf(\cC\otimes\cD) = \langle \Perf(\cC\otimes\cD_1), \Perf(\cC\otimes\cD_2) \rangle,
\end{equation*}
where $\cD_1$ and $\cD_2$ are the enhancements of $\cT_1$ and $\cT_2$ induced by $\cD$ and $\Perf$
stands for the homotopy category of perfect DG-modules.

\end{cor}


%
%


\subsection{Quasifunctors to and from the gluing}


One can describe the category of quasifunctors to and from the gluing.

\begin{defi}\label{defqfun}
A DG-bimodule $\varphi \in (\cD_1^\op\otimes\cD_2)\dgm$ is {\sf right quasirepresentable} (also called a {\sf quasifunctor})
if for any $X_1 \in \cD_1$ the right $\cD_2$-module $\varphi(X_1,-)$ is quasiisomorphic to a representable DG-module. 
\end{defi}


\begin{prop}\label{uniglu}
$(i)$ To give a quasifunctor $\alpha:\cC \to  \cD_1\times_\varphi\cD_2$ is equivalent to giving
a quasifunctor $\alpha_1:\cC \to \cD_1$, a quasifunctor $\alpha_2:\cC \to \cD_2$, and
a morphism from $\alpha_1$ to $\varphi\circ\alpha_2$ in the derived category $\bD(\cC^\op\otimes\cD_1)$.

$(ii)$ If $\cC$ is pretriangulated and $\varphi$ is a quasifunctor then to give a quasifunctor $\beta:\cD_1\times_\varphi \cD_2 \to \cC$ is equivalent to giving
a quasifunctor $\beta_1:\cD_1 \to \cC$, a quasifunctor $\beta_2:\cD_2 \to \cC$, and
a morphism from $\beta_2$ to $\beta_1\circ\varphi$ in the derived category $\bD(\cD_2^\op\otimes\cC)$.
\end{prop}
\begin{proof}


$(i)$ 
We apply Proposition~\ref{glutens} with $\cC^\op$ instead of $\cC$. By Lemma~\ref{objsod} to give
an object $\alpha \in \bD(\cC^\op\otimes\cD)$ is equivalent to giving its components 
$\alpha_1 = I_1^*\alpha \in \bD(\cC^\op\otimes\cD_1)$ and $\alpha_2 = I_2^!\alpha \in \bD(\cC^\op\otimes\cD_2)$
and a morphism $\alpha_1 \to \alpha_2\lotimes_{\cD_2}\varphi = \varphi\circ\alpha_2$.
Thus we only have to check that $\alpha$ is a quasifunctor 
if and only if both its components $\alpha_1 \in \bD(\cC^\op\otimes\cD_1)$ and $\alpha_2 \in \bD(\cC^\op\otimes\cD_2)$ are.
Substituting $C$ into distinguished triangle
\begin{equation*}
I_2\alpha_2 \to \alpha \to I_1\alpha_1
\end{equation*}
we obtain a distinguished triangle
\begin{equation*}
(I_2\alpha_2)(C,-) \to \alpha(C,-) \to (I_1\alpha_1)(C,-).
\end{equation*}
Note also that 
\begin{equation*}
(I_k\alpha_k)(C,-) = (\alpha_k\lotimes_{\cD_k}{}_{i_k}\cD)(C,-) = \alpha_k(C,-)\lotimes_{\cD_k}\cD(i_k(-),-) \cong I_k(\alpha_k(C,-)).
\end{equation*}
Thus $\alpha_k(C,-) \in \bD(\cD_k)\subset \bD(\cD)$ are just the components of $\alpha(C,-)$ with respect to the semiorthogonal
decomposition $\bD(\cD) = \langle \bD(\cD_1),\bD(\cD_2) \rangle$.

Now assume that $\alpha$ is a quasifunctor so $\alpha(C,-) \cong \sY^{(M_1,M_2,\mu)}$. 
By the above argument $\alpha_k(C,-) \in \bD(\cD_k)$ are just the components of $\sY^{(M_1,M_2,\mu)}$ in $\bD(\cD_k)$. 
The latter are evidently given by $\sY^{M_k}$. Hence $\alpha_1$ and $\alpha_2$ are quasifunctors.


Vice versa, if both $\alpha_1$ and $\alpha_2$ are quasifunctors and $I_k(\alpha_k(C,-)) \cong I_k(\sY^{M_k})\cong \sY^{i_k(M_k)}$,
we deduce that $\alpha(C,-)$ fits into the triangle
\begin{equation*}
\sY^{i_2(M_2)} \to \alpha(C,-) \to \sY^{i_1(M_1)}
\end{equation*}
in $\bD(\cD) = \langle \bD(\cD_1),\bD(\cD_2) \rangle$. Thus $\alpha(C,-)$ is quasiisomorphic to the cone of a morphism
from $\sY^{M_1}[-1]$ to $\sY^{M_2}$. The space of such morphisms is nothing but the zero cohomology of the complex $\varphi(M_2,M_1)$,
and if we take the morphism corresponding to some closed element $\mu \in \varphi(M_2,M_1)$ of degree zero we will obtain
precisely $\sY^{(M_1,M_2,\mu)}$. Thus we have $\alpha(C,-) \cong \sY^{(M_1,M_2,\mu)}$ for appropriate $\mu$, 
hence $\alpha$ is a quasilinear.

$(ii)$ 
Analogously, we apply Proposition~\ref{glutens} with $\cD^\op = \cD_2^\op\times_\varphi\cD_1^\op$ instead of $\cD$. 
Note that the order of DG-categories $\cD_1$ and $\cD_2$ in the gluing is interchanged. By this reason the embedding
and the projection functors of the associated semiorthogonal decompositions indexed by $1$ (such as $i_1$, $I_1$, $I_1^*$
and so on) are related to the category $\cD_2$ and those indexed by $2$ --- to the category $\cD_1$.

According to Lemma~\ref{objsod} to give an object $\beta \in \bD(\cD^\op\otimes\cC)$ is equivalent to giving its components 
$\beta_2 = I_1^*\beta \in \bD(\cD_2^\op\otimes\cC)$ and $\beta_1 = I_2^!\beta \in \bD(\cD_1^\op\otimes\cC)$
and a morphism $\beta_2 \to \varphi\lotimes_{\cD_1}\beta_1 = \beta_1\circ\varphi$.
Thus we only have to check that $\beta$ is a quasifunctor 
if and only if both its components $\beta_2 \in \bD(\cD_2^\op\otimes\cC)$ and $\beta_1 \in \bD(\cD_1^\op\otimes\cC)$ are.

So, assume that $\beta$ is a quasifunctor. Then 
$\beta_1(M_1,-) = (I_2^!\beta)(M_1,-) = \beta((M_1,0,0),-)$ 
is quasirepresentable, so $\beta_1$ is a quasifunctor. 
The case of $\beta_2$ is slightly more complicated.
We have to check that $\beta_2(M_2,-) = (I_1^*\beta)(M_2,-)$ is quasirepresentable.
Since $\cC$ is pretriangulated, using triangle~\eqref{I1s} we see that it suffices 
to check quasirepresentability of DG-modules $\beta_{i_1}(M_2,-) = \beta((0,M_2,0),-)$ 
and of $(\varphi\lotimes_{\cD_1}\beta_{i_2})(M_2,-) = \varphi(M_2,-)\lotimes_{\cD_1}\beta_{i_2}$.
The first holds since $\beta$ is a quasifunctor.
For the second note that since $\varphi$ is a quasifunctor we know that $\varphi(M_2,-) \cong \sY^{N_1}$ for some $N_1 \in \cD_1$.
Therefore 
\begin{equation*}
\varphi(M_2,-)\lotimes_{\cD_1}\beta_{i_2} \cong \sY^{N_1}\otimes_{\cD_1}\beta_{i_2} = (\beta_{i_2})(N_1,-) = \beta((N_1,0,0),-).
\end{equation*}
This latter DG-module is quasirepresentable since $\beta$ is a quasifunctor. 
Combining all this we conclude that $\beta_2$ is also a quasifunctor.

Vice versa, assume that $\beta_1$ and $\beta_2$ are quasifunctors. Then 
\begin{equation*}
(I_1\beta_2)((M_1,M_2,\mu),-) = 
(\cD_{i_1}\lotimes_{\cD_2}\beta_2)((M_1,M_2,\mu),-).
\end{equation*}
Note that 
\begin{equation*}
(\cD_{i_1})((M_1,M_2,\mu),N_2) =
\Hom_{\cD}((0,N_2,0),(M_1,M_2,\mu)) =
\Hom_{\cD_2}(N_2,M_2) = \sY^{M_2}(N_2).
\end{equation*}
Thus the above tensor product equals $\sY^{M_2}\lotimes_{\cD_2}\beta_2 = \beta_2(M_2,-)$.
Since $\beta_2$ is a quasifunctor, this $\cC$-DG-module is quasirepresentable.

Analogously,
\begin{equation*}
(I_2\beta_1)((M_1,M_2,\mu),-) = 
(\cD_{i_2}\lotimes_{\cD_1}\beta_1)((M_1,M_2,\mu),-)
\end{equation*}
and we have
\begin{equation*}
(\cD_{i_2})((M_1,M_2,\mu),N_1) =
\Hom_{\cD}((N_1,0,0),(M_1,M_2,\mu)) =
\Hom_{\cD_1}(N_1,M_1) \oplus \varphi(M_2,N_1)[-1]
\end{equation*}
with the differential coinciding with that of $\Cone(\sY^{M_1} \xrightarrow{\ \mu\ } \varphi(M_2,-))[-1]$
evaluated on~$N_1$. Therefore the above tensor product is quasiisomorphic to 
\begin{equation*}
\Cone(\sY^{M_1}\lotimes_{\cD_1}\beta_1 \xrightarrow{\ \mu\ } \varphi(M_2,-)\lotimes_{\cD_1}\beta_1)[-1].
\end{equation*}
The first term here is quasiisomorphic to $\beta_1(M_1,-)$, so is quasirepresentable.
On the other hand, since $\varphi$ is a quasifunctor, $\varphi(M_2,-) \cong \sY^{N_1}$ for some $N_1 \in \cD_1$,
hence the second term is quasiisomorphic to $\sY^{N_1}\otimes_{\cD_1}\beta_1 = \beta_1(N_1,-)$,
so it is also quasirepresentable. Since $\cC$ is pretriangulated, the above cone is also quasirepresentable.

Thus we have checked that both $(I_1\beta_2)((M_1,M_2,\mu),-)$ and $(I_2\beta_1)((M_1,M_2,\mu),-)$
are quasirepresentable $\cC$-modules, hence so is $\beta((M_1,M_2,\mu),-)$ as $\cC$ is pretriangulated.
\end{proof}

\section{Appendix B. More on Auslander algebras}

\subsection{Quiver-sheaves}\label{ss-qs}

In this section we give an alternative description of the category of $\csA$-modules
and translate most of the constructions of section~\ref{s-nred} to this language.

\begin{defi}\label{dqc}
A quasicoherent quiver-sheaf on $(S,\fr,n)$  is
\begin{itemize}
\item a collection $M_1,\dots,M_n$ of quasicoherent sheaves on $S$,
\item a collection of morphisms $\alpha:M_i \to M_{i+1}$, 
\item a collection of morphisms $\beta:M_i\otimes_{\cO_S}\fr \to M_{i-1}$,
\end{itemize}
such that 
\begin{enumerate}
\item the diagram 
\begin{equation*}
\vcenter{\xymatrix@C=1cm{
M_n\otimes_{\cO_S}\fr \ar[dr]^\beta \ar[d]_\ba & 
M_{n-1}\otimes_{\cO_S}\fr \ar[l]_-\alpha \ar[dr]^\beta \ar[d]_\ba & 
\dots \ar[l]_-\alpha \ar[dr]^\beta & 
M_2\otimes_{\cO_S}\fr \ar[l]_-\alpha \ar[dr]^\beta \ar[d]_\ba & 
M_1\otimes_{\cO_S}\fr \ar[l]_-\alpha \ar[dr]^\beta \ar[d]_\ba \\
M_n & 
M_{n-1} \ar[l]_-\alpha & 
\dots \ar[l]_-\alpha  & 
M_2 \ar[l]_-\alpha  & 
M_1 \ar[l]_-\alpha & 0 \ar[l]
}}
\end{equation*}
(where $\ba:M_i\otimes_{\cO_S}\fr \to M_i$ is the action of $\fr$ on $M_i$) is commutative;

\item for all $i$ and $k$ there is a map $\beta_k:M_i\otimes_{\cO_S}\fr^k \to M_{i-k}$
such that the diagram
\begin{equation*}
\xymatrix{
M_i\otimes_{\cO_S}\fr^{\otimes k} \ar[r]^-\beta \ar[d] &
M_{i-1}\otimes_{\cO_S}\fr^{\otimes {k-1}} \ar[r]^-\beta &
\dots \ar[r]^-\beta &
M_{i-k+1}\otimes_{\cO_S}\fr \ar[r]^-\beta &
M_{i-k} \ar@{=}[d] \\
M_i \otimes_{\cO_S}\fr^k \ar@{..>}[rrrr]^-{\beta_k} &&&& M_{i-k}
}
\end{equation*}
commutes, where the vertical arrow is the map $M_i\otimes_{\cO_S}\fr^{\otimes k} \to M_i \otimes_{\cO_S}\fr^k$
given by the multiplication in $\fr$.
\end{enumerate}

A {\em morphism of quiver-sheaves} from $\bar M$ to $\bar N$ is a collection of morphisms of quasicoherent sheaves $f_i:M_i \to N_i$
commuting with $\alpha$ and $\beta$.
\end{defi}

\begin{lemma}\label{qsam}
The category of quiver sheaves is equivalent to the category $\Qcoh(\csA)$ of quasicoherent $\csA$-modules.
\end{lemma}
\begin{proof}
Let $\bar M = (M_n,\dots,M_1)$ be a quiver-sheaf. Then we consider a quasicoherent sheaf
$M := M_n \oplus M_{n-1} \oplus \dots \oplus M_1$ on $S$ and define a structure of a right $\csA$-module on it
as follows. For each $i,j$ we consider the map
\begin{equation*}
M_{n+1-i}\otimes\csA_{ij} =
M_{n+1-i}\otimes (\fr^{\max(j-i,0)}/\fr^{n+1-i}) \to M_{n+1-j}
\end{equation*}
which is given by the map $\beta_{j-i}$ if $j \ge i$ and by the map $\alpha^{i-j}$ if $j < i$.
Note that by part~(2) of Definition~\ref{dqc} the ideal $\fr^{n+1-i}$ acts trivially on $M_{n+1-i}$, 
so the above map is well defined.

Vice versa, assume that $M$ is a quasicoherent $\csA$-module. Using idempotents $\epsilon_i$ we define $M_i = M\epsilon_{n+1-i}$.
This gives a decomposition $M = M_n \oplus M_{n-1} \oplus \dots \oplus M_1$. Now we equip
$\bar M = (M_n,M_{n-1},\dots,M_1)$ with a structure of a quiver sheaf by defining the map $\alpha:M_{n-i} \to M_{n+1-i}$
as a map given by $\alpha_i = 1 \in \cO_S/\fr^{n-i} = \csA_{i+1,i}$, and the map $\beta:M_{n+1-i} \otimes \fr \to M_{n-i}$
as a map induced by the action of $\csA_{i,i+1} = \fr/\fr^{n+1-i}$.

It is a straightforward exercise left to the reader to check that these a mutually inverse equivalences.
\end{proof}

It is easy to write down the functors of the semiorthogonal decomposition in terms of quiver sheaves.
An easy verification shows that
%
%
\begin{equation*}\label{sie1}
\renewcommand{\arraystretch}{2}
\begin{array}{rcl}
\sip(M_0) &=&
\big(\xymatrix@1{
M_0\ \ar@<-.5ex>@{..>}[r]_-\beta & 
\ 0\ \ar@<-.5ex>[l]_-\alpha \ar@<-.5ex>@{..>}[r]_-\beta & 
\ \dots\  \ar@<-.5ex>[l]_-\alpha \ar@<-.5ex>@{..>}[r]_-\beta & 
\ 0 \ar@<-.5ex>[l]_-\alpha  
}\big),\\[-10pt]
\sip^*\big(\xymatrix@1{
M_n\ \ar@<-.5ex>@{..>}[r]_-\beta & 
\ M_{n-1}\ \ar@<-.5ex>[l]_-\alpha \ar@<-.5ex>@{..>}[r]_-\beta & 
\ \dots\  \ar@<-.5ex>[l]_-\alpha \ar@<-.5ex>@{..>}[r]_-\beta & 
\ M_1 \ar@<-.5ex>[l]_-\alpha  
}\big)
&=& 
\Coker(\alpha:M_{n-1} \to M_n),\\[-10pt]
\sep\big(\xymatrix@1{
M_{n-1}\ \ar@<-.5ex>@{..>}[r]_-\beta & 
\ \dots\  \ar@<-.5ex>[l]_-\alpha \ar@<-.5ex>@{..>}[r]_-\beta & 
\ M_1 \ar@<-.5ex>[l]_-\alpha  
}\big)
&=&
\big(\xymatrix@1{
M_{n-1}\ \ar@<-.5ex>@{..>}[r]_-{\ba} & 
\ M_{n-1}\ \ar@<-.5ex>[l]_-1 \ar@<-.5ex>@{..>}[r]_-\beta & 
\ \dots\  \ar@<-.5ex>[l]_-\alpha \ar@<-.5ex>@{..>}[r]_-\beta & 
\ M_1 \ar@<-.5ex>[l]_-\alpha  
}\big),\\[-10pt]
%
\sep^!\big(\xymatrix@1{
M_n\ \ar@<-.5ex>@{..>}[r]_-\beta & 
\ M_{n-1}\ \ar@<-.5ex>[l]_-\alpha \ar@<-.5ex>@{..>}[r]_-\beta & 
\ \dots\  \ar@<-.5ex>[l]_-\alpha \ar@<-.5ex>@{..>}[r]_-\beta & 
\ M_1 \ar@<-.5ex>[l]_-\alpha  
}\big)
&=&
\big(\xymatrix@1{
M_{n-1}\ \ar@<-.5ex>@{..>}[r]_-\beta & 
\ \dots\  \ar@<-.5ex>[l]_-\alpha \ar@<-.5ex>@{..>}[r]_-\beta & 
\ M_1 \ar@<-.5ex>[l]_-\alpha  
}\big).
\end{array}
\end{equation*} 

It is also easy to write down the resolution functors
\begin{equation*}\label{auspis1}
\begin{array}{l}
\rho_S^*(M) =
\big(\xymatrix@1@C=16pt{
M\ \ar@<-.5ex>@{..>}[r]_-\beta & 
\ M\otimes_{\cO_S}\fr\ \ar@<-.5ex>[l]_-\alpha \ar@<-.5ex>@{..>}[r]_-\beta & 
\ \dots\  \ar@<-.5ex>[l]_-\alpha \ar@<-.5ex>@{..>}[r]_-\beta & 
\ M\otimes_{\cO_S}\fr^{n-1} \ar@<-.5ex>[l]_-\alpha  
}\big),\\
\rho_{S*}
\big(\xymatrix@1@C=16pt{
M_n\ \ar@<-.5ex>@{..>}[r]_-\beta & 
\ M_{n-1}\ \ar@<-.5ex>[l]_-\alpha \ar@<-.5ex>@{..>}[r]_-\beta & 
\ \dots\  \ar@<-.5ex>[l]_-\alpha \ar@<-.5ex>@{..>}[r]_-\beta & 
\ M_1 \ar@<-.5ex>[l]_-\alpha  
}\big) 
= M_n
\end{array}
\end{equation*} 
where in the first line morphisms $\alpha$ are induced by the injections $\fr^i \to \fr^{i-1}$,
and morphisms $\beta$ are induced by the multiplication $\fr^i\otimes_{\cO_S}\fr \to \fr^{i+1}$.

\subsection{Bounded coherent sheaves on a nonreduced scheme}\label{ss-bcaus}

%

In this section we show that the functor $\rho_{S*}:\bD^b(\coh(\csA)) \to \bD^b(\coh(S))$
is a localization of triangulated categories. First note that the statement is evident on the level of big categories.

\begin{lemma}\label{ausloc}
The category $\bD(S)$ is a localization of the category $\bD(\csA)$.
\end{lemma}
\begin{proof}
Since the functor $L\rho_S^*:\bD(S) \to  \bD(\csA)$ is fully faithful and has a right adjoint
we have a semiorthogonal decomposition
\begin{equation*}
\bD(\csA_S) = \langle \Ker \rho_{S*}, \bD(S) \rangle.
\end{equation*}
In particular, $\bD(S)$ is equivalent to the localization of $\bD(\csA)$ by $\Ker \rho_{S*}$.
\end{proof}

The only problem with extending this result to bounded coherent categories is that the functor $L\rho_S^*$
does not preserve boundedness. By definition~\eqref{auspis} it is given by a tensor product with a nonperfect
$\cO_S$-module, hence $L\rho_S^*(M)$ is unbounded from below. So, for example if we want to show that 
for any $M \in \bD^b(\coh(S))$ there is $\cM \in \bD^b(\coh(\csA))$ such that $M \cong \rho_{S*}(\cM)$
we cannot just take $\cM = L\rho_S^*(M)$. On the other hand we can take $\cM$ to be a suitable truncation 
of $L\rho_S^*(M)$ with respect to the standard t-structure.



%


\begin{lemma}
The functor $\rho_{S*}:\bD^b(\coh(\csA)) \to \bD^b(\coh(S))$ is t-exact with respect to the standard t-structures.
In other words,
\begin{equation}\label{rhotau}
\rho_{S*}\circ \tau^{\le k} = \tau^{\le k}\circ \rho_{S*},
\qquad
\rho_{S*}\circ \tau^{\ge k} = \tau^{\ge k}\circ \rho_{S*},
\end{equation} 
where $\tau^{\le k}$ and $\tau^{\ge k}$ are the truncation functors 
of the standard t-structures on $\bD^b(\coh(\csA))$ and $\bD^b(\coh(S))$.
\end{lemma}
\begin{proof}
Recall that $\rho_{S*}(M) = M\epsilon_1$, the direct summand corresponding to the idempotent $\epsilon_1$ of $\csA$. 
So it is evidently t-exact.
\end{proof}

We start with the following preparatory result.

\begin{prop}\label{auslift}
Any morphism in $\bD^b(\coh(S))$ can be lifted to $\bD^b(\coh(\csA))$. In other words for any 
morphism $f:M \to N$ in $\bD^b(\coh (S))$ there are objects $\cM,\cN \in \bD^b(\coh(\csA))$
and a morphism $\bar f:\cM \to \cN$ such that $\rho_{S*}(\cM) \cong M$, 
$\rho_{S*}(\cN) \cong N$ and the diagram
\begin{equation*}\label{bmbn}
\vcenter{\xymatrix{
\rho_{S*}(\cM) \ar[r]^\cong \ar[d]_{\rho_{S*}(\bar f)} & M \ar[d]^f \\ \rho_{S*}(\cN) \ar[r]^\cong & N
}} 
\end{equation*}
commutes.
%
\end{prop}
\begin{proof}
First let us show that for any $M \in \bD^b(\coh(S))$ there is an object $\cM \in \bD^b(\coh(\csA))$
such that $\rho_{S*}\cM \cong M$. Indeed, by~\eqref{rhotau}
%
\begin{equation*}
\tau^{\ge k} M =
\tau^{\ge k} \rho_{S*}L\rho_S^* M =
\rho_{S*}(\tau^{\ge k}L\rho_S^* M).
\end{equation*}
Now choose $k$ such that $\tau^{\ge k} M \cong M$ (this is possible since $M$ has bounded cohomology)
and note that $\tau^{\ge k}L\rho_S^* M \in \bD^b(\coh(\csA))$.

The same argument applies to morphisms. Again, take $k$ such that $\tau^{\ge k} M = M$ and $\tau^{\ge k} N = N$
and consider the morphism $\tau^{\ge k}L\rho_S^*M \xrightarrow{\tau^{\ge k}L\rho_S^*f} \tau^{\ge k}L\rho_S^*N$.
Since the truncation is functorial and comes with a morphism of functors $\id \to \tau^{\ge k}$,
the result follows.
\end{proof}

\begin{cor}
The category $\bD^b(\coh(S))$ is the quotient of $\bD^b(\coh(\csA))$ by $\Ker \rho_{S*}$.
\end{cor}
\begin{proof}
By definition of the quotient category the functor $\rho_{S*}$ factors through a functor 
\begin{equation}\label{ausequi}
\bD^b(\coh(\csA))/\Ker \rho_{S*} \xrightarrow{\qquad} \bD^b(\coh(S))
\end{equation}
and we have to check that this functor is an equivalence.
Note that it is essentially surjective on objects by Proposition~\ref{auslift}.
Let us check that it is fully faithful. 

First let us check that the functor is surjective on morphisms.
Take arbitrary objects $M = \rho_{S*}\cM$, $N = \rho_{S*}\cN$ and consider a morphism
$f:\rho_{S*}\cM \to \rho_{S*}\cN$. Our goal is to construct a morphism $\cM \to \cN$
which goes to $f$ under $\rho_{S*}$. Note that by Proposition~\ref{auslift} there is a morphism
$\bar f:\cM' \to \cN'$ in $\bD^b(\coh(\csA))$ such that diagram
\begin{equation*}
\vcenter{\xymatrix{
\rho_{S*}(\cM') \ar[r]^\cong \ar[d]_{\rho_{S*}(\bar f)} & M \ar[d]^f \\ \rho_{S*}(\cN') \ar[r]^\cong & N
}} 
\end{equation*}
commutes.
By adjunction it gives a morphism
$L\rho_S^*\rho_{S*}\cM' \to \cM$. On the other hand we have an adjunction morphism $L\rho_S^*\rho_{S*}\cM' \to \cM'$.
Since both $\cM$ and $\cM'$ have finite number of cohomology sheaves, there exists $k$ such that
both maps factor through $\tau^{\ge k}L\rho_S^*\rho_{S*} \cM'$. Moreover, the argument of Proposition~\ref{auslift}
shows that $k$ may be chosen in such a way that after application of $\rho_{S*}$ both maps will be isomorphisms.
Thus their cones are in $\Ker \rho_{S*}$, hence both maps are isomorphisms in the quotient category. Therefore
$\cM'$ and $\cM$ are isomorphic in the quotient category, hence the morphism $\bar f$ gives a morphism
from $\cM$ to $\cN$ which maps into $f$ by $\rho_{S*}$. Thus~\eqref{ausequi} is surjective on morphisms.

Now consider a morphism $\bar f:\cM \to \cN$ and assume that $\rho_{S*}\bar f:\rho_{S*}\cM \to \rho_{S*}\cN$ is zero.
By adjunction the composition $L\rho_S^*\rho_{S*}\cM \xrightarrow{} \cM \xrightarrow{\bar f} \cN$ is zero.
But since both $\cM$ and $\cN$ are bounded, it follows that there is $k$ such that the composition
$\tau^{\ge k}L\rho_S^*\rho_{S*}\cM \xrightarrow{} \cM \xrightarrow{\bar f} \cN$ is zero and 
$\rho_{S*}(\tau^{\ge k}L\rho_S^*\rho_{S*}\cM) \cong \rho_{S*}\cM$. Therefore the cone of the map 
$\tau^{\ge k}L\rho_S^*\rho_{S*}\cM \xrightarrow{} \cM$ is in $\Ker \rho_{S*}$,
hence the first map is an isomorphism in the quotient category, hence $\bar f$ in the quotient category is zero.
Thus~\eqref{ausequi} is injective on morphisms and we are done.
\end{proof}

\subsection{The opposite Auslander algebra}\label{ss-aop}

Besides the usual Auslander algebra one can consider its opposite $\csA^\opp$. It turns out that most
of the properties established for the Auslander algebra hold for the opposite as well.
As an illustration we sketch a construction of the semiorthogonal decomposition
analogous to that of section~\ref{ss-aussod}.

First, we have a functor $\si_\op:\bD(S_0) \to \bD(\csA^\op)$ defined by the same formula~\eqref{sia} as the functor~$\si$,
as well as its (derived) left adjoint
\begin{equation*}
L\si_\op^*:\bD(\csA^\op) \to \bD(S_0),
\qquad
M \mapsto \cO_{S_0}\lotimes_\csA M.
\end{equation*}
We also have an adjoint pair of functors
\begin{equation*}
\begin{array}{ll}
L\se_\op:\bD({\csA'}^\op) \to \bD(\csA^\op),
\qquad&
M \mapsto \csA(1-\epsilon_1)\lotimes_{\csA'}M,\\
\hphantom{L}\se_\op^!:\bD({\csA}^\op) \to \bD({\csA'}^\op),
\qquad&
N \mapsto \Hom_\csA(\csA(1-\epsilon_1),N) = (1-\epsilon_1)N
\end{array}
\end{equation*}
(so the only difference is that the functor $\se_\op$ is not exact and we have to take its derived functor).
This functors still enjoy the analogues of relations~\eqref{iec}, 
and tensoring~\eqref{cicos0} by arbitrary $M \in \bD(\csA^\op)$ we
obtain an analogue of triangle~\eqref{eemii}. Thus there is a semiorthogonal decomposition
\begin{equation}
\bD(\csA^\op) = \langle \si_\op(\bD({S_0})),L\se_\op(\bD({\csA'}^\op)) \rangle.
\end{equation} 
By iteration, we deduce an $n$-component semiorthogonal decomposition
\begin{equation*}
\bD(\csA^\op) = \langle \bD({S_0}),\dots,\bD({S_0}) \rangle.
\end{equation*}

\subsection{Homological dimension}\label{ss-gldim}

Recall that the {\sf projective dimension $\pdim_A(M)$ of a right module $M$} over an algebra $A$ is the minimal length
of its projective resolution. Furthermore, the {\sf global dimension $\gldim A$ of $A$} is the supremum of the projective
dimensions of right $A$-modules.

\begin{prop}\label{gda}
Assume that $S$ is a regular local scheme of dimension $d$. Then
\begin{equation*}
\gldim\csA_{S,\fr,n} \le nd + 2(n-1),
\qquad
\gldim\csA_{S,\fr,n}^\op \le nd + 2(n-1).
\end{equation*}
\end{prop}
\begin{proof}
We are going to prove the claim by induction on~$n$. The case $n = 1$ is evident --- in this case $\csA = \cO_{S_0}$
is a regular local ring, so its global dimension equals its Krull dimension which agrees with the statement of the Proposition.

Now assume the claim is known for $n-1$.
Note that by~\eqref{cidef} and~\eqref{cadef} as a left $\csA$-module we have $\csI \cong \csA\epsilon_2 \oplus \csA(1-\epsilon_1)$,
in particular it is projective, so
\begin{equation}\label{pdilop}
\pdim_{\csA^\op}\csI = 0.
\end{equation} 
On the other hand, $(1-\epsilon_1)\csA$ is projective as a right $\csA$-module, hence the functor $\se(-) = -\lotimes_{\csA'}(1-\epsilon_1)\csA$ 
does not increase the projective dimension of a right $\csA'$-module. This means that 
\begin{equation}\label{pdil}
\pdim_{\csA}\csI \le \gldim \csA'.
\end{equation} 
Using~\eqref{cicos0}, \eqref{pdilop}, and~\eqref{pdil} we conclude that
\begin{equation}\label{pdo}
\pdim_{\csA^\op}\cO_{S_0} \le 1,
\qquad
\pdim_{\csA}\cO_{S_0} \le \gldim\csA' + 1.
\end{equation}

%
%
%

Now let $M$ be a right $\csA$-module and consider the canonical triangle
\begin{equation*}
\se(\se^!(M)) \to M \to \si(L\si^*(M)).
\end{equation*}
Note also that we have 
\begin{equation*}
\si(L\si^*(M)) = M \lotimes_{\csA} \cO_{S_0}.
\end{equation*}
It follows from the first part of~\eqref{pdo} that $L\si^*(M)$ has cohomology only in degrees $0$ and $-1$, hence 
it has a free over $\cO_{S_0}$ resolution of length $d + 1$. On the other hand, by the second part of~\eqref{pdo}
$\si(\cO_{S_0})$ has a projective resolution of length $\gldim\csA' + 1$. Thus $\si(L\si^*(M))$ has a projective
resolution of length $\gldim\csA' + d + 2$. On the other hand, since $\se^!$ is exact, $\se^!(M)$ is an $\csA'$-module,
hence $\pdim_{\csA'}(\se^!(M)) \le \gldim\csA'$ and since the functor $\se$ does not increase the projective dimension, 
the module $\se(\se^!(M))$ has a projective resolution of length $\gldim\csA'$. Combining these two observations 
we conclude that $\gldim \csA  \le \gldim\csA' + d + 2$.
So, the induction hypothesis $\gldim\csA' \le (n-1)d + 2(n-2)$ implies that we have $\gldim\csA \le nd + 2(n-1)$.

Analogously, let $M$ be a left $\csA$-module and consider the canonical triangle
\begin{equation*}
L\se_\op(\se_\op^!(M)) \to M \to \si_\op(L\si_\op^*(M)).
\end{equation*}
Now by the second part of~\eqref{pdo} we know that $L\si_\op^*(M)$ has cohomology only in degrees from $0$ to $-(\gldim{\csA'}^\op + 1)$, hence 
it has a free over $\cO_{S_0}$ resolution of length $\gldim{\csA'}^\op + d + 1$. On the other hand, by the first part of~\eqref{pdo}
$\si_\op(\cO_{S_0})$ has a projective resolution of length $1$. Thus $\si_\op(L\si_\op^*(M))$ has a projective
resolution of length $\gldim{\csA'}^\op + d + 2$. On the other hand, since $\se_\op^!$ is exact, $\se_\op^!(M)$ is an $\csA'$-module,
hence $\pdim_{{\csA'}^\op}(\se^!(M)) \le \gldim{\csA'}^\op$. Since $\csA(1-\epsilon_1)$ is a projective left $\csA$-module,
the functor $L\se_\op$ takes projective resolutions to projective resolutions of the same length, so $\se_\op(\se_\op^!(M))$ has a projective resolution
of length $\gldim{\csA'}^\op$. Combining these two observations we again conclude that $\gldim \csA^\op  \le \gldim{\csA'}^\op + d + 2$.
Using the induction hypothesis in the same way as above, we deduce the claim.
%
%
%
%
%
%
%
\end{proof}

\subsection{Generalized Auslander spaces}\label{ss-ausgen}

In this section we indicate how the algebra $\csA$ can be generalized.

Let $S$ be a scheme. Choose a collection of ideals $\fa_{ij} \subset \cO_S$, where $1 \le i \le n$ and $2 \le j \le n+1$,
and the following incidence conditions are satisfied
\begin{equation}\label{a-inc}
\begin{array}{ccccccccc}
\fa_{12} & \supset & \fa_{13} & \supset & \dots & \supset & \fa_{1n} & \supset & \fa_{1,n+1} \\
&& \cap &&&& \cap && \cap \\
&& \fa_{23} & \supset & \dots & \supset & \fa_{2n} & \supset & \fa_{2,n+1} \\
&&&&&& \cap && \cap \\
&&&& \ddots && \vdots && \vdots \\
&&&&&& \cap && \cap \\
&&&&&& \fa_{n-1,n} & \supset & \fa_{n-1,n+1} \\
&&&&&&&& \cap \\
&&&&&&&& \fa_{n,n+1} \\
\end{array}
\end{equation}
Assume also 
\begin{equation}\label{a-comp}
\fa_{ij}\cdot\fa_{jk} \subset \fa_{ik}
\end{equation}
for all $1 \le i < j < k \le n+1$.
We define a {\sf generalized Auslander algebra} as $\csA = \oplus_{i,j = 1}^n \csA_{ij}$ with
\begin{equation*}
\csA_{ij} = 
\begin{cases}
\fa_{ij}/\fa_{i,n+1}, & \text{if $i < j$},\\
\cO_S/\fa_{i,n+1}, & \text{if $i \ge j$}
\end{cases}
\end{equation*}
In other words,
\begin{equation}\label{gadef}
\csA :=
\begin{pmatrix}
\cO_S & \fa_{12}/\fa_{1,n+1} & \fa_{13}/\fa_{1,n+1} & \dots & \fa_{1n}/\fa_{1,n+1} \\
\cO_S/\fa_{2,n+1} & \cO_S/\fa_{2,n+1} & \fa_{23}/\fa_{2,n+1} & \dots & \fa_{2,n}/\fa_{2,n+1} \\
\cO_S/\fa_{3,n+1} & \cO_S/\fa_{3,n+1} & \cO_S/\fa_{3,n+1} & \dots & \fa_{3,n}/\fa_{3,n+1} \\
\vdots & \vdots & \vdots & \ddots & \vdots \\
\cO_S/\fa_{n,n+1} & \cO_S/\fa_{n,n+1} & \cO_S/\fa_{n,n+1} & \dots & \cO_S/\fa_{n,n+1} 
\end{pmatrix}.
\end{equation}
The multiplication is induced by the natural embedding 
\begin{equation}
\csA \subset \End(\cO_S/\fa_{1,n+1} \oplus \cO_S/\fa_{2,n+1} \oplus \cO_S/\fa_{3,n+1} \oplus \dots \oplus \cO_S/\fa_{n,n+1}).
\end{equation} 

\begin{example}
Let $\fr \subset \cO_S$ be an ideal such that $\fr^n = 0$. Then
\begin{equation*}
\fa_{ij} = \fr^{n+1-i}:\fr^{n+1-j}
\end{equation*}
gives $\csA = \End(\cO_S \oplus \cO_S/\fr \oplus \cO_S/\fr^2 \oplus \dots \oplus \cO_S/\fr^{n-1})$,
the {\sf original Auslander algebra}.
\end{example}

\begin{example}
Let $\fr \subset \cO_S$ be an ideal such that $\fr^n = 0$. Then
\begin{equation*}
\fa_{ij} = \fr^{j-i}
\end{equation*}
gives the algebra~\eqref{cadef}, the {\sf special Auslander algebra}.
\end{example}

It turns out that the generalized Auslander algebras enjoy the same properties as the special Auslander algebras
considered in section~\ref{s-nred}.
For example the argument of Proposition~\ref{dqcsod} and Corollary~\ref{sodqqc} proves the following

\begin{prop}\label{sodga}
There are semiorthogonal decompositions
\begin{equation*}
\begin{array}{lll}
\bD(\csA) &=& \langle \bD(S_1),\bD(S_2),\dots,\bD(S_n) \rangle,\\
\bD^b(\coh(\csA)) &=& 
\langle \bD^b(\coh(S_1)), \bD^b(\coh(S_2)), \dots, \bD^b(\coh(S_n)) \rangle.
\end{array}
\end{equation*}
where $S_i$ is the subscheme of $S$ corresponding to the ideal $\fa_{i,i+1}$, so that $\cO_{S_i} = \cO_S/\fa_{i,i+1}$.
\end{prop}

\begin{remark}
Note that the components of the decomposition do not depend on the ideals $\fa_{i,j}$ with $j - i > 1$.
However one can check that these ideals govern the gluing functors of the decomposition.
\end{remark}

For the special Auslander algebra we have $\fa_{i,i+1} = \fr$ for all $i$, hence $S_1 = \dots = S_n = S_0$,
so all the components of the semiorthogonal decomposition coincide. On the other hand, for the original Auslander algebra
we have $\fa_{i,i+1} = \fr^{n+1-i}:\fr^{n-i}$, so the components may be different.

\begin{example}
Consider $S = \Spec \kk[x,y]/(x^2,xy)$. Take $\fr = (x)$ and $n = 2$. Then we have $\fa_{1,2} = \Ann(\fr) = (x,y)$, $\fa_{1,3} = 0$, $\fa_{2,3} = \fr$,
so $S_1 = \Spec\kk$, $S_2 = \Spec \kk[y]$. In particular, the original Auslander algebra gives a semiorthogonal decomposition
with components $\bD(\kk)$ and $\bD(\kk[y])$, while the special Auslander algebra gives a decomposition with both components
equal to $\bD(\kk[y])$.
\end{example}

\begin{prop}
In the assumptions of Proposition~\ref{sodga},
assume that the subschemes $S_1,\dots,S_n \subset S$ are local regular schemes of dimensions $d_1,\dots,d_n$. Then
\begin{equation*}
\gldim\csA \le d_1 + \dots + d_n + 2(n-1),
\qquad
\gldim\csA^\op \le d_1 + \dots + d_n + 2(n-1).
\end{equation*}
\end{prop}

\begin{prop}
In the assumptions of Proposition~\ref{sodga},
if all the subschemes $S_1,\dots,S_n$ of $S$ are smooth then 
\begin{equation*}
\bD(\csA)^\comp = \bD^\perf(\csA) = \bD^b(\coh(\csA)).
\end{equation*}
\end{prop}

Again, one can prove these Propositions by the same arguments as the corresponding statements in section~\ref{s-nred}
and section~\ref{ss-gldim}.

Moreover, one can define a DG-category $\cdcf(\csA)$ in the same way as it is done in section~\ref{s-nred},
construct a DG-functor $\rho_S:\cdcf(S) \to \cdcf(\csA)$. Again, the same arguments prove

\begin{theo}
In the assumptions of Proposition~\ref{sodga},
if all the subschemes $S_1,\dots,S_n$ of $S$ are smooth then 
the functor $\rho_S$ is a DG-resolution of singularities and the derived
functor $L\rho_S^*:\bD(S) \to \bD(\csA)$ is a categorical resolution.
\end{theo}

%
%

\subsection{Generalized morphisms of $\csA$-spaces}\label{ss-genmaps}

In this section we describe the notion of generalized morphisms of $\csA$-spaces.
Let $(S,\fa)$ be a generalized Auslander space as defined in section~\ref{ss-ausgen}.

\begin{defi}
A {\sf generalized morphism of $\csA$-spaces} from $(T,\fa^T,n_T)$ to $(S,\fa^S,n_S)$ consists of
a continuous map $f:T \to S$ of the underlying topological spaces and a sheaf $\cP$ of $f^{-1}\csA_S$-$\csA_T$-bimodules
which is projective over $\csA_T$.
%
\end{defi}

If $\bff = (f,\cP)$ is a morphism of $\csA$-spaces we can define the functors
\begin{equation*}
\bff^*:\Qcoh(\csA_S) \to \Qcoh(\csA_T),
\qquad
M \mapsto f^{-1}M\otimes_{f^{-1}\csA_S}\cP
\end{equation*}
\begin{equation*}
\bff_*:\Qcoh(\csA_T) \to \Qcoh(\csA_S),
\qquad
N \mapsto f_*\cHom_{\csA_T}(\cP,N).
\end{equation*}

One can easily prove the following

\begin{lemma}
The functor $\bff^*$ is right exact and the functor $\bff_*$ is left exact.
Moreover, the functor $\bff^*$ is the left adjoint of the functor $\bff_*$.
\end{lemma}

Assume that $f:T \to  S$ is a morphism of schemes compatible with ideals $\fa^T_{ij}$ and $\fa^S_{ij}$, i.e.
\begin{equation*}
f^{-1}(\fa^S_{ij}) \subset \fa^T_{ij}
\end{equation*}
for all $1 \le i < j \le n_S + 1$ and $n_S\le n_T$. Then the morphism $f$ induces a morphism $f^{-1}\csA_S \to \csA_T$
compatible with the addition and multiplication laws
and taking the unit of $\csA_S$ to the idempotent
$\epsilon_1 + \dots + \epsilon_{n_S}$ of $\csA_T$.
Thus 
\begin{equation*}\label{cpf}
\cP_f := (\epsilon_1 + \dots + \epsilon_{n_S})\csA_T
\end{equation*} 
is a $f^\bull(\csA_S)-\csA_T$-bimodule which is projective over $\csA_T$. So, the pair $\bff := (f,\cP_f)$ 
defines a generalized morphism of $\csA$-spaces, which we will refer to as {\sf the morphism of $\csA$-spaces
induced by the morphism $f$ of schemes}. It is easy to see that the induced pullback and pushforward functors
on categories of $\csA$-modules coincide with those defined in section~\ref{s-nred}.

\begin{example}\label{egm}
Let $S_1 \subset S$ be the subscheme corresponding to the ideal $\fa_{1,2} \subset \cO_S$.
Then we have a generalized morphism from $(S_1,0,1)$ to $(S,\fa^S,n_S)$ given by the embedding of
schemes $i:S_1\to S$ and the bimodule $\cP_1 = \cO_{S_1}$. Then the corresponding pullback 
and pushforward functors coincide with the functors $\si^*$ and $\si$ investigated in section~\ref{s-nred}.


On the other hand, let $S' \subset S$ be the subscheme corresponding to the ideal $\fa_{2,n+1} \subset \cO_S$
and let $\csA'$ be the generalized Auslander algebra corresponding to the system of ideals
in $\cO_{S'} = \cO_S / \fa_{2,n+1}$ obtained from~\eqref{a-inc} by removing the first line.
Note that the schemes $S$ and $S'$ have the same underlying topological spaces. Let $e:S' \to S$
be the identity morphism of those (note that it does not extend to a morphism of schemes).
Further, $\cP' = (1-\epsilon_1)\csA$ is a $\csA'$-$\csA$-bimodule which is projective over $\csA$,
hence $\be = (e,\cP')$ is a morphism of $\csA$-spaces. 
Then the pullback functor $\be^*$ is isomorphic to the functor $\se$
while the pushforward functor $\be_*$ is isomorphic to $\se^!$ defined 
in~\eqref{sea} and~\eqref{sesa} respectively.

It is also easy to see that the induction and the restriction functors 
considered in Lemma~\ref{ahf} are also the pullback and the pushforward
for appropriate generalized morphisms of $\csA$-spaces.
\end{example}

\end{document}